\documentclass[12pt,american,british,english]{amsart}
\usepackage[T1]{fontenc}
\usepackage[latin9]{inputenc}
\usepackage{geometry}
\usepackage{babel}
\usepackage{array}
\usepackage{float}
\usepackage{multirow}
\usepackage{amsthm}
\usepackage{amsbsy}
\usepackage{amstext}
\usepackage{amssymb}
\usepackage{graphicx}
\usepackage{esint}
\usepackage[unicode=true,pdfusetitle,
 bookmarks=true,bookmarksnumbered=false,bookmarksopen=false,
 breaklinks=false,pdfborder={0 0 1},backref=false,colorlinks=false]
 {hyperref}
\usepackage{breakurl}

\makeatletter

\newcommand{\noun}[1]{\textsc{#1}}
\providecommand{\tabularnewline}{\\}

\numberwithin{equation}{section}
\numberwithin{figure}{section}
\theoremstyle{plain}
\newtheorem{thm}{\protect\theoremname}
  \theoremstyle{definition}
  \newtheorem{defn}[thm]{\protect\definitionname}
  \theoremstyle{plain}
  \newtheorem{prop}[thm]{\protect\propositionname}
  \theoremstyle{plain}
  \newtheorem{lem}[thm]{\protect\lemmaname}

\newcommand{\psfragBody}{
\psfrag{D}{\hspace{-0.15mm}\raisebox{0.07mm}{\small{D}}}
\psfrag{N}{\hspace{-0.15mm}\raisebox{0.07mm}{\small{N}}}
}
\newcommand{\psfragBodyNew}{
\psfrag{D}{\hspace{-0.35mm}\raisebox{0.07mm}{\small{D}}}
\psfrag{N}{\hspace{-0.35mm}\raisebox{0.07mm}{\small{N}}}
}
\newcommand{\psfragBodySmall}{
\psfrag{D}{\hspace{-0.15mm}\raisebox{0.1mm}{\tiny{D}}}
\psfrag{N}{\hspace{-0.15mm}\raisebox{0.1mm}{\tiny{N}}}
}
\newcommand{\psfragAppendix}{
\psfrag{D}{\hspace{-0.15mm}\raisebox{0.1mm}{\tiny{D}}}
\psfrag{N}{\hspace{-0.15mm}\raisebox{0.1mm}{\tiny{N}}}
}

\newlength{\myArraycolsep}
\usepackage{chngcntr}
\counterwithout{figure}{section}

\usepackage{subfig}
\def\footnote{\xdef\@thefnmark{}\@footnotetext}

\usepackage{psfrag}

\newtheorem*{RegularityTheorem}{Regularity Theorem}
\newtheorem*{ReflectionPrinciple}{Reflection Principle}
\newtheorem*{UniqueContinuationTheorem}{Unique Continuation Theorem}
\newtheorem*{TransplantationTheorem}{Transplantation Theorem}
\newtheorem*{SubstitutionTheorem}{Substitution Theorem}
\newtheorem*{TraceTheorem}{Trace Theorem}
\theoremstyle{definition}
\newtheorem*{Assumption}{Assumption A}

\@ifundefined{showcaptionsetup}{}{%
 \PassOptionsToPackage{caption=false}{subfig}}
\usepackage{subfig}
\makeatother

  \addto\captionsamerican{\renewcommand{\definitionname}{Definition}}
  \addto\captionsamerican{\renewcommand{\lemmaname}{Lemma}}
  \addto\captionsamerican{\renewcommand{\propositionname}{Proposition}}
  \addto\captionsamerican{\renewcommand{\theoremname}{Theorem}}
  \addto\captionsbritish{\renewcommand{\definitionname}{Definition}}
  \addto\captionsbritish{\renewcommand{\lemmaname}{Lemma}}
  \addto\captionsbritish{\renewcommand{\propositionname}{Proposition}}
  \addto\captionsbritish{\renewcommand{\theoremname}{Theorem}}
  \addto\captionsenglish{\renewcommand{\definitionname}{Definition}}
  \addto\captionsenglish{\renewcommand{\lemmaname}{Lemma}}
  \addto\captionsenglish{\renewcommand{\propositionname}{Proposition}}
  \addto\captionsenglish{\renewcommand{\theoremname}{Theorem}}
  \providecommand{\definitionname}{Definition}
  \providecommand{\lemmaname}{Lemma}
  \providecommand{\propositionname}{Proposition}
\providecommand{\theoremname}{Theorem}

\begin{document}

\title[On Inaudible Properties of Broken Drums]{On inaudible properties of broken drums -- Isospectrality with mixed
Dirichlet-Neumann boundary conditions}

\author{Peter Herbrich}
\begin{abstract}
We study isospectrality for manifolds with mixed Dirichlet-Neumann
boundary conditions and express the well-known transplantation\emph{
}method in graph- and representation-theoretic terms. This leads to
a characterization of transplantability in terms of monomial relations
in finite groups and allows for the generating of new transplantable
pairs from given ones as well as a computer-aided search for isospectral
pairs. In particular, we show that the Dirichlet spectrum of a manifold
does not determine whether it is connected and that an orbifold can
be Dirichlet isospectral to a manifold.
\end{abstract}
\maketitle
\begin{center}

\par\end{center}

\global\long\def\Colour{\boldsymbol{c}}
 \global\long\def\NumberOfColours{C}
 \global\long\def\NumberOfVertices{V}
 \global\long\def\Trace{\mathrm{Tr}}
 \global\long\def\Graph{\Gamma}
 \global\long\def\AdjacencyMatrix{A}
 \global\long\def\OnSecond#1{\widehat{#1}}
 \global\long\def\Dimension{d}
 \global\long\def\HomFKtoG{\Phi}
 \global\long\def\BasisVector{\boldsymbol{e}}
 \global\long\def\IsoGandGhat{\mathcal{I}}
 \global\long\def\word{w}

\global\long\def\InducedRep#1#2#3{\mathrm{Ind}_{#1}^{#2}(#3)}
 \global\long\def\Square{\mathbb{S}}
 \global\long\def\ProductOfAdjMat{p}
 \global\long\def\ElementInG{g}
 \global\long\def\GeneratorsOfG{\gamma}
 \global\long\def\CayleyGraph#1#2{\Graph(#1,#2)}
 \global\long\def\RepresentationOfG{\Phi}
 \global\long\def\Cross{\otimes}
 \global\long\def\Crossing#1#2{#1\Cross#2}
 \global\long\def\GraphOfSubstituent{\Gamma_{S}}
 \global\long\def\NumberOfVerticesOfSubstituent{V_{S}}
 \global\long\def\NumberOfEdgeColoursOfSubstituent{\NumberOfColours_{S}}
 \global\long\def\GroupOfSubstituent{G_{S}}

\global\long\def\ColourOfSubstituent{\boldsymbol{\chi}}
 \global\long\def\AdjacencyMatrixOfSubstituent#1{\AdjacencyMatrix_{S}^{#1}}

\global\long\def\LoopIndices{\boldsymbol{I}}
 \global\long\def\Substitution{\triangleright}
 \global\long\def\SubstitutedGraph#1#2{#2\Substitution#1}

\section{Introduction\label{sec:Introduction}}

In the style of Kac's famous question ``Can one hear the shape of
a drum?''~\cite{Kac1966}, we consider\emph{ }broken drums with
partially attached drumheads. Let $M$ be a compact flat manifold
with piecewise smooth boundary $\partial M$, that features disjoint
smooth open subsets $\partial_{D}M$ and $\partial_{N}M$, representing
the attached and unattached parts of the drumhead, such that
\begin{equation}
\partial M=\overline{\partial_{D}M\cup\partial_{N}M}.\label{eq:BoundaryDecomposition}
\end{equation}
The audible frequencies of the broken drum $M$ are determined by
the Zaremba eigenvalue problem 
\begin{equation}
\begin{array}{cccl}
\Delta\varphi & = & \lambda\varphi & \mbox{on }M^{\circ}\\
\varphi & = & 0 & \mbox{on }\partial_{D}M\quad\text{(Dirichlet)}\\
\frac{\partial\varphi}{\partial n} & = & 0 & \mbox{on }\partial_{N}M\quad\text{(Neumann)},
\end{array}\label{eq:MixedEigenvalueProblem}
\end{equation}
where $M^{\circ}$ denotes the interior of $M$ and $\frac{\partial\varphi}{\partial n}$
denotes the normal derivative of~$\varphi$.  Recall that integration
with respect to the Riemannian measure gives rise to an inner product
on $C^{\infty}(M)=C^{\infty}(M,\mathbb{R})$ with Hilbert space completion
$L^{2}(M)=L^{2}(M,\mathbb{R})$. The following assumption will allow
us to compare solutions of~(\ref{eq:MixedEigenvalueProblem}) in
a combinatorial manner.

\begin{Assumption}The eigenvalues $\lambda$ for which there exists
$\varphi\in C(M)\cap C^{\infty}(M^{\circ}\cup\partial_{D}M\cup\partial_{N}M)$
satisfying~(\ref{eq:MixedEigenvalueProblem}) are given as an unbounded
sequence
\begin{equation}
0\leq\lambda_{0}\leq\lambda_{1}\leq\lambda_{2}\leq\ldots,\label{eq:Spectrum}
\end{equation}
where each eigenspace is finite dimensional, and $L^{2}(M)$ is the
orthogonal direct sum of all eigenspaces.\end{Assumption}

Once mixed boundary conditions are imposed on $M$, we refer to (\ref{eq:Spectrum})
as the spectrum\emph{ }of\emph{ }$M$. We call two manifolds isospectral
if their spectra coincide and reserve the terms Dirichlet\emph{ }isospectral\emph{
}and Neumann\emph{ }isospectral\emph{ }for pure\emph{ }boundary\emph{
}conditions, that is, if $\partial_{N}M=\varnothing$ or $\partial_{D}M=\varnothing$,
respectively. An inspection of the heat kernel shows that isospectral
manifolds have the same dimension, volume and difference of Neumann
boundary volume and Dirichlet boundary volume~\cite{BransonGilkey1990}.

\begin{figure}
\noindent \begin{centering}
\hfill{}\subfloat[First known pair\label{fig:First-planar-pair}]{\centering{}%
\begin{minipage}[b]{53mm}%
\noindent \begin{center}
\psfrag{1}{} \psfrag{2}{} \psfrag{3}{} \psfrag{4}{} \psfrag{5}{} \psfrag{6}{} \psfrag{7}{}\includegraphics[scale=0.31]{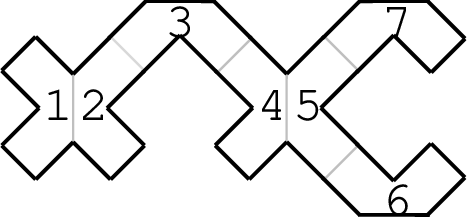}\hspace{3mm}\includegraphics[scale=0.31]{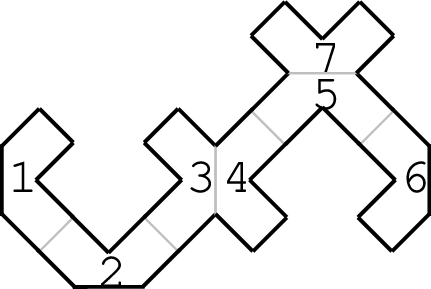}
\par\end{center}%
\end{minipage}}\hfill{}\subfloat[Gordon-Webb-Wolpert drums\label{fig:IsospectralDrumsWithTriangles}]{\noindent \begin{centering}
\begin{minipage}[b]{62mm}%
\noindent \begin{center}
\psfrag{1}{} \psfrag{2}{} \psfrag{3}{} \psfrag{4}{} \psfrag{5}{} \psfrag{6}{} \psfrag{7}{}\includegraphics[scale=0.31]{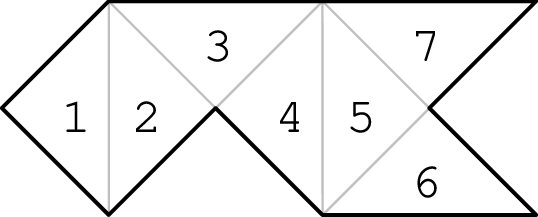}\hspace{3mm}\includegraphics[scale=0.31]{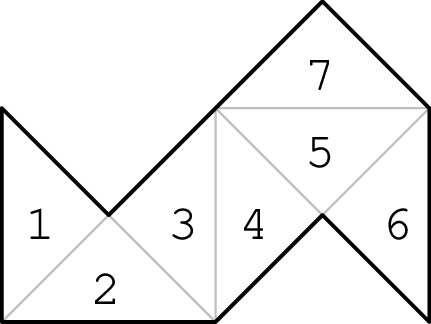}
\par\end{center}%
\end{minipage}
\par\end{centering}

}\hfill{}
\par\end{centering}

\caption{Dirichlet isospectral planar domains~\cite{GordonWebbWolpert1992}.}
\end{figure}

Motivated by number-theoretic ideas, Sunada~\cite{Sunada1985} developed
a celebrated method to construct isospectral manifolds, which has
since been widely extended~\cite{DeTurckGordon1989,Pesce1996,Sutton2002,BandParzanchevskiBen-Shach2009,ParzanchevskiBand2010}.
Gordon~et~al.~\cite{GordonWebbWolpert1992} used B\'erard's extension
to the orbifold setting \cite{B'erard1992} to answer Kac's question
by constructing the pair in Figure~\ref{fig:First-planar-pair}.

Studying the combinatorial aspects of the Sunada method, Buser \cite{Buser1986}
developed the transplantation method, which produces isospectral manifolds
without explicit reference to the underlying group structure. The
method can be applied to pairs of manifolds that are composed of identical
building blocks. Roughly speaking, each solution to~(\ref{eq:MixedEigenvalueProblem})
on one of the manifolds is cut into its restrictions to blocks, and
these restrictions are superposed linearly on the blocks of the other
manifold in such a way that the resulting function also solves~(\ref{eq:MixedEigenvalueProblem}).
All known pairs of Dirichlet isospectral planar domains are obtained
in this way~\cite{BuserConwayDoyleSemmler1994}. Since transplantability
only depends on the combinatorial decomposition of manifolds into
blocks, the so-called Gordon-Webb-Wolpert drums in Figure~\ref{fig:IsospectralDrumsWithTriangles}
are equally isospectral, which has been confirmed experimentally using
microwave cavities~\cite{SridharKudrolli1994}. 

For the case of pure boundary conditions, Okada and Shudo~\cite{OkadaShudo2001}
derived sufficient conditions for transplantability in terms of associated
edge-colored graphs in order to perform a computer-aided search for
isospectral pairs. In Section~\ref{sec:TheTransplantationMethod},
we employ regularity and continuation theorems for elliptic operators
to characterize the transplantation method in terms of edge-colored
graphs with signed loops, that encode mixed boundary conditions.

The work of Levitin~et~al. on transplantable planar domains with
mixed boundary conditions~\cite{LevitinParnovskiPolterovich2006,JakobsonLevitinNadirashviliPolterovich2006}
motivated the generalization of the Sunada method by Band~et~al.
\cite{BandParzanchevskiBen-Shach2009,ParzanchevskiBand2010}. In Section~\ref{sec:IndRepAndTrans},
we show that every transplantable pair arises from this generalization.
In particular, we show that each pair of transplantable graphs gives
rise to a finite group $G$ and two tuples $(R_{i})_{i}$ and $(\widehat{R}_{j})_{j}$
of real one-dimensional representations of subgroups $(H_{i})_{i}$
and $(\widehat{H}_{j})_{j}$ of $G$ such that the sums of their inductions
to $G$ are equivalent,
\[
\bigoplus_{i}\textnormal{Ind}_{H_{i}}^{G}(R_{i})\simeq\bigoplus_{j}\textnormal{Ind}_{\widehat{H}_{j}}^{G}(\widehat{R}_{j}).
\]
Moreover, the graphs can be recovered from this representation-theoretic
data. This viewpoint ties the transplantation method to the study
of so-called monomial relations in finite groups, which have been
investigated intensively because of their applications to $L$-functions~\cite{Langlands1970,Deligne1973}.
In the case of pure Neumann boundary conditions, the representations
$(R_{i})_{i}$ and $(\widehat{R}_{j})_{j}$ are trivial, leading to
so-called Brauer relations, which have been classified for all finite
groups~\cite{BartelDokchitser2011}. In particular, Sunada's original
method uses Gassmann triples which correspond to the shortest non-trivial
Brauer relations.

In Section~\ref{sec:GeneratingTools}, we introduce methods which
allow to generate new transplantable pairs from given ones. In Section~\ref{sec:Examples_of_Trans_Pairs},
we comment on their implications such as the existence of infinitely
many transplantable pairs and the existence of arbitrarily long transplantable
tuples of certain types. Thereafter, we classify all connected pairs
with $2$ edge colors, which yields another proof of the extension
of~\cite[Theorem 4.2]{LevitinParnovskiPolterovich2006} given in~\cite{BandParzanchevskiBen-Shach2009}.
Moreover, we explain the algorithm that was used to search for new
transplantable pairs systematically. Amongst others, we found $10$
pairs of Gordon-Webb-Wolpert drums with mixed boundary conditions,
one of whose isospectrality had been conjectured~\cite{DriscollGottlieb2003}.

In Section~\ref{sec:InaudibleProps}, we discuss inaudible properties
of broken drums and present pairs which are the first non-trivial
known ones of their kind:
\begin{itemize}
\item Section~\ref{sub:SelfDualAndAlmostSelfDualPairs}: A connected domain
with one Dirichlet and one Neumann boundary component whose spectrum
is invariant under swapping Dirichlet with Neumann boundary conditions.
\item Section~\ref{sub:Orientability}: A pair of isospectral connected
flat manifolds with mixed boundary conditions one of which is orientable
while the other is not. The manifolds also have different numbers
of Dirichlet boundary components.
\item Section~\ref{sub:NbrOfComp}: A simply-connected flat manifold that
is Dirichlet isospectral to a disconnected one. In particular, this
is the first known pair that is Dirichlet but not Neumann isospectral.
\item Section~\ref{sub:BrokennessAndIsotropyOrder}: A pair of Dirichlet
isospectral connected flat manifolds with the same heat content.
\item Section~\ref{sub:BrokennessAndIsotropyOrder}: A simply-connected
flat manifold with mixed boundary conditions that is isospectral to
a connected one with pure Dirichlet boundary conditions. In particular,
the number of Neumann boundary components is not spectrally determined,
and orbifolds can be Dirichlet isospectral to manifolds.
\end{itemize}

\section{The transplantation method\label{sec:TheTransplantationMethod}}

The transplantation method applies to manifolds that are obtained
by successive reflections and gluings of building blocks defined as
follows.
\begin{defn}
\label{def:BuildingBlock}A building\emph{ }block\emph{ }is a compact
Riemannian manifold $B$ with piecewise smooth boundary $\partial B$,
which contains disjoint open smooth subsets $\partial_{R}^{1}B,\partial_{R}^{2}B,\ldots,\partial_{R}^{\NumberOfColours}B$,
called reflecting\emph{ }faces, each of which has a neighborhood in
$B$ that is isometric to an open subset of closed Euclidean upper
half space.
\end{defn}
Note that two copies $B_{1}$ and $B_{2}$ of a building block $B$
can be glued together along their reflecting faces $\partial_{R}^{\Colour}B_{1}$
and $\partial_{R}^{\Colour}B_{2}$ by identifying their closures $\overline{\partial_{R}^{\Colour}B_{1}}\subset\partial B_{1}$
and $\overline{\partial_{R}^{\Colour}B_{2}}\subset\partial B_{2}$.
We only consider cases in which the resulting topological space $B_{1}\cup B_{2}/\sim$
is a Riemannian manifold. By assumption, each point of $\partial_{R}^{\Colour}B$
lies in the image of a local isometry $\psi$ with Euclidean domain
$(-l,l)^{\Dimension-1}\times(-l,0]$ for some $l>0$. If $\psi_{1}$
and $\psi_{2}$ denote the corresponding local isometries of $B_{1}$
and $B_{2}$, then
\begin{equation}
\psi_{\cup}(x_{1},\ldots,x_{\Dimension-1},x_{\Dimension})=\begin{cases}
\psi_{1}(x_{1},\ldots,x_{\Dimension-1},x_{\Dimension}) & \mbox{for }x_{\Dimension}\leq0\\
\psi_{2}(x_{1},\ldots,x_{\Dimension-1},-x_{\Dimension}) & \mbox{for }x_{\Dimension}>0
\end{cases}\label{eq:Chart_near_inner_segment}
\end{equation}
shall be a local isometry of $B_{1}\cup B_{2}/\sim$ with domain $(-l,l)^{\Dimension}$.
For the sake of simplicity, we use Euclidean space locally; however,
the theory can be easily extended, for instance, to other constant
curvature spaces or singular manifolds such as quantum graphs.

\begin{figure}
\noindent \begin{centering}
\hfill{}\subfloat[Building block\label{fig:BuildingBlockWithMixedBC}]{\noindent \centering{}%
\begin{minipage}[b]{50mm}%
\noindent \begin{center}
\psfrag{a}{\hspace{-4mm}\footnotesize{$\partial_{R}^{2}B$}} \psfrag{b}{\hspace{-4mm}\footnotesize{$\partial_{R}^{3}B$}} \psfrag{c}{\footnotesize{$\partial_{R}^{1}B$}} \psfrag{d}{\hspace{-4mm}\footnotesize{$\partial_{D}B$}} \psfrag{e}{\footnotesize{$\partial_{D}B$}} \psfrag{n}{\footnotesize{$\partial_{N}B$}}\raisebox{1.28mm}{\includegraphics[scale=0.48]{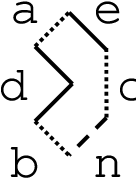}}
\par\end{center}%
\end{minipage}}\hfill{}\subfloat[Figure~\ref{fig:First-planar-pair} with new boundary conditions\label{fig:DomainsWithMixedBC}]{\noindent \begin{centering}
\begin{minipage}[b]{90mm}%
\noindent \begin{center}
\includegraphics[scale=0.48]{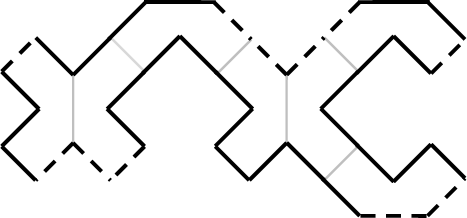}\hspace{3mm}\includegraphics[scale=0.48]{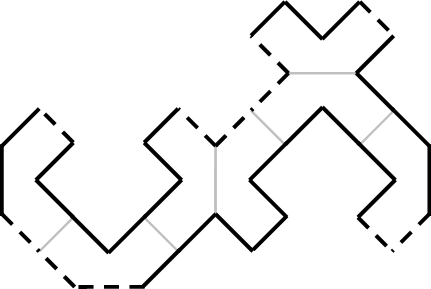}
\par\end{center}%
\end{minipage}
\par\end{centering}

}\hfill{}
\par\end{centering}

\caption{Tiled manifolds with mixed boundary conditions that are represented
by solid (Dirichlet) and dashed (Neumann) lines. Dotted lines are
reflecting faces.\label{fig:TiledDomainsWithMixedBC}}
\end{figure}

\begin{defn}
\label{def:Tiled_Domain}A tiled\emph{ }manifold is a compact Riemannian
manifold $M=(B_{1}\cup B_{2}\cup\ldots\cup B_{\NumberOfVertices})/\sim$
with piecewise smooth boundary $\partial M$, such that\end{defn}
\begin{enumerate}
\item $B_{1},B_{2},\ldots,B_{\NumberOfVertices}$ are copies of some building
block $B$ which has reflecting\emph{ }faces $(\partial_{R}^{\Colour}B)_{\Colour=1}^{\NumberOfColours}$
and disjoint open smooth subsets $\partial_{D}B$ and $\partial_{N}B$
of $\partial B\backslash\bigcup_{\Colour=1}^{\NumberOfColours}\overline{\partial_{R}^{\Colour}B}$
such that
\begin{equation}
\partial B=\overline{\left(\bigcup_{\Colour=1}^{\NumberOfColours}\partial_{R}^{\Colour}B\right)\cup\partial_{D}B\cup\partial_{N}B}.\label{eq:BoundaryDecompositionOfBuildingBlock}
\end{equation}

\begin{enumerate}
\item $B_{1},B_{2},\ldots,B_{\NumberOfVertices}$ are glued along the disjoint
pairs of reflecting faces 
\[
((\partial_{R}^{\Colour_{k}}B_{i_{k}},\partial_{R}^{\Colour_{k}}B_{j_{k}}))_{k},
\]

\item we have 
\[
M^{\circ}=\biggl(B_{1}^{\circ}\cup B_{2}^{\circ}\cup\ldots\cup B_{\NumberOfVertices}^{\circ}\cup\bigcup_{k}\partial_{R}^{\Colour_{k}}B_{i_{k}}\biggr)/\sim,
\]

\item $(\partial_{R}^{\Colour_{k}}B_{i_{k}}/\sim){}_{k}$ are called inner
faces in contrast to outer faces which remain unglued and belong to
$\partial M$,
\item Assumption A is satisfied for all boundary conditions obtained as
follows. On each outer face, we impose either Dirichlet or Neumann
boundary conditions to obtain the Zaremba problem
\[
\begin{array}{cccl}
\Delta\varphi & = & \lambda\varphi & \mbox{on }M^{\circ}\\
\varphi & = & 0 & \mbox{on }\partial_{D}M\\
\frac{\partial\varphi}{\partial n} & = & 0 & \mbox{on }\partial_{N}M
\end{array}
\]
with desired solutions $\varphi\in C(M)\cap C^{\infty}(M^{\circ}\cup\partial_{D}M\cup\partial_{N}M)$,
where $\partial_{D}M$ denotes the union of $\partial_{D}B_{1},\partial_{D}B_{2},\ldots,\partial_{D}B_{\NumberOfVertices}$
and of those outer faces that carry Dirichlet boundary conditions,
$\partial_{N}M$ is defined analogously.
\end{enumerate}
\end{enumerate}
Note that Definition~\ref{def:Tiled_Domain} avoids interior singularities.
Figure~\ref{fig:TiledDomainsWithMixedBC} exemplifies the construction.
\begin{defn}
Let $M$ and $\OnSecond M$ be tiled manifolds each of which is composed
of $\NumberOfVertices$ copies of the building block $B$. Then, each
invertible real $\NumberOfVertices\times\NumberOfVertices$ matrix
$T$ gives rise to a linear isomorphism $T\colon L^{2}(M)\longrightarrow L^{2}(\OnSecond M)$,
called transplantation, such that the transform $\OnSecond{\varphi}=T(\varphi)$
can be written as
\[
\OnSecond{\varphi}_{i}=\sum_{j=1}^{\NumberOfVertices}T_{ij}\varphi_{j}\qquad\textrm{almost everywhere},
\]
where $(\varphi_{j})_{j}$ and $(\OnSecond{\varphi}_{i})_{i}$ denote
the restrictions of $\varphi$ and $\OnSecond{\varphi}$ to the building
blocks of $M$ and $\OnSecond M$, respectively.\label{def:Transplantation_between_tiled_domains}
\end{defn}
The inverse of a transplantation matrix $T$ gives rise to the inverse\emph{
}transplantation\emph{ }$T^{-1}$. If $\varphi$ solves a Zaremba
problem on $M$ with eigenvalue $\lambda$, then any transplant $\OnSecond{\varphi}=T(\varphi)$
on $\OnSecond M$ will solve the Helmholtz equation $\Delta\OnSecond{\varphi}=\lambda\OnSecond{\varphi}$
almost everywhere, but will, in general, neither be smooth nor satisfy
the desired boundary conditions on $\OnSecond M$.
\begin{defn}
Two tiled manifolds with predefined boundary conditions are called
transplantable if there exists a transplantation $T$ such that $T$
and $T^{-1}$ convert solutions of the Zaremba problem on one manifold
into such on the other manifold. In this case, $T$ is said to be
intertwining.
\end{defn}
Note that if $T\colon L^{2}(M)\longrightarrow L^{2}(\OnSecond M)$
is intertwining, then for any solution $\varphi$ of the Zaremba problem
on $M$ with eigenvalue $\lambda$, we have $(\OnSecond{\Delta}\circ T)(\varphi)=\lambda T(\varphi)=T(\lambda\varphi)=(T\circ\Delta)(\varphi)$.
Since $T$ and $T^{-1}$ map eigenspaces into eigenspaces, the spectra
of $M$ and $\OnSecond M$ coincide.
\begin{prop}
\label{prop:TransDomsAreIso}Transplantable manifolds are isospectral.
\end{prop}
\begin{table}
\noindent \begin{centering}
\begin{tabular}{ll}
Tiled manifold & Loop-signed graph\tabularnewline
\hline 
Building blocks & Vertices\tabularnewline
Reflecting faces & Edges\tabularnewline
- Glued / Unglued & - Links / Loops\tabularnewline
- Indices & - Colors\tabularnewline
- Boundary conditions & - Loop signs\tabularnewline
\end{tabular}
\par\end{centering}

\caption{Graph representation of tiled manifolds.\label{tab:CorrespondenceDomainsGraphs}}
\end{table}

In the style of~\cite{OkadaShudo2001}, we encode tiled manifolds
by edge-colored graphs each of whose loops carries one of the signs
$D$ and $N$ indicating the boundary conditions. The correspondence
is summarized in Table~\ref{tab:CorrespondenceDomainsGraphs}, and
an example is shown in Figure~\ref{fig:GraphRepresentation}, where
we used different types of lines (\emph{straight}, \emph{wavy}, \emph{zig-zag})
instead of colors.
\begin{defn}
A loop\emph{-}signed\emph{ }graph $\Graph$ is a finite graph with
vertices $1,2,\ldots,\NumberOfVertices$ together with edge colors
$1,2,\ldots,\NumberOfColours$ such that at each vertex there is exactly
one incident link or loop of each color and each loop carries one
of the signs $D$ or $N$.
\end{defn}

\begin{defn}
The loopless\emph{ }version of a loop-signed graph is the edge-colored
graph obtained by removing all loops. A loop-signed graph is called
treelike if its loopless version is a tree.
\end{defn}
\begin{figure}
\noindent \begin{centering}
\hfill{}\subfloat[Building block\label{fig:BuildingBlockWithGraph}]{\noindent \begin{centering}
\begin{minipage}[b]{42mm}%
\noindent \begin{center}
\raisebox{5.2mm}{\includegraphics[scale=0.42]{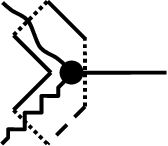}}
\par\end{center}%
\end{minipage}
\par\end{centering}

}\hfill{}\subfloat[Loop-signed graphs\label{fig:GraphRepresentation}]{\centering{}%
\begin{minipage}[b][24.5mm]{85mm}%
\noindent \begin{center}
\psfragBody\includegraphics[scale=0.42]{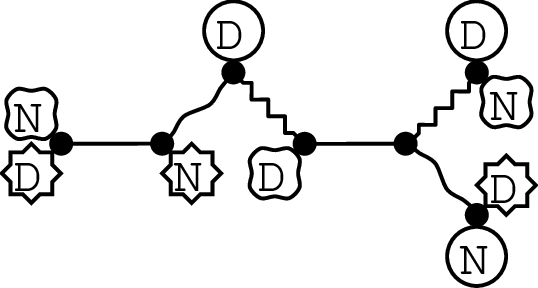}\hspace{5mm}\includegraphics[scale=0.42]{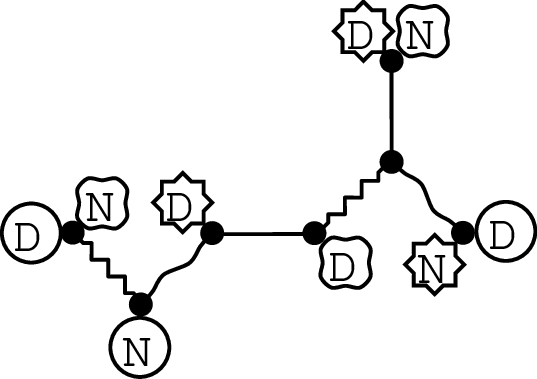}
\par\end{center}%
\end{minipage}}\hfill{}
\par\end{centering}

\caption{Graph representation of the manifolds in Figure~\ref{fig:TiledDomainsWithMixedBC}.}
\end{figure}

\begin{table}
\noindent \centering{}\setlength{\arraycolsep}{2pt}%
\begin{tabular}{ccc}
\emph{$\AdjacencyMatrix^{straight}=\OnSecond{\AdjacencyMatrix}^{zig-zag}$} & \emph{$\AdjacencyMatrix^{zig-zag}=\OnSecond{\AdjacencyMatrix}^{straight}$} & $T=T^{T}=T^{-1}$\tabularnewline
\noalign{\vskip1dd}
{\small{$\left(\begin{array}{ccccccc}
0 & 1 & 0 & 0 & 0 & 0 & 0\\
1 & 0 & 0 & 0 & 0 & 0 & 0\\
0 & 0 & -1 & 0 & 0 & 0 & 0\\
0 & 0 & 0 & 0 & 1 & 0 & 0\\
0 & 0 & 0 & 1 & 0 & 0 & 0\\
0 & 0 & 0 & 0 & 0 & 1 & 0\\
0 & 0 & 0 & 0 & 0 & 0 & -1
\end{array}\right)$}} & {\small{$\left(\begin{array}{ccccccc}
-1 & 0 & 0 & 0 & 0 & 0 & 0\\
0 & 1 & 0 & 0 & 0 & 0 & 0\\
0 & 0 & 0 & 1 & 0 & 0 & 0\\
0 & 0 & 1 & 0 & 0 & 0 & 0\\
0 & 0 & 0 & 0 & 0 & 0 & 1\\
0 & 0 & 0 & 0 & 0 & -1 & 0\\
0 & 0 & 0 & 0 & 1 & 0 & 0
\end{array}\right)$}} & \setlength{\arraycolsep}{1pt}$\frac{1}{2}\negthinspace${\small{$\left(\begin{array}{ccccccc}
-1 & 1 & 1 & 0 & 0 & 0 & 1\\
1 & 1 & 0 & 1 & 1 & 0 & 0\\
1 & 0 & 1 & -1 & 0 & 1 & 0\\
0 & 1 & -1 & 0 & -1 & 1 & 0\\
0 & 1 & 0 & -1 & 0 & -1 & -1\\
0 & 0 & 1 & 1 & -1 & 0 & -1\\
1 & 0 & 0 & 0 & -1 & -1 & 1
\end{array}\right)$}}\tabularnewline
\noalign{\vskip1dd}
\end{tabular}\setlength{\arraycolsep}{\myArraycolsep}\caption{Some adjacency matrices of the graphs in Figure~\ref{fig:GraphRepresentation}
and a transplantation matrix satisfying~(\ref{eq:TransplantationCondition}).\label{tab:AdjacencyMatricesGWWType7} }
\end{table}

Each loop-signed graph is determined by its adjacency\emph{ }matrices
defined as follows.
\begin{defn}
\label{def:AdjacencyMatrices}The adjacency\emph{ }matrices $(\AdjacencyMatrix^{\Colour})_{\Colour=1}^{\NumberOfColours}$
of a loop-signed graph with vertices $1,2,\ldots,\NumberOfVertices$
and edge colors $1,2,\ldots,\NumberOfColours$ are the $\NumberOfVertices\times\NumberOfVertices$
matrices with off-diagonal entries 
\[
\AdjacencyMatrix_{ij}^{\mathbf{c}}=\begin{cases}
1 & \mbox{if vertices }i\mbox{ and }j\mbox{ are joined by a }\Colour\mbox{-colored link}\\
0 & \mbox{otherwise},
\end{cases}
\]
encoding connectivity, and diagonal entries
\begin{equation}
\AdjacencyMatrix_{ii}^{\mathbf{c}}=\begin{cases}
-1 & \text{if vertex }i\text{ has a }\Colour\mbox{-colored loop with sign }D\\
1 & \text{if vertex }i\text{ has a }\Colour\mbox{-colored loop with sign }N\\
0 & \mbox{otherwise},
\end{cases}\label{eq:Sign_convention}
\end{equation}
encoding boundary conditions.
\end{defn}
Note that adjacency matrices are symmetric signed permutation matrices
with non-negative off-diagonal entries. In particular, each row and
each column of an adjacency matrix $\AdjacencyMatrix^{\Colour}$ contains
exactly one non-vanishing entry taking values in $\{-1,1\}$ and 
\[
\AdjacencyMatrix^{\Colour}=(\AdjacencyMatrix^{\Colour})^{T}=(\AdjacencyMatrix^{\Colour})^{-1}.
\]
Table~\ref{tab:AdjacencyMatricesGWWType7} lists some of the adjacency
matrices of the loop-signed graphs in Figure~\ref{fig:GraphRepresentation}
with vertices numbered from left to right. In general, a vertex renumbering
gives rise to new adjacency matrices of the form
\begin{equation}
P\AdjacencyMatrix^{\Colour}P^{-1}\qquad\text{for some permutation matrix }P.\label{eq:Renumbering}
\end{equation}
With regard to the following central transplantability criterion,
this corresponds to the special case of isometric tiled manifolds. 

\begin{TransplantationTheorem}\label{thm: TransplantationTheorem}
Let $M$ and $\OnSecond M$ be tiled manifolds with Zaremba problems
that are described by adjacency matrices $(\AdjacencyMatrix^{\Colour})_{\Colour=1}^{\NumberOfColours}$
and $(\OnSecond{\AdjacencyMatrix}^{\Colour})_{\Colour=1}^{\NumberOfColours}$,
respectively. Then, a transplantation between $M$ and $\OnSecond M$
with transplantation matrix $T$ is intertwining if and only if
\begin{equation}
\OnSecond{\AdjacencyMatrix}^{\Colour}=T\AdjacencyMatrix^{\Colour}T^{-1}\qquad\text{for all edge colors }\Colour.\label{eq:TransplantationCondition}
\end{equation}
In particular, $M$ and $\OnSecond M$ are transplantable precisely
if there exists an invertible real matrix $T$ satisfying (\ref{eq:TransplantationCondition}).

\end{TransplantationTheorem}

Note that~(\ref{eq:Renumbering}) implies that transplantability
is independent of the chosen numberings of the building blocks of
$M$ and $\OnSecond M$. The condition~(\ref{eq:TransplantationCondition})
first appeared in~\cite{OkadaShudo2001}, where the case of treelike
graphs with uniform loop signs is considered. As an example, the matrices
in Table~\ref{tab:AdjacencyMatricesGWWType7} satisfy~(\ref{eq:TransplantationCondition})
and the domains in Figure~\ref{fig:DomainsWithMixedBC} are isospectral.
The Transplantation Theorem motivates the following definition.
\begin{defn}
Two loop-signed graphs given by adjacency matrices $(\AdjacencyMatrix^{\Colour})_{\Colour=1}^{\NumberOfColours}$
and $(\OnSecond{\AdjacencyMatrix}^{\Colour})_{\Colour=1}^{\NumberOfColours}$
are called transplantable if there exists a transplantation matrix
$T$ satisfying $\OnSecond{\AdjacencyMatrix}^{\Colour}=T\AdjacencyMatrix^{\Colour}T^{-1}$
for $\Colour=1,2,\ldots,\NumberOfColours$. The graphs are called
isomorphic if $T$ can be chosen as a permutation matrix.
\end{defn}
In order to prove the Transplantation Theorem, we start with a regularity
theorem, which will imply that solutions to Zaremba problems can be
extended across reflecting faces.

\begin{RegularityTheorem}\label{thm:RegularityTheorem} Let $M$
be a tiled manifold consisting of building blocks $B_{1},B_{2},\ldots,B_{\NumberOfVertices}$.
If $\varphi\in C^{1}(M^{\circ})\cap C^{\infty}(\bigcup_{i=1}^{\NumberOfVertices}B_{i}^{\circ})$
and $\Delta\varphi=\lambda\varphi$ on $\bigcup_{i=1}^{\NumberOfVertices}B_{i}^{\circ}$
for some $\lambda\geq0$, then $\varphi\in C^{\infty}(M^{\circ})$
and $\Delta\varphi=\lambda\varphi$ on $M^{\circ}$.\end{RegularityTheorem}
\begin{proof}
According to Definition~\ref{def:Tiled_Domain}, we only have to
consider inner faces. Using local isometries of the form (\ref{eq:Chart_near_inner_segment}),
the claim reduces to the statement that if $\varphi\in C^{1}((-l,l)^{\Dimension})\cap C^{\infty}((-l,l)^{\Dimension}\backslash\{x_{d}=0\})$
for some $l>0$ and $\Delta\varphi=\lambda\varphi$ on $(-l,l)^{\Dimension}\backslash\{x_{d}=0\}$,
then $\varphi\in C^{\infty}((-l,l)^{\Dimension})$. This follows from
the elliptic regularity theorem~\cite[Chapter 6, Theorem 3]{Evans2010}
once we verified that $\varphi$ is a weak solution of $(\Delta-\lambda)\varphi=0$
on $(-l,l)^{\Dimension}$. In other words, if $\psi\in C_{0}^{\infty}(-l,l)^{\Dimension}$
and
\[
I(\alpha,\beta)=\int_{\alpha}^{\beta}\int_{-l}^{l}\ldots\int_{-l}^{l}\biggl(\sum_{i=1}^{\Dimension}\frac{\partial\varphi}{\partial x_{i}}\frac{\partial\psi}{\partial x_{i}}-\lambda\varphi\psi\biggr)dx_{1}\ldots dx_{\Dimension-1}dx_{\Dimension},
\]
we must show that $I(-l,l)=0$. If we apply Fubini's theorem and integrate
by parts for either $\alpha=-l$ and \textbf{$\beta<0$} or $\alpha>0$
and $\beta=l$, we obtain
\begin{eqnarray*}
I(\alpha,\beta) & = & \int_{-l}^{l}\ldots\int_{-l}^{l}\biggl(\frac{\partial\varphi}{\partial x_{d}}\,\psi\biggr)(x_{1},\ldots,x_{\Dimension-1},\beta)dx_{1}\ldots dx_{\Dimension-1}\\
 & - & \int_{-l}^{l}\ldots\int_{-l}^{l}\biggl(\frac{\partial\varphi}{\partial x_{d}}\,\psi\biggr)(x_{1},\ldots,x_{\Dimension-1},\alpha)dx_{1}\ldots dx_{\Dimension-1},
\end{eqnarray*}
where we used that $\psi\in C_{0}^{\infty}(-l,l)^{\Dimension}$ and
$(\Delta-\lambda)\varphi=0$ on the open set $(-l,l)^{\Dimension}\backslash\{x_{d}=0\}$.
As $\varphi\in C^{1}((-l,l)^{\Dimension})$ and $\psi\in C_{0}^{\infty}(-l,l)^{\Dimension}$,
we obtain $I(-l,0)=-I(0,l)$, which implies $I(-l,l)=0$ as desired.
\end{proof}
\begin{ReflectionPrinciple}\label{thm:ReflectionPrinciple}Let $B$
be a building block and let $\partial_{R}B$ be one of its reflecting
faces. If $\varphi\in C^{\infty}(B^{\circ}\cup\partial_{R}B)$ and
$\Delta\varphi=\lambda\varphi$ on $B^{\circ}$ for some $\lambda\geq0$
as well as either
\[
\frac{\partial\varphi}{\partial n}|_{\partial_{R}B}\equiv0\qquad\text{or}\qquad\varphi|_{\partial_{R}B}\equiv0,
\]
then $\varphi$ can be smoothly extended across $\partial_{R}B$ by
itself (Neumann case) or by $-\varphi$ (Dirichlet case), respectively.
That is, if we reflect $B$ in $\partial_{R}B$ and take $\pm\varphi$
on its pasted copy, then the resulting function is smooth on the interior
of the resulting tiled manifold consisting of two copies of $B$.\end{ReflectionPrinciple}
\begin{proof}
In view of the Regularity Theorem, it suffices to show that the resulting
function is continuously differentiable near $\partial_{R}B$. If
we argue as before, the claim reduces to the statement that if $\varphi\in C^{\infty}((-l,l)^{\Dimension-1}\times(-l,0])$
for some $l>0$ and $\Delta\varphi=\lambda\varphi$ on $(-l,l)^{\Dimension-1}\times(-l,0)$
as well as either 
\[
\frac{\partial\varphi}{\partial x_{\Dimension}}|_{x_{\Dimension}=0}\equiv0\qquad\text{or }\qquad\varphi|_{x_{\Dimension}=0}\equiv0,
\]
then $\varphi$ can be smoothly extended to $(-l,l)^{\Dimension}$
by defining 
\[
\varphi(x_{1},\ldots,x_{\Dimension-1},x_{\Dimension})=\pm\varphi(x_{1},\ldots,x_{\Dimension-1},-x_{\Dimension})\qquad\textnormal{for }x_{\Dimension}>0.
\]
In the Neumann case, we immediately obtain that $\varphi\in C((-l,l)^{\Dimension})$
as well as $\frac{\partial\varphi}{\partial x_{i}}\in C((-l,l)^{\Dimension})$
for $i=1,2,\ldots,\Dimension-1$. Since $\frac{\partial\varphi}{\partial x_{d}}$
vanishes on the hyperplane $\{x_{\Dimension}=0\}$, we also have $\frac{\partial\varphi}{\partial x_{d}}\in C((-l,l)^{\Dimension})$
as desired. In the Dirichlet case, $\varphi$ and therefore also $\frac{\partial\varphi}{\partial x_{i}}$
vanish on $\{x_{\Dimension}=0\}$ for $i=1,2,\ldots,\Dimension-1$.
The continuity of $\frac{\partial\varphi}{\partial x_{\Dimension}}$
follows from
\[
\frac{\partial\varphi}{\partial x_{\Dimension}}(x_{1},\ldots,x_{\Dimension-1},x_{\Dimension})=\frac{\partial\varphi}{\partial x_{\Dimension}}(x_{1},\ldots,x_{\Dimension-1},-x_{\Dimension})\qquad\textnormal{for }x_{\Dimension}>0.
\]

\end{proof}
The following uniqueness theorem follows from a classical result by
Aronszajn~\cite{Aronszajn1957}. In the case of building blocks with
analytic metrics, it is a corollary of elliptic regularity theory.

\begin{UniqueContinuationTheorem}\label{thm:UniqueContinuationTheorem}
Let $B$ be a connected building block and let $\partial_{R}B$ be
one of its reflecting faces. If $\varphi\in C^{\infty}(B^{\circ}\cup\partial_{R}B)$
and $\Delta\varphi=\lambda\varphi$ on $B^{\circ}$ for some $\lambda\geq0$,
then any extension of $\varphi$ across $\partial_{R}B$ to a smooth
eigenfunction of $\Delta$ is unique. That is, if we reflect $B$
in $\partial_{R}B$, then there is at most one function on its pasted
copy that extends $\varphi$ to a smooth eigenfunction of $\Delta$
on the interior of the resulting tiled manifold consisting of two
copies of $B$.\end{UniqueContinuationTheorem}

The preceding existence and uniqueness theorems are summarized in
the following proposition, which justifies the sign convention (\ref{eq:Sign_convention})
and will imply the Transplantation Theorem.
\begin{prop}
Let $M$ be a tiled manifold with connected building blocks $B_{1},B_{2},\ldots,B_{\NumberOfVertices}$.
If $\varphi$ is a solution of the Zaremba problem on $M$ given by
the adjacency matrices $(\AdjacencyMatrix^{\Colour})_{\Colour=1}^{\NumberOfColours}$,
then for each $i=1,2,\ldots,\NumberOfVertices$, the restriction $\varphi_{i}$
of $\varphi$ to $B_{i}$ has a unique extension to a smooth eigenfunction
of $\Delta$ across each non-empty reflecting face $\partial_{R}^{\Colour}B_{i}$,
and this extension is the non-vanishing summand of $\sum_{j}\AdjacencyMatrix_{ij}^{\Colour}\varphi_{j}$.\label{prop:Natural_Sign_Conventions}\end{prop}
\begin{proof}
[Proof of the Transplantation Theorem] We may assume that the underlying
building block is connected, otherwise we consider its components
separately. We first show that a transplantation is intertwining if
its transplantation matrix $T$ satisfies~(\ref{eq:TransplantationCondition}).
In other words, we show that if $\varphi$ solves the Zaremba problem
represented by $(\AdjacencyMatrix^{\Colour})_{\Colour=1}^{\NumberOfColours}$,
then $\OnSecond{\varphi}=T(\varphi)$ given by 
\begin{equation}
\OnSecond{\varphi}_{i}=\sum_{k}T_{ik}\varphi_{k}\label{eq:ithComponentByTransplantation}
\end{equation}
solves the Zaremba problem represented by $(\OnSecond{\AdjacencyMatrix}^{\Colour})_{\Colour=1}^{\NumberOfColours}$,
where $(\varphi_{k})_{k}$ and $(\OnSecond{\varphi}_{i})_{i}$ denote
the restrictions of $\varphi$ and $\OnSecond{\varphi}$ to the blocks
of $M$ and $\OnSecond M$, respectively. We repeatedly regard $\varphi$
and $\OnSecond{\varphi}$ as column vectors of their restrictions
$(\varphi_{k})_{k}$ and $(\OnSecond{\varphi}_{i})_{i}$, in particular,
$\OnSecond{\varphi}=T\varphi$ as vectors.

In order to prove that $\OnSecond{\varphi}\in C(\OnSecond M)\cap C^{\infty}(\OnSecond M^{\circ}\cup\partial_{D}\OnSecond M\cup\partial_{N}\OnSecond M)$,
it suffices to show that whenever two blocks of $\OnSecond M$ numbered
$i$ and $j$ share an inner face corresponding to some edge color
$\Colour$, that is, if $\OnSecond{\AdjacencyMatrix}_{ij}^{\Colour}=1$,
then $\OnSecond{\varphi}_{i}$ and $\OnSecond{\varphi}_{j}$ are smoothly
connected. Assumption~(\ref{eq:TransplantationCondition}) yields
\[
\OnSecond{\AdjacencyMatrix}^{\Colour}\OnSecond{\varphi}=\OnSecond{\AdjacencyMatrix}^{\Colour}T\varphi=T\AdjacencyMatrix^{\Colour}\varphi.
\]
Since $\OnSecond{\AdjacencyMatrix}_{ij}^{\Colour}$ is the only non-vanishing
entry of the $i$th row of $\OnSecond{\AdjacencyMatrix}^{\Colour}$,
we have
\begin{equation}
\OnSecond{\varphi}_{j}=\OnSecond{\AdjacencyMatrix}_{ij}^{\Colour}\OnSecond{\varphi}_{j}=(\OnSecond{\AdjacencyMatrix}^{\Colour}\OnSecond{\varphi})_{i}=(T\AdjacencyMatrix^{\Colour}\varphi)_{i}=\sum_{k}T_{ik}\biggl(\sum_{l}\AdjacencyMatrix_{kl}^{\Colour}\varphi_{l}\biggr).\label{eq:jthComponentByAdjacency}
\end{equation}
For each value of $k$ in~(\ref{eq:ithComponentByTransplantation})
and~(\ref{eq:jthComponentByAdjacency}), Proposition~\ref{prop:Natural_Sign_Conventions}
says that $\sum_{l}\AdjacencyMatrix_{kl}^{\Colour}\varphi_{l}$ is
a smooth extension of $\varphi_{k}$, showing that $\OnSecond{\varphi}_{i}$
and $\OnSecond{\varphi}_{j}$ are smoothly connected. In order to
prove that $\OnSecond{\varphi}$ satisfies the boundary conditions
given by $(\OnSecond{\AdjacencyMatrix}^{\Colour})_{\Colour=1}^{\NumberOfColours}$,
it suffices to consider outer faces of $\OnSecond M$. Assume that
its $i$th block $\OnSecond B_{i}$ has an outer reflecting face $\partial_{R}^{\Colour}\OnSecond B_{i}$
carrying Neumann or Dirichlet boundary conditions, that is, $\OnSecond{\AdjacencyMatrix}_{ii}^{\Colour}=\pm1$.
Arguing as before, we obtain that
\[
\pm\OnSecond{\varphi}_{i}=\OnSecond{\AdjacencyMatrix}_{ii}^{\Colour}\OnSecond{\varphi}_{i}=(\OnSecond{\AdjacencyMatrix}^{\Colour}\OnSecond{\varphi})_{i}=(T\AdjacencyMatrix^{\Colour}\varphi)_{i}=\sum_{k}T_{ik}\biggl(\sum_{l}\AdjacencyMatrix_{kl}^{\Colour}\varphi_{l}\biggr)
\]
extends $\OnSecond{\varphi}_{i}$ smoothly across $\partial_{R}^{\Colour}\OnSecond B_{i}$.
As the resulting function is symmetric, respectively antisymmetric,
with respect to reflection in $\partial_{R}^{\Colour}\OnSecond B_{i}$,
we see that $\frac{\partial\OnSecond{\varphi}}{\partial n}$, respectively
$\OnSecond{\varphi}$, has to vanish along $\partial_{R}^{\Colour}\OnSecond B_{i}$.
Finally, note that since $\AdjacencyMatrix^{\Colour}=T^{-1}\OnSecond{\AdjacencyMatrix}^{\Colour}(T^{-1})^{-1}$
for all edge colors $\Colour$, the given arguments equally apply
to the inverse transplantation. Hence, the transplantation given by
$T$ is intertwining.

In the following, we show that an intertwining transplantation satisfies~(\ref{eq:TransplantationCondition}).
If $\varphi$ is a solution of the Zaremba problem represented by
$(\AdjacencyMatrix^{\Colour})_{\Colour=1}^{\NumberOfColours}$, then
$\OnSecond{\varphi}=T(\varphi)$ with restrictions $\OnSecond{\varphi}_{i}=\sum_{k}T_{ik}\varphi_{k}$
solves the Zaremba problem represented by $(\OnSecond{\AdjacencyMatrix}^{\Colour})_{\Colour=1}^{\NumberOfColours}$.
If we apply Proposition~\ref{prop:Natural_Sign_Conventions} twice,
we obtain that $\OnSecond{\varphi}_{i}$ is smoothly extended across
$\partial_{R}^{\Colour}\OnSecond B_{i}$ by
\[
\sum_{j}\OnSecond{\AdjacencyMatrix}_{ij}^{\Colour}\OnSecond{\varphi}_{j}\qquad\text{ as well as }\qquad\sum_{k}T_{ik}\biggl(\sum_{l}\AdjacencyMatrix_{kl}^{\Colour}\varphi_{l}\biggr),
\]
which have to coincide according to the Unique Continuation Theorem.
In particular, 
\begin{equation}
\OnSecond{\AdjacencyMatrix}^{\Colour}T\varphi=\OnSecond{\AdjacencyMatrix}^{\Colour}\OnSecond{\varphi}=T\AdjacencyMatrix^{\Colour}\varphi\label{eq:EqualVectors}
\end{equation}
for any solution $\varphi$ of the Zaremba problem on $M$. We show
that this implies $\OnSecond{\AdjacencyMatrix}^{\Colour}T=T\AdjacencyMatrix^{\Colour}$.
Each row $(a_{i1},a_{i2},\ldots,a_{i\NumberOfVertices})$ of $\OnSecond{\AdjacencyMatrix}^{\Colour}T-T\AdjacencyMatrix^{\Colour}$
may be identified with an $L^{2}$-function $\psi^{i}$, whose restriction
to the interior of block $j$ of $M$ is equal to $a_{ij}$. Then,
(\ref{eq:EqualVectors}) says that $\psi^{i}$ is $L^{2}$-orthogonal
to all solutions of the Zaremba problem on $M$. Assumption~A implies
that $\psi^{i}=0$, hence, $a_{i1}=a_{i2}=\cdots=a_{i\NumberOfVertices}=0$,
which completes the proof.
\end{proof}
Intertwining transplantations do not superpose Dirichlet with Neumann
boundary conditions in the following sense. If the $k$th block of
$M$ and the $i$th block of $\OnSecond M$ have outer faces that
correspond to the same edge color~$\Colour$, but that carry different
boundary conditions, that is, if $\AdjacencyMatrix_{kk}^{\Colour}\OnSecond{\AdjacencyMatrix}_{ii}^{\Colour}=-1$,
then $T_{ik}=0$ since (\ref{eq:TransplantationCondition}) implies
\[
T_{ik}\AdjacencyMatrix_{kk}^{\Colour}=\left(T\AdjacencyMatrix^{\Colour}\right)_{ik}=(\OnSecond{\AdjacencyMatrix}^{\Colour}T)_{ik}=\OnSecond{\AdjacencyMatrix}_{ii}^{\Colour}T_{ik}.
\]

Okada and Shudo~\cite{OkadaShudo2001} studied the relation between
isospectrality and isolength spectrality and derived a special case
of the following theorem.

\begin{TraceTheorem}\label{thm:TraceTheorem}Two loop-signed graphs
with adjacency matrices $(\AdjacencyMatrix^{\Colour})_{\Colour=1}^{\NumberOfColours}$
and $(\OnSecond{\AdjacencyMatrix}^{\Colour})_{\Colour=1}^{\NumberOfColours}$
are transplantable if and only if for all finite sequences $\Colour_{1}\Colour_{2}\ldots\Colour_{l}$
of edge colors
\begin{equation}
\Trace(\OnSecond{\AdjacencyMatrix}^{\Colour_{1}}\OnSecond{\AdjacencyMatrix}^{\Colour_{2}}\cdots\OnSecond{\AdjacencyMatrix}^{\Colour_{l}})=\Trace(\AdjacencyMatrix^{\Colour_{1}}\AdjacencyMatrix^{\Colour_{2}}\cdots\AdjacencyMatrix^{\Colour_{l}}).\label{eq:TraceCondition}
\end{equation}

\end{TraceTheorem}
\begin{proof}
The proof in~\cite{OkadaShudo2001} extends to our setting. At first,
assume that the graphs are transplantable. According to the Transplantation
Theorem, there exists an invertible real matrix $T$ satisfying 
\[
\OnSecond{\AdjacencyMatrix}^{\Colour}=T\AdjacencyMatrix^{\Colour}T^{-1}\qquad\text{for }\Colour=1,2,\ldots,\NumberOfColours,
\]
which immediately implies (\ref{eq:TraceCondition}) using the cyclic
invariance of the trace. In order to show that (\ref{eq:TraceCondition})
implies transplantability, we consider the groups 
\begin{equation}
G=\langle\AdjacencyMatrix^{1},\AdjacencyMatrix^{2},\ldots,\AdjacencyMatrix^{\NumberOfColours}\rangle\qquad\mbox{and}\qquad\OnSecond G=\langle\OnSecond{\AdjacencyMatrix}^{1},\OnSecond{\AdjacencyMatrix}^{2},\ldots,\OnSecond{\AdjacencyMatrix}^{\NumberOfColours}\rangle.\label{eq:Operator_Groups}
\end{equation}
Since the empty word $\varnothing$ is associated with the identity
matrices $I_{\NumberOfVertices}\in G$ and $I_{\OnSecond{\NumberOfVertices}}\in\OnSecond G$,
each of the graphs has $\NumberOfVertices=\OnSecond{\NumberOfVertices}$
vertices. Note that $G$ and $\OnSecond G$ are finite as they act
faithfully on
\[
\{\BasisVector_{1},\BasisVector_{2},\ldots,\BasisVector_{\NumberOfVertices},-\BasisVector_{1},-\BasisVector_{2},\ldots,-\BasisVector_{\NumberOfVertices}\},
\]
where $\BasisVector_{i}$ denotes the $i$th standard basis vector
of $\mathbb{R}^{\NumberOfVertices}$. We let $F^{\NumberOfColours}$
denote the free group generated by the letters $1,2,\ldots,\NumberOfColours$
and consider the surjective homomorphism $\HomFKtoG\colon F^{\NumberOfColours}\twoheadrightarrow G$
given by
\[
\HomFKtoG(\Colour_{1}^{\pm1}\Colour_{2}^{\pm1}\ldots\Colour_{l}^{\pm1})=\AdjacencyMatrix^{\Colour_{1}}\AdjacencyMatrix^{\Colour_{2}}\cdots\AdjacencyMatrix^{\Colour_{l}},
\]
which is well-defined as adjacency matrices are self-inverse. We define
$\OnSecond{\HomFKtoG}\colon F^{\NumberOfColours}\twoheadrightarrow\OnSecond G$
analogously. Since $I_{\NumberOfVertices}$ is the only element with
trace $\NumberOfVertices$ in $G$ and in $\OnSecond G$, we have
\[
\begin{array}{ccccc}
\ker(\HomFKtoG) & = & \{\word\in F^{\NumberOfColours}\,|\,\Trace(\HomFKtoG(\word))=\NumberOfVertices\}\\
 & = & \{\word\in F^{\NumberOfColours}\,|\,\Trace(\OnSecond{\HomFKtoG}(\word))=\NumberOfVertices\} & = & \ker(\OnSecond{\HomFKtoG}).
\end{array}
\]
Hence, $\IsoGandGhat\colon G\rightarrow\OnSecond G$ given by $\IsoGandGhat(\HomFKtoG(\word))=\OnSecond{\HomFKtoG}(\word)$
for $w\in F^{\NumberOfColours}$ defines an isomorphism. We obtain
two representations of $G$, 
\[
id_{G}\colon G\rightarrow GL(\NumberOfVertices,\mathbb{C})\qquad\text{and}\qquad id_{\OnSecond G}\circ\IsoGandGhat\colon G\rightarrow GL(V,\mathbb{C}),
\]
where $id_{G}$ and $id_{\OnSecond G}$ are the identity maps on $G$
and $\OnSecond G$, respectively. Their characters are the mappings
$\HomFKtoG(\word)\mapsto\Trace(\HomFKtoG(\word))$ and $\HomFKtoG(\word)\mapsto\Trace(\OnSecond{\HomFKtoG}(\word))$
for $w\in F^{\NumberOfColours}$, which are equal by~(\ref{eq:TraceCondition}).
Thus, $id_{G}$ and $id_{\OnSecond G}\circ\IsoGandGhat$ are equivalent
and there exists $T\in GL(\NumberOfVertices,\mathbb{C})$ such that
\[
T\AdjacencyMatrix^{\Colour}=T\HomFKtoG(\Colour)=\OnSecond{\HomFKtoG}(\Colour)T=\OnSecond{\AdjacencyMatrix}^{\Colour}T\qquad\mbox{for }\Colour=1,2,\ldots,\NumberOfColours.
\]
For $z\in\mathbb{C}$, let $T(z)=\mathrm{Re}(T)+z\mathrm{Im}(T)$
and note that $T(z)\AdjacencyMatrix^{\Colour}=\OnSecond{\AdjacencyMatrix}^{\Colour}T(z)$
for $\Colour=1,2,\ldots,\NumberOfColours$. Since $\det(T(i))\neq0$,
the mapping $z\mapsto\det(T(z))$ defines a non-zero polynomial. Hence,
we can choose $r\in\mathbb{R}$ such that $T(r)\in GL(\NumberOfVertices,\mathbb{R})$,
which completes the proof.
\end{proof}
Thas~et~al.~\cite{Thas2006a,Thas2006b,Thas2006,SchillewaertThas2011}
partially classified the groups appearing in (\ref{eq:Operator_Groups})
for transplantable treelike graphs with uniform loop signs. Recall
that transplantable manifolds have equal heat invariants and therefore
share certain geometric properties, some of which can be identified
with expressions of the form (\ref{eq:TraceCondition}):
\begin{itemize}
\item $\Trace(\varnothing)$ encodes the number of building blocks, that
is, the volume of a tiled manifold.
\item $\Trace(\AdjacencyMatrix^{\Colour})$ encodes the difference of the
Neumann boundary volume and the Dirichlet boundary volume coming from
outer faces that correspond to $\Colour$-colored loops.
\item For tiled manifolds $M$ that are polygons, the quantity 
\begin{equation}
\sum_{DD}\frac{\pi^{2}-\alpha^{2}}{\alpha}+\sum_{NN}\frac{\pi^{2}-\alpha^{2}}{\alpha}-\frac{1}{2}\sum_{DN}\frac{\pi^{2}+2\alpha^{2}}{\alpha}.\label{eq:SpectralInvariantInvolvingAngles}
\end{equation}
is a spectral invariant, where the sums are taken over all corners
of $\partial M$ formed by Dirichlet-Dirichlet ($DD$), Neumann-Neumann
($NN$) and Dirichlet-Neumann ($DN$) sides, respectively; in each
case, $\alpha$ is the corresponding angle~\cite[Theorem 5.1]{LevitinParnovskiPolterovich2006}.
Let $\Colour_{1}$ and $\Colour_{2}$ be two edge colors corresponding
to neighboring sides of the underlying building block, enclosing the
angle $\beta$. If $C_{DD}$, $C_{NN}$ and $C_{DN}$ denote the number
of corners of $M$ that are formed by sides corresponding to $\Colour_{1}$
and $\Colour_{2}$, respectively, carrying Dirichlet-Dirichlet, Neumann-Neumann
and Dirichlet-Neumann boundary conditions, then the contribution of
these corners to~(\ref{eq:SpectralInvariantInvolvingAngles}) is
\[
\frac{\pi^{2}}{\beta}\left(C_{DD}+C_{NN}-\frac{1}{2}C_{DN}\right)-\beta\left(C_{DD}+C_{NN}+C_{DN}\right).
\]
If the underlying graph is treelike, this quantity is determined by
\begin{eqnarray*}
\Trace(\AdjacencyMatrix^{\Colour_{1}}\AdjacencyMatrix^{\Colour_{2}}) & = & C_{DD}+C_{NN}-C_{DN}\text{ and}\\
\Trace(\AdjacencyMatrix^{\Colour_{1}}\AdjacencyMatrix^{\Colour_{2}}\AdjacencyMatrix^{\Colour_{1}}\AdjacencyMatrix^{\Colour_{2}}) & = & C_{DD}+C_{NN}+C_{DN}.
\end{eqnarray*}

\end{itemize}
The Trace Theorem allows for a computer-aided search for transplantable
pairs as described in Section~\ref{sub:TheAlgorithm}.

\section{Transplantability and induced representations\label{sec:IndRepAndTrans}}

In order to characterize transplantability in group-theoretic terms,
we recall the Sunada method~\cite{Sunada1985}.
\begin{defn}
A triple $(G,H,\OnSecond H)$ consisting of a finite group $G$ with
subgroups $H$ and $\OnSecond H$ is called a Gassmann triple if each
conjugacy class $[g]\subseteq G$ satisfies
\begin{equation}
|\left[g\right]\cap H|=|\left[g\right]\cap\OnSecond H|.\label{eq:GassmannCondition}
\end{equation}
\end{defn}
\begin{thm}
(\cite{Sunada1985}) If $(G,H,\OnSecond H)$ is a Gassmann triple
such that $G$ acts freely on some closed Riemannian manifold $M$
by isometries, then $M/H$ and $M/\OnSecond H$ are isospectral.
\end{thm}
The Gassmann criterion (\ref{eq:GassmannCondition}) is equivalent
to the condition that the induced representations $\InducedRep HG{\boldsymbol{1}_{H}}$
and $\InducedRep{\OnSecond H}G{\boldsymbol{1}_{\OnSecond H}}$ of
the trivial representations $\boldsymbol{1}_{H}$ and $\boldsymbol{1}_{\OnSecond H}$
of $H$ and $\OnSecond H$ are equivalent~\cite{Brooks1999}. Band~et~al.
extended the Sunada method using non-trivial representations \cite{BandParzanchevskiBen-Shach2009,ParzanchevskiBand2010}.
In particular, they showed that if $G$ is a finite group of isometries
of the tiled manifold $M$ mapping $\partial_{D}M$ and $\partial_{N}M$
to themselves and if $H$ and $\OnSecond H$ are subgroups of $G$
with one-dimensional real representations $R$ and $\OnSecond R$
such that $\InducedRep HGR\simeq\InducedRep{\OnSecond H}G{\OnSecond R}$,
then $R$ and $\OnSecond R$ give rise to mixed boundary conditions
on $M/H$ and $M/\OnSecond H$ turning them into isospectral manifolds.
Considering the action of $G$ on $L^{2}(M)$, they showed that $R$
and $\OnSecond R$ can be used to single out spaces of solutions of
the Zaremba problem on $M$ with a particular transformation behavior
with respect to crossings between fundamental domains of the actions
of $H$ and $\OnSecond H$ on $M$. As an example, take the square
$\Square$ in Figure~\ref{fig:TheSquareS} carrying Dirichlet boundary
conditions. Its isometry group is the dihedral group
\[
G=D_{4}=\{e,\sigma,\sigma^{2},\sigma^{3},\tau,\tau\sigma,\tau\sigma^{2},\tau\sigma^{3}\},
\]
where $\tau$ and $\sigma$ denote vertical reflection and rotation
by $\frac{\pi}{2}$, respectively. The group $G$ has subgroups $H$
and $\OnSecond H$ with respective representations $R$ and $\OnSecond R$
given by 
\begin{equation}
\begin{array}{l}
H=\{e,\tau,\tau\sigma^{2},\sigma^{2}\}\\
\OnSecond H=\{e,\tau\sigma,\tau\sigma^{3},\sigma^{2}\}
\end{array}\begin{array}{l}
R\colon e\mapsto1,\tau\mapsto-1,\tau\sigma^{2}\mapsto1,\sigma^{2}\mapsto-1\\
\OnSecond R\colon e\mapsto1,\tau\sigma\mapsto1,\tau\sigma^{3}\mapsto-1,\sigma^{2}\mapsto-1,
\end{array}\label{eq:SubgroupsHandHhatAndRepsRandRhat}
\end{equation}
such that $\InducedRep HGR\simeq\InducedRep{\OnSecond H}G{\OnSecond R}$.
The domains $\Square/R$ and $\Square/\OnSecond R$ in Figure~\ref{fig:SquareAndTri}
are fundamental domains for the actions of $H$ and $\OnSecond H$
on $\Square$ and their boundary conditions are determined by $R$
and $\OnSecond R$ as follows. Using a slight generalization of Proposition~\ref{prop:Natural_Sign_Conventions},
one sees that each solution of the Zaremba problem on $\Square/R$
gives rise to a solution of the Zaremba problem on $\Square$ which
transforms according to $R$ and vice versa, that is, a solution which
is horizontally symmetric ($\tau\sigma^{2}\mapsto1$) and vertically
antisymmetric ($\tau\mapsto-1$), similarly for $\Square/\OnSecond R$
and $\OnSecond R$. Using Frobenius reciprocity, one can show that
the existence of this mapping of solutions together with $\InducedRep HGR\simeq\InducedRep{\OnSecond H}G{\OnSecond R}$
implies transplantability of $\Square/R$ and $\Square/\OnSecond R$
\cite[Corollary 7.6]{BandParzanchevskiBen-Shach2009}. The isospectrality
of these domains had been established by explicit computation beforehand~\cite{LevitinParnovskiPolterovich2006}.
We complete these findings with the following characterization of
transplantability.

\begin{figure}
\noindent \begin{centering}
\hfill{}\subfloat[Square $\Square$ with axes of reflection\label{fig:TheSquareS}]{\noindent \centering{}%
\begin{minipage}[b]{70mm}%
\begin{center}
\psfrag{1}{\raisebox{-1mm}{\hspace{-3mm}$\tau\sigma^{2}\mapsto (1)$}}
\psfrag{2}{\raisebox{1mm}{\hspace{0mm}$\tau\sigma^{3}\mapsto (-1)$}}
\psfrag{3}{\raisebox{0.0mm}{\hspace{0mm}$\tau\mapsto (-1)$}}
\psfrag{4}{\raisebox{0.0mm}{\hspace{0mm}$\tau\sigma\mapsto (1)$}}
\psfrag{5}{\raisebox{-4.5mm}{\hspace{1.5mm}$1$}}
\psfrag{6}{\raisebox{0.0mm}{\hspace{1.5mm}$2$}}
\psfrag{7}{\raisebox{-0.5mm}{\hspace{1.5mm}$3$}}
\psfrag{8}{\raisebox{0.0mm}{\hspace{5mm}$4$}}\includegraphics[scale=0.55]{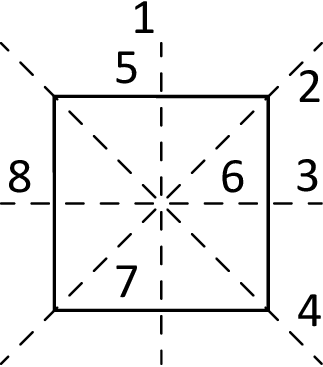}
\par\end{center}%
\end{minipage}}\hfill{}\subfloat[Isospectral quotients\label{fig:SquareAndTri}]{\noindent \centering{}%
\begin{minipage}[b]{60mm}%
\noindent \begin{center}
\raisebox{5.3mm}{\psfrag{1}{\raisebox{5mm}{\hspace{3mm}$\Square/ R$}}\includegraphics[scale=0.55]{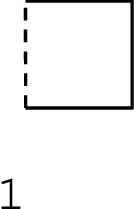}}\hspace{5mm}\raisebox{4.6mm}{\psfrag{1}{\hspace{-1mm}$\Square/\OnSecond R$}\includegraphics[scale=0.55]{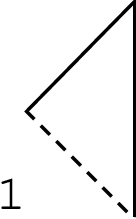}}
\par\end{center}%
\end{minipage}}\hfill{}
\par\end{centering}

\caption{The method of Band et al.~\cite{BandParzanchevskiBen-Shach2009,ParzanchevskiBand2010}.\label{fig:BandsMethod}}
\end{figure}

\begin{thm}
\label{thm:Transplantability_is_Equivalent_Inductions}Each pair of
transplantable loop-signed graphs gives rise to a triple
\begin{equation}
(G,((H_{i},R_{i}))_{i},((\OnSecond H_{j},\OnSecond R_{j}))_{j}),\label{eq:Group_Data}
\end{equation}
consisting of a finite group $G$ and two tuples of pairs of the form
$(H,R)$, where $H$ is a subgroup of $G$ and $R$ is a one-dimensional
real representation of $H$, such that 
\begin{equation}
\bigoplus_{i}\InducedRep{H_{i}}G{R_{i}}\simeq\bigoplus_{j}\InducedRep{\OnSecond H_{j}}G{\OnSecond R_{j}}.\label{eq:Sums_of_Inductions_Equivalent}
\end{equation}
Moreover, if the graphs have adjacency matrices $(\AdjacencyMatrix^{\Colour})_{\Colour=1}^{\NumberOfColours}$
and $(\OnSecond{\AdjacencyMatrix}^{\Colour})_{\Colour=1}^{\NumberOfColours}$,
respectively, then they are isomorphic to the following unions of
Schreier coset graphs defined as below 
\[
\bigcup_{i}\CayleyGraph G{(\AdjacencyMatrix^{\Colour})_{\Colour=1}^{\NumberOfColours}}/R_{i}\qquad\text{and}\qquad\bigcup_{j}\CayleyGraph G{(\AdjacencyMatrix^{\Colour})_{\Colour=1}^{\NumberOfColours}}/\OnSecond R_{j},
\]
where $\CayleyGraph G{(\AdjacencyMatrix^{\Colour})_{\Colour=1}^{\NumberOfColours}}$
denotes the Cayley graph of the group $G$ with respect to its generators
$(\AdjacencyMatrix^{\Colour})_{\Colour=1}^{\NumberOfColours}$. In
other words, the graphs can be recovered from $(G,((H_{i},R_{i}))_{i},((\OnSecond H_{j},\OnSecond R_{j}))_{j})$
up to isomorphism.
\end{thm}
In contrast to~\cite{BandParzanchevskiBen-Shach2009,ParzanchevskiBand2010},
we do not assume that the group $G$ in~(\ref{eq:Group_Data}) is
given as a group of isometries of some covering manifold. Instead,
we use the group structure of $G$ to construct such a cover and obtain
quotients which are transplantable precisely if (\ref{eq:Sums_of_Inductions_Equivalent})
holds.

In the following, let $\Graph=\bigcup_{i}\Gamma_{i}$ and $\OnSecond{\Graph}=\bigcup_{j}\OnSecond{\Gamma}_{j}$
be transplantable loop-signed graphs with edge colors $1,2,\ldots,\NumberOfColours$,
vertices $1,2,\ldots,\NumberOfVertices$ and connected components
$(\Gamma_{i})_{i}$ and $(\OnSecond{\Gamma}_{j})_{j}$, respectively.
We let $(\AdjacencyMatrix^{\Colour})_{\Colour=1}^{\NumberOfColours}$
and $(\OnSecond{\AdjacencyMatrix}^{\Colour})_{\Colour=1}^{\NumberOfColours}$
denote their $\NumberOfVertices\times\NumberOfVertices$ adjacency
matrices and consider the groups 
\[
G=\langle\AdjacencyMatrix^{1},\AdjacencyMatrix^{2},\ldots,\AdjacencyMatrix^{\NumberOfColours}\rangle\qquad\mbox{and}\qquad\OnSecond G=\langle\OnSecond{\AdjacencyMatrix}^{1},\OnSecond{\AdjacencyMatrix}^{2},\ldots,\OnSecond{\AdjacencyMatrix}^{\NumberOfColours}\rangle.
\]
Note that each $\Gamma_{i}$ gives rise to an invariant subspace of
the action of $G$ on $\mathbb{R}^{\NumberOfVertices}$, similarly
for $(\OnSecond{\Gamma}_{j})_{j}$ and $\OnSecond G$. By assumption,
there exists a transplantation matrix $T\in GL(\NumberOfVertices,\mathbb{R})$
such that
\[
\OnSecond{\AdjacencyMatrix}^{\Colour}=T\AdjacencyMatrix^{\Colour}T^{-1}\qquad\mbox{for }\Colour=1,2,\ldots,\NumberOfColours.
\]
In particular, the conjugation map $\IsoGandGhat\colon G\to\OnSecond G$
given by 
\begin{equation}
\IsoGandGhat(\AdjacencyMatrix^{\Colour_{1}}\AdjacencyMatrix^{\Colour_{2}}\cdots\AdjacencyMatrix^{\Colour_{l}})=\OnSecond{\AdjacencyMatrix}^{\Colour_{1}}\OnSecond{\AdjacencyMatrix}^{\Colour_{2}}\cdots\OnSecond{\AdjacencyMatrix}^{\Colour_{l}}\label{eq:IsoGandGhat}
\end{equation}
is an isomorphism. The main observation is that $G$ and $\OnSecond G$
act on the following set of two-sets
\[
\{\{\BasisVector_{1},-\BasisVector_{1}\},\{\BasisVector_{2},-\BasisVector_{2}\},\ldots,\{\BasisVector_{\NumberOfVertices},-\BasisVector_{\NumberOfVertices}\}\},
\]
where $\BasisVector_{i}$ denotes the $i$th standard basis vector
of $\mathbb{R}^{\NumberOfVertices}$. The point stabilizer subgroups
of these actions will turn out to encode $\Graph$ and $\OnSecond{\Graph}$,
respectively. Let $(v_{i})_{i}$ and $(\OnSecond v_{j})_{j}$ be the
respective smallest vertices in $(\Gamma_{i})_{i}$ and $(\OnSecond{\Gamma}_{j})_{j}$.
In particular, $v_{i}$ is in $\Graph_{i}$ and $\OnSecond v_{j}$
is in $\OnSecond{\Graph}_{j}$ for every $i$ and $j$.
\begin{defn}
\label{def:The_associated_subgroups_and_representations}The associated\emph{
}pairs $((H_{i},R_{i}))_{i}$ and $((\OnSecond H_{j},\OnSecond R_{j}))_{j}$
consist of subgroups $(H_{i}){}_{i}$ and $(\OnSecond H_{j}){}_{j}$
of $G$ and respective one-dimensional real representations $(R_{i})_{i}$
and $(\OnSecond R_{j})_{j}$ defined as\setlength{\arraycolsep}{1pt}
\begin{equation}
\begin{array}{cclcll}
H_{i} & = & G_{\{\BasisVector_{v_{i}},-\BasisVector_{v_{i}}\}} & = & \left\{ \ElementInG\in G\,|\,\ElementInG_{v_{i}v_{i}}=\pm1\right\}  & R_{i}(\ElementInG)=\ElementInG_{v_{i}v_{i}}\\
\OnSecond H_{j} & = & \IsoGandGhat^{-1}\left(\OnSecond G_{\{\BasisVector_{\OnSecond v_{j}},-\BasisVector_{\OnSecond v_{j}}\}}\right) & = & \left\{ \ElementInG\in G\,|\,(\IsoGandGhat(\ElementInG))_{\OnSecond v_{j}\OnSecond v_{j}}=\pm1\right\}  & \OnSecond R_{j}(\ElementInG)=(\IsoGandGhat(\ElementInG))_{\OnSecond v_{j}\OnSecond v_{j}}.
\end{array}\label{eq:AssSubGroupsAndRepsDef}
\end{equation}

\end{defn}
\setlength{\arraycolsep}{\myArraycolsep}Note that for any $v$ in
$\Graph_{i}$ with $v\neq v_{i}$, there exists $\ElementInG\in G$
with $\ElementInG_{v_{i}v}=1$ and $G_{\{\BasisVector_{v}-\BasisVector_{v}\}}=\ElementInG^{-1}H_{i}\ElementInG$.
In other words, the point stabilizer subgroups are conjugated in $G$.
The orbit-stabilizer theorem implies that the index $[G:H_{i}]$ of
$H_{i}$ in $G$ equals the number of vertices of $\Graph_{i}$. Similar
statements hold for $(\OnSecond{\Graph}_{j})_{j}$ and $(\OnSecond H_{j})_{j}$.

Definition~\ref{def:The_associated_subgroups_and_representations}
has the following geometrical motivation. Let $\Graph^{DC}$ be the
double cover of $\Graph$ obtained by taking copies $\Graph^{+}$
and $\Graph^{-}$ of the loopless version of $\Graph$ and joining
their respective $i$th vertices with a $\Colour$-colored link whenever
$\Graph$ has a $\Colour$-colored Dirichlet loop at its $i$th vertex.
Note that $G$ entails a faithful permutation action on the vertices
of $\Graph^{DC}$ and we may interpret products of adjacency matrices
of $\Graph$ as walks on $\Graph^{DC}$. According to the Reflection
Principle, each solution of a Zaremba problem corresponding to $\Graph$
gives rise to a Neumann eigenfunction on the manifold corresponding
to $\Graph^{DC}$ which is antisymmetric with respect to interchanging
$\Graph^{+}$ and $\Graph^{-}$. In other words, whenever we cross
from $\Graph^{+}$ to $\Graph^{-}$, a solution changes sign which
is incorporated in (\ref{eq:AssSubGroupsAndRepsDef}). The following
theorem is the converse of the method of Band et al.~\cite{BandParzanchevskiBen-Shach2009,ParzanchevskiBand2010}.
\begin{thm}
Let $\Graph=\bigcup_{i}\Gamma_{i}$ and $\OnSecond{\Graph}=\bigcup_{j}\OnSecond{\Gamma}_{j}$
be as above.\label{thm:TransGraphsGiveEquivInds}\end{thm}
\begin{enumerate}
\item The representation $id_{G}$ is equivalent to $\bigoplus_{i}\InducedRep{H_{i}}G{R_{i}}$.

\begin{enumerate}
\item The representation $id_{\OnSecond G}\circ\IsoGandGhat$ given by $\AdjacencyMatrix^{\Colour_{1}}\AdjacencyMatrix^{\Colour_{2}}\cdots\AdjacencyMatrix^{\Colour_{l}}\mapsto\OnSecond{\AdjacencyMatrix}^{\Colour_{1}}\OnSecond{\AdjacencyMatrix}^{\Colour_{2}}\cdots\OnSecond{\AdjacencyMatrix}^{\Colour_{l}}$
is equivalent to $\bigoplus_{j}\InducedRep{\OnSecond H_{j}}G{\OnSecond R_{j}}$.
\item We have 
\[
\bigoplus_{i}\InducedRep{H_{i}}G{R_{i}}\simeq\bigoplus_{j}\InducedRep{\OnSecond H_{j}}G{\OnSecond R_{j}}.
\]
\label{enu:InductionsEquiv}
\end{enumerate}
\end{enumerate}
Theorem~\ref{thm:TransGraphsGiveEquivInds} follows from Theorem~\ref{thm:TransGraphsAreSchreierGraphs}
and Theorem~\ref{thm:TransSchreierGraphs} below. In order to reconstruct
$\Graph$ and $\OnSecond{\Graph}$ from the data in~(\ref{eq:AssSubGroupsAndRepsDef}),
we consider the action of $G$ on its Cayley graph.
\begin{defn}
Let $G$ be a finite group generated by elements $(\GeneratorsOfG^{\Colour})_{\Colour=1}^{\NumberOfColours}$
of order $2$. The Cayley\emph{ }graph $\CayleyGraph G{(\GeneratorsOfG^{\Colour})_{\Colour=1}^{\NumberOfColours}}$
is the isomorphism class of edge-colored graphs with edge colors $1,2,\ldots,\NumberOfColours$
and $|G|$ vertices, which are identified with the elements of $G$.
Moreover, the vertices $\ElementInG,\ElementInG'\in G$ are joined
by a $\Colour$-colored edge if and only if $\ElementInG'=\ElementInG\GeneratorsOfG^{\Colour}$.\label{def:CayleyGraph}
\end{defn}
As an example, Figure~\ref{fig:CayleyGraph} shows $\CayleyGraph{D_{4}}{(\tau,\tau\sigma^{3})}$.
The striking similarity with the square~$\Square$, on which $D_{4}$
acts by isometries, originates from the action of $D_{4}$ on $\CayleyGraph{D_{4}}{(\tau,\tau\sigma^{3})}$.
More precisely, each $\ElementInG\in D_{4}$ entails an isomorphism
of $\CayleyGraph{D_{4}}{(\tau,\tau\sigma^{3})}$ by mapping $h\in D_{4}$
to $\ElementInG h$. Note that, in general, two vertices $\ElementInG h_{1},\ElementInG h_{2}\in\CayleyGraph G{(\GeneratorsOfG^{\Colour})_{\Colour=1}^{\NumberOfColours}}$
are joined by a $\Colour$-colored edge if and only if the same holds
for $h_{1},h_{2}\in\CayleyGraph G{(\GeneratorsOfG^{\Colour})_{\Colour=1}^{\NumberOfColours}}$
since $\ElementInG h_{1}=\ElementInG h_{2}\GeneratorsOfG^{\Colour}$
if and only if $h_{1}=h_{2}\GeneratorsOfG^{\Colour}$. In Figure~\ref{fig:CayleyGraph},
the actions of $\tau$ and $\tau\sigma^{2}$ on $\CayleyGraph{D_{4}}{(\tau,\tau\sigma^{3})}$
correspond to their former actions on $\Square$, that is, vertical
and horizontal reflection. The pairs $(H,R)$ and $(\OnSecond H,\OnSecond R)$
in~(\ref{eq:SubgroupsHandHhatAndRepsRandRhat}) give rise to quotients
of $\CayleyGraph{D_{4}}{(\tau,\tau\sigma^{3})}$ defined as follows.
\begin{defn}
\label{def:SchreierCosetGraph}Let $G$ be a finite group generated
by elements $(\GeneratorsOfG^{\Colour})_{\Colour=1}^{\NumberOfColours}$
of order $2$ and let $H$ be a subgroup of $G$ with one-dimensional
real representation $R$. The Schreier\emph{ }coset\emph{ }graph\emph{
}$\CayleyGraph G{(\GeneratorsOfG^{\Colour})_{\Colour=1}^{\NumberOfColours}}/R$
is the isomorphism class of loop-signed graphs with edge colors $1,2,\ldots,\NumberOfColours$
and $[G:H]$ vertices, which are identified with the right cosets
of $H$. Moreover, the vertices $H\ElementInG$ and $H\ElementInG'$
are joined by a $\Colour$-colored edge if and only if $H\ElementInG'=H\ElementInG\GeneratorsOfG^{\Colour}$,
and $H\ElementInG$ has a $\Colour$-colored loop carrying a Neumann
or Dirichlet sign, respectively, if and only if $H\ElementInG=H\ElementInG\GeneratorsOfG^{\Colour}$
and $R(\ElementInG\GeneratorsOfG^{\Colour}\ElementInG^{-1})=\pm1$.
\end{defn}
Although we defined Cayley graphs and Schreier coset graphs only up
to isomorphism, we use the notation $\CayleyGraph G{(\GeneratorsOfG^{\Colour})_{\Colour=1}^{\NumberOfColours}}/R$
to mean any of its isomorphic elements. Figure~\ref{fig:TransQuotients}
shows $\CayleyGraph{D_{4}}{(\tau,\tau\sigma^{3})}/R$ and $\CayleyGraph{D_{4}}{(\tau,\tau\sigma^{3})}/\OnSecond R$,
where $R$ and $\OnSecond R$ are given by (\ref{eq:SubgroupsHandHhatAndRepsRandRhat}).
These graphs underlie the transplantable broken square and triangle
in Figure~\ref{fig:SquareAndTri} which can be explained as follows.

\begin{figure}
\noindent \begin{centering}
\hfill{}\subfloat[Cayley graph $\CayleyGraph{D_{4}}{(\tau,\tau\sigma^{3})}$\label{fig:CayleyGraph}]{\noindent \centering{}%
\begin{minipage}[b]{65mm}%
\begin{center}
\psfrag{1}{\raisebox{0.0mm}{\hspace{3mm}$\sigma$}} \psfrag{2}{\raisebox{0.0mm}{\hspace{1mm}$\tau\sigma^{3}$}} \psfrag{3}{\raisebox{0.0mm}{\hspace{2mm}$e$}}
\psfrag{4}{\raisebox{0.0mm}{\hspace{2mm}$\tau$}}
\psfrag{5}{\raisebox{0.5mm}{\hspace{2mm}$\sigma^{3}$}}
\psfrag{6}{\raisebox{0.5mm}{$\tau\sigma$}}
\psfrag{7}{\raisebox{0.0mm}{$\sigma^{2}$}}
\psfrag{8}{\raisebox{0.0mm}{$\tau\sigma^{2}$}}\includegraphics[scale=0.42]{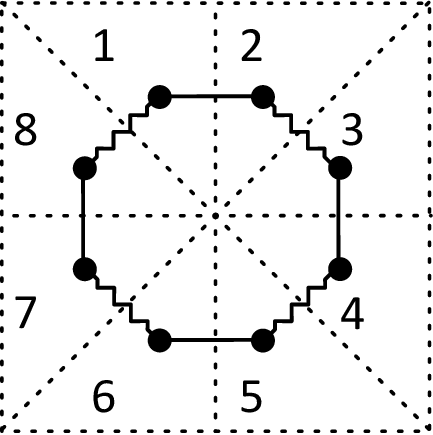}
\par\end{center}%
\end{minipage}}\hfill{}\subfloat[Transplantable quotients\label{fig:TransQuotients}]{\noindent \centering{}%
\begin{minipage}[b]{70mm}%
\hspace{15mm}\psfragBodyNew\raisebox{1.2mm}{\psfrag{5}{\raisebox{8mm}{\hspace{-13mm}$\CayleyGraph{D_{4}}{(\tau,\tau\sigma^{3})}/R$}} \psfrag{1}{$1$} \psfrag{2}{$2$}  \psfrag{3}{}	\includegraphics[scale=0.42]{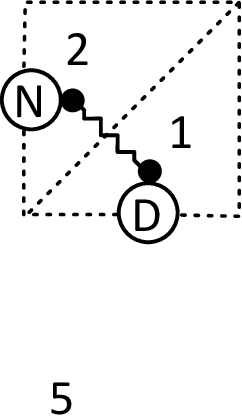}}\hspace{10mm}
\raisebox{0.3mm}{\psfrag{5}{\raisebox{-0mm}{\hspace{-24mm}$\CayleyGraph{D_{4}}{(\tau,\tau\sigma^{3})}/\OnSecond R$}} \psfrag{1}{$1$} \psfrag{2}{$2$}\includegraphics[scale=0.42]{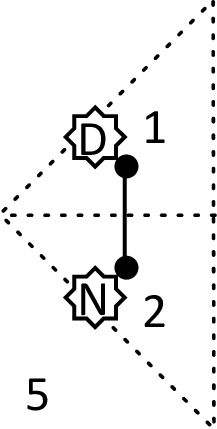}}%
\end{minipage}}\hfill{}
\par\end{centering}

\caption{Schreier coset graphs corresponding to (\ref{eq:SubgroupsHandHhatAndRepsRandRhat}).\label{fig:TransGraphFromCayleyCover}}
\end{figure}

\begin{table}
\noindent \centering{}%
\begin{tabular}{c|c|c|c|c}
$\AdjacencyMatrix^{straight}$ & $\AdjacencyMatrix^{zig-zag}$ & $T=T^{T}=T^{-1}$ & $\OnSecond{\AdjacencyMatrix}^{straight}$ & $\OnSecond{\AdjacencyMatrix}^{zig-zag}$\tabularnewline
\hline 
\noalign{\vskip-2mm}
 &  &  &  & \tabularnewline
$\left(\begin{array}{cc}
-1 & 0\\
0 & 1
\end{array}\right)$ & $\left(\begin{array}{cc}
0 & 1\\
1 & 0
\end{array}\right)$ & $\frac{1}{\sqrt{2}}\left(\begin{array}{cc}
-1 & 1\\
1 & 1
\end{array}\right)$ & $\left(\begin{array}{cc}
0 & 1\\
1 & 0
\end{array}\right)$ & $\left(\begin{array}{cc}
-1 & 0\\
0 & 1
\end{array}\right)$\tabularnewline[4mm]
$(1,3)(2)(4)$ & $(1,2)(3,4)$ &  & $(1,2)(3,4)$ & $(1,3)(2)(4)$\tabularnewline[1mm]
$\tau$ & $\tau\sigma^{3}$ &  & $\IsoGandGhat(\tau)$ & $\IsoGandGhat(\tau\sigma^{3})$\tabularnewline
\noalign{\vskip2mm}
\end{tabular}\caption{Adjacency matrices of the graphs in Figure~\ref{fig:TransQuotients}
and their action on $\{\BasisVector_{1},\BasisVector_{2},-\BasisVector_{1},-\BasisVector_{2}\}$.\label{tab:AdjMatPlusPermImage}}
\end{table}

\begin{thm}
\label{thm:TransGraphsAreSchreierGraphs}Let $\Graph=\bigcup_{i}\Gamma_{i}$
and $\OnSecond{\Graph}=\bigcup_{j}\OnSecond{\Gamma}_{j}$ be transplantable
loop-signed graphs with adjacency matrices $(\AdjacencyMatrix^{\Colour})_{\Colour=1}^{\NumberOfColours}$
and $(\OnSecond{\AdjacencyMatrix}^{\Colour})_{\Colour=1}^{\NumberOfColours}$
generating $G$ and $\OnSecond G$, respectively. If $((H_{i},R_{i}))_{i}$
and $((\OnSecond H_{j},\OnSecond R_{j}))_{j}$ denote the associated
pairs given by~(\ref{eq:AssSubGroupsAndRepsDef}), then the graphs
$\bigcup_{i}\CayleyGraph G{(\AdjacencyMatrix^{\Colour})_{\Colour=1}^{\NumberOfColours}}/R_{i}$
and $\bigcup_{j}\CayleyGraph G{(\AdjacencyMatrix^{\Colour})_{\Colour=1}^{\NumberOfColours}}/\OnSecond R_{j}$
are isomorphic to $\Graph$ and $\OnSecond{\Graph}$, respectively.
More precisely, for each $i$ and $j$ 
\begin{equation}
\Gamma_{i}\simeq\CayleyGraph G{(\AdjacencyMatrix^{\Colour})_{\Colour=1}^{\NumberOfColours}}/R_{i}\quad\text{and}\quad\OnSecond{\Gamma}_{j}\simeq\CayleyGraph G{(\AdjacencyMatrix^{\Colour})_{\Colour=1}^{\NumberOfColours}}/\OnSecond R_{j}.\label{eq:Graph_Components_are_Schreier_Graphs}
\end{equation}
\end{thm}
\begin{proof}
It suffices to prove~(\ref{eq:Graph_Components_are_Schreier_Graphs}).
In order to simplify notation, we fix some $\Graph_{i}$ and assume
it has vertices $\{1,2,\ldots,\NumberOfVertices_{i}\}$, in particular,
$v_{i}=1$ and $H_{i}=\{\ElementInG\in G\,|\,\ElementInG_{11}=\pm1\}$.
The orbit-stabilizer theorem implies that $[G:H_{i}]=\NumberOfVertices_{i}$
and for each $k\in\{1,2,\ldots,\NumberOfVertices_{i}\}$ there exists
$\ElementInG_{k}\in G$ such that $\ElementInG_{k}\BasisVector_{k}=\BasisVector_{1}$.
The elements $(\ElementInG_{k})_{k=1}^{\NumberOfVertices_{i}}$ yield
a system of representatives for the right cosets of $H_{i}$ since
$H_{i}\ElementInG_{k}\cap H_{i}\ElementInG_{l}\neq\varnothing$ implies
$\ElementInG_{k}\ElementInG_{l}^{-1}\in H_{i}=\left\{ g\in G\,|\, g\BasisVector_{1}=\pm\BasisVector_{1}\right\} $,
that is, $\ElementInG_{l}^{-1}\BasisVector_{1}=\pm\ElementInG_{k}^{-1}\BasisVector_{1}$,
which happens only if $k=l$. By definition, the vertices $H_{i}\ElementInG_{k}$
and $H_{i}\ElementInG_{l}$ of $\CayleyGraph G{(\AdjacencyMatrix^{\Colour})_{\Colour=1}^{\NumberOfColours}}/R_{i}$
are joined by a $\Colour$-colored edge if and only if $H_{i}\ElementInG_{k}\AdjacencyMatrix^{\Colour}=H_{i}\ElementInG_{l}$,
which happens precisely if $\ElementInG_{k}\AdjacencyMatrix^{\Colour}\ElementInG_{l}^{-1}\BasisVector_{1}=\ElementInG_{k}\AdjacencyMatrix^{\Colour}\BasisVector_{l}=\pm\BasisVector_{1}$,
that is, if $\AdjacencyMatrix_{kl}^{\Colour}=\pm1$. Moreover, the
vertex $H_{i}\ElementInG_{k}$ has a $\Colour$-colored Neumann or
Dirichlet loop if and only if $\ElementInG_{k}\AdjacencyMatrix^{\Colour}\ElementInG_{k}^{-1}\in H_{i}$
and $R_{i}(\ElementInG_{k}\AdjacencyMatrix^{\Colour}\ElementInG_{k}^{-1})=\pm1$,
which is equivalent to $\ElementInG_{k}\AdjacencyMatrix^{\Colour}\ElementInG_{k}^{-1}\BasisVector_{1}=\ElementInG_{k}\AdjacencyMatrix^{\Colour}\BasisVector_{k}=\pm\BasisVector_{1}$,
that is, $\AdjacencyMatrix_{kk}^{\Colour}=\pm1$. Hence, the first
$\NumberOfVertices_{i}\times\NumberOfVertices_{i}$ block of $\AdjacencyMatrix^{\Colour}$
coincides with the adjacency matrix corresponding to the edge color
$\Colour$ of $\CayleyGraph G{(\AdjacencyMatrix^{\Colour})_{\Colour=1}^{\NumberOfColours}}/R_{i}$
up to a renumbering of its vertices. The statement $\OnSecond{\Gamma}_{j}\simeq\CayleyGraph G{(\AdjacencyMatrix^{\Colour})_{\Colour=1}^{\NumberOfColours}}/\OnSecond R_{j}$
is proven along the same lines using the isomorphism~(\ref{eq:IsoGandGhat}).
\end{proof}
As an example, start with the transplantable graphs in Figure~\ref{fig:TransQuotients}.
Table~\ref{tab:AdjMatPlusPermImage} lists their adjacency matrices
as well as the associated permutations of $\{\BasisVector_{1},\BasisVector_{2},-\BasisVector_{1},-\BasisVector_{2}\}$.
For instance, $\AdjacencyMatrix^{straight}$ interchanges $\BasisVector_{1}$
and $-\BasisVector_{1}$ and it maps $\BasisVector_{2}$ and $-\BasisVector_{2}$
to themselves, hence, it gives rise to $(1,3)(2)(4)$. If we label
the midpoints of the sides of $\Square$ as indicated in Figure~\ref{fig:TheSquareS},
then $(1,3)(2)(4)$ can be identified with the reflection $\tau$
of $\Square$ as noted in the last line of Table~\ref{tab:AdjMatPlusPermImage}.
We see that the group generated by $\AdjacencyMatrix^{straight}$
and $\AdjacencyMatrix^{zig-zag}$ is $D_{4}$. Moreover, the associated
pairs $(H_{1},R_{1})$ and $(\OnSecond H_{1},\OnSecond R_{1})$ given
by~(\ref{eq:AssSubGroupsAndRepsDef}) coincide with $(H,R)$ and
$(\OnSecond H,\OnSecond R)$ in~(\ref{eq:SubgroupsHandHhatAndRepsRandRhat}).
In accordance with Theorem~\ref{thm:TransGraphsAreSchreierGraphs},
$\CayleyGraph{D_{4}}{(\tau,\tau\sigma^{3})}/R$ and $\CayleyGraph{D_{4}}{(\tau,\tau\sigma^{3})}/\OnSecond R$
are the graphs we started with.

As we are solely interested in the spectrum of the Laplace operator
on the trivial bundle, we make the following additional assumption.
\begin{defn}
Let $G$ be a finite group generated by elements $(\GeneratorsOfG^{\Colour})_{\Colour=1}^{\NumberOfColours}$
of order $2$. A pair $(H,R)$ consisting of a subgroup $H$ of $G$
and a one-dimensional real representation $R$ of $H$ is called bipartite\emph{
}with respect to $(\GeneratorsOfG^{\Colour})_{\Colour=1}^{\NumberOfColours}$
if the right cosets of $H$ have a system of representatives $(\ElementInG_{i})_{i=1}^{[G:H]}$
such that whenever $i\neq j$ and $H\ElementInG_{i}\GeneratorsOfG^{\Colour}=H\ElementInG_{j}$,
we have $R(\ElementInG_{i}\GeneratorsOfG^{\Colour}\ElementInG_{j}^{-1})=1$.
\end{defn}
Note that the associated pairs in Definition~\ref{def:The_associated_subgroups_and_representations}
are bipartite with respect to the adjacency matrices generating $G$
precisely because off-diagonal entries of adjacency matrices are non-negative.
The following theorem implies Theorem~\ref{thm:Transplantability_is_Equivalent_Inductions}
and reveals the group-theoretic nature of transplantability.
\begin{thm}
\label{thm:TransSchreierGraphs}Let $G$ be a finite group generated
by elements $(\GeneratorsOfG^{\Colour})_{\Colour=1}^{\NumberOfColours}$
of order $2$ and let $(H,R)$ be a pair that is bipartite with respect
to $(\GeneratorsOfG^{\Colour})_{\Colour=1}^{\NumberOfColours}$. Then,
the adjacency matrices of $\CayleyGraph G{(\GeneratorsOfG^{\Colour})_{\Colour=1}^{\NumberOfColours}}/R$
yield a representation of $G$ which is equivalent to $\InducedRep HGR$.
In particular, if $((H_{i},R_{i}))_{i}$ and $((\OnSecond H_{j},\OnSecond R_{j}))_{j}$
are tuples of pairs that are bipartite with respect to $(\GeneratorsOfG^{\Colour})_{\Colour=1}^{\NumberOfColours}$,
then the loop-signed graphs $\bigcup_{i}\CayleyGraph G{(\GeneratorsOfG^{\Colour})_{\Colour=1}^{\NumberOfColours}}/R_{i}$
and $\bigcup_{j}\CayleyGraph G{(\GeneratorsOfG^{\Colour})_{\Colour=1}^{\NumberOfColours}}/\OnSecond R_{j}$
are transplantable if and only if $\bigoplus_{i}\InducedRep{H_{i}}G{R_{i}}\simeq\bigoplus_{j}\InducedRep{\OnSecond H_{j}}G{\OnSecond R_{j}}$.\end{thm}
\begin{proof}
It suffices to prove the first statement. Let $\NumberOfVertices=[G:H]$
and choose a system of representatives $(\ElementInG_{i})_{i=1}^{\NumberOfVertices}$
for the right cosets of $H$ such that whenever $i\neq j$ and $H\ElementInG_{i}\GeneratorsOfG^{\Colour}=H\ElementInG_{j}$,
then $R(\ElementInG_{i}\GeneratorsOfG^{\Colour}\ElementInG_{j}^{-1})=1$.
If we number the vertices of $\CayleyGraph G{(\GeneratorsOfG^{\Colour})_{\Colour=1}^{\NumberOfColours}}/R$
according to $(\ElementInG_{i})_{i=1}^{\NumberOfVertices}$, then
its $\NumberOfVertices\times\NumberOfVertices$ adjacency matrices
$(\AdjacencyMatrix^{\Colour})_{\Colour=1}^{\NumberOfColours}$ read
\[
\AdjacencyMatrix_{ij}^{\Colour}=\begin{cases}
R(\ElementInG_{i}\GeneratorsOfG^{\Colour}\ElementInG_{j}^{-1}) & \textrm{if }H\ElementInG_{i}\GeneratorsOfG^{\Colour}=H\ElementInG_{j}\\
0 & \textrm{otherwise.}
\end{cases}
\]
In the following, we show that the map
\begin{equation}
\RepresentationOfG\colon\GeneratorsOfG^{\Colour_{1}}\GeneratorsOfG^{\Colour_{2}}\cdots\GeneratorsOfG^{\Colour_{l}}\mapsto\AdjacencyMatrix^{\Colour_{1}}\AdjacencyMatrix^{\Colour_{2}}\cdots\AdjacencyMatrix^{\Colour_{l}}\label{eq:Definition_of_Rep_on_Products_of_Generators}
\end{equation}
yields a well-defined homomorphism $\RepresentationOfG\colon G\to GL(\NumberOfVertices,\mathbb{R})$,
that is, a representation. More precisely, we prove by induction that
for any $\ProductOfAdjMat\in G$, definition~(\ref{eq:Definition_of_Rep_on_Products_of_Generators})
leads to 
\begin{equation}
\RepresentationOfG(\ProductOfAdjMat)_{ij}=\begin{cases}
R(\ElementInG_{i}\ProductOfAdjMat\ElementInG_{j}^{-1}) & \textrm{if }H\ElementInG_{i}\ProductOfAdjMat=H\ElementInG_{j}\\
0 & \textrm{otherwise.}
\end{cases}\label{eq:EntriesOfImageOfHom}
\end{equation}
Assume that (\ref{eq:EntriesOfImageOfHom}) holds for any $\ProductOfAdjMat=\GeneratorsOfG^{\Colour_{1}}\GeneratorsOfG^{\Colour_{2}}\cdots\GeneratorsOfG^{\Colour_{l}}$
with $l\leq L$. For arbitrary $\GeneratorsOfG^{\Colour}$, we get
\[
\RepresentationOfG(\ProductOfAdjMat\GeneratorsOfG^{\Colour})_{ij}=\left(\RepresentationOfG(\ProductOfAdjMat)\RepresentationOfG(\GeneratorsOfG^{\Colour})\right)_{ij}=\sum_{k=1}^{\NumberOfVertices}\RepresentationOfG(\ProductOfAdjMat)_{ik}\RepresentationOfG(\GeneratorsOfG^{\Colour})_{kj}.
\]
Assume that $\RepresentationOfG(\ProductOfAdjMat)_{im}$ is the non-vanishing
entry in the $i$th row of $\RepresentationOfG(\ProductOfAdjMat)$,
that is, $H\ElementInG_{i}\ProductOfAdjMat=H\ElementInG_{m}$ and
$\RepresentationOfG(\ProductOfAdjMat)_{im}=R(\ElementInG_{i}\ProductOfAdjMat\ElementInG_{m}^{-1})$.
Then, $\RepresentationOfG(\ProductOfAdjMat\GeneratorsOfG^{\Colour})_{ij}\neq0$
if and only if $\AdjacencyMatrix_{mj}^{\Colour}=\RepresentationOfG(\GeneratorsOfG^{\Colour})_{mj}\neq0$,
that is, if $H\ElementInG_{m}\GeneratorsOfG^{\Colour}=H\ElementInG_{j}$,
and in this case $\RepresentationOfG(\GeneratorsOfG^{\Colour})_{mj}=R(\ElementInG_{m}\GeneratorsOfG^{\Colour}\ElementInG_{j}^{-1})$.
Since $\ElementInG_{i}\ProductOfAdjMat\ElementInG_{m}^{-1}\in H$,
this is equivalent to
\[
\ElementInG_{i}\ProductOfAdjMat\ElementInG_{m}^{-1}\ElementInG_{m}\GeneratorsOfG^{\Colour}\ElementInG_{j}^{-1}=\ElementInG_{i}\ProductOfAdjMat\GeneratorsOfG^{\Colour}\ElementInG_{j}^{-1}\in H,
\]
and in this case 
\[
\RepresentationOfG(\ProductOfAdjMat\GeneratorsOfG^{\Colour})_{ij}=\RepresentationOfG(\ProductOfAdjMat)_{im}\RepresentationOfG(\GeneratorsOfG^{\Colour})_{mj}=R(\ElementInG_{i}\ProductOfAdjMat\ElementInG_{m}^{-1})R(\ElementInG_{m}\GeneratorsOfG^{\Colour}\ElementInG_{j}^{-1})=R(\ElementInG_{i}\ProductOfAdjMat\GeneratorsOfG^{\Colour}\ElementInG_{j}^{-1}),
\]
where we used that $R$ is a homomorphism. Hence, (\ref{eq:EntriesOfImageOfHom})
follows by induction on $L$.

It remains to show that $\RepresentationOfG$ and $\InducedRep HGR$
have equal characters. For arbitrary $\ProductOfAdjMat\in G$, let
\[
B^{\pm}(\ProductOfAdjMat)=\{i\in\{1,2,\ldots,\NumberOfVertices\}\,|\,\ElementInG_{i}\ProductOfAdjMat\ElementInG_{i}^{-1}\in H\textrm{ and }R(\ElementInG_{i}\ProductOfAdjMat\ElementInG_{i}^{-1})=\pm1\}.
\]
According to (\ref{eq:EntriesOfImageOfHom}), the character $\chi_{\RepresentationOfG}$
of $\RepresentationOfG$ satisfies
\[
\chi_{\RepresentationOfG}(\ProductOfAdjMat)=\Trace(\RepresentationOfG(\ProductOfAdjMat))=\left|B^{+}(\ProductOfAdjMat)\right|-\left|B^{-}(\ProductOfAdjMat)\right|.
\]
The character of $\InducedRep HGR$ reads~\foreignlanguage{british}{\cite[Chapter 3, Theorem 12]{Serre1977}}
\[
\chi_{\InducedRep HGR}(\ProductOfAdjMat)=\frac{1}{\left|H\right|}\sum_{\{\ElementInG\in G\,|\,\ElementInG\ProductOfAdjMat\ElementInG^{-1}\in H\}}R(\ElementInG\ProductOfAdjMat\ElementInG^{-1}).
\]
Since any $\ElementInG\in G$ can be uniquely written as $\ElementInG=h\ElementInG_{i}$
for some $h\in H$ and $i\in\{1,2,\ldots,\NumberOfVertices\}$, the
condition $\ElementInG\ProductOfAdjMat\ElementInG^{-1}\in H$ is equivalent
to $g\in H\ElementInG_{i}$ for some $i\in B^{+}(\ProductOfAdjMat)\cup B^{-}(\ProductOfAdjMat)$.
Moreover, if $\ElementInG=h\ElementInG_{i}$ with $h\in H$ and $i\in B^{\pm}(\ProductOfAdjMat)$,
then 
\[
R(\ElementInG\ProductOfAdjMat\ElementInG^{-1})=R(h\ElementInG_{i}\ProductOfAdjMat\ElementInG_{i}^{-1}h^{-1})=R(\ElementInG_{i}\ProductOfAdjMat\ElementInG_{i}^{-1})=\pm1.
\]
Hence, 
\begin{eqnarray*}
\chi_{\InducedRep HGR}(\ProductOfAdjMat) & = & \frac{1}{\left|H\right|}\left(\sum_{i\in B^{+}(\ProductOfAdjMat)}\sum_{\ElementInG\in H\ElementInG_{i}}1+\sum_{i\in B^{-}(\ProductOfAdjMat)}\sum_{\ElementInG\in H\ElementInG_{i}}-1\right)\\
 & = & \left|B^{+}(\ProductOfAdjMat)\right|-\left|B^{-}(\ProductOfAdjMat)\right|=\chi_{\RepresentationOfG}(\ProductOfAdjMat),
\end{eqnarray*}
which completes the proof.
\end{proof}

\section{Generating methods\label{sec:GeneratingTools}}

The Transplantation Theorem allows to derive methods with which new
transplantable pairs can be generated from given ones. In the following,
let $\Graph=\bigcup_{i}\Gamma_{i}$ and $\OnSecond{\Graph}=\bigcup_{j}\OnSecond{\Gamma}_{j}$
be transplantable loop-signed graphs with edge colors $1,2,\ldots,\NumberOfColours$,
vertices $1,2,\ldots,\NumberOfVertices$ and connected components
$(\Gamma_{i})_{i}$ and $(\OnSecond{\Gamma}_{j})_{j}$, respectively.
We let $(\AdjacencyMatrix^{\Colour})_{\Colour=1}^{\NumberOfColours}$
and $(\OnSecond{\AdjacencyMatrix}^{\Colour})_{\Colour=1}^{\NumberOfColours}$
denote their $\NumberOfVertices\times\NumberOfVertices$ adjacency
matrices and choose $T\in GL(\NumberOfVertices,\mathbb{R})$ with
\begin{equation}
\OnSecond{\AdjacencyMatrix}^{\Colour}=T\AdjacencyMatrix^{\Colour}T^{-1}\qquad\text{for }\Colour=1,2,\ldots,\NumberOfColours.\label{eq:TransplantationConditionInSectionFour}
\end{equation}

\subsection{Partial dualization\label{sub:Dualisation}}

Dualization refers to the process of changing loop signs of $\Graph$
and $\OnSecond{\Graph}$ without affecting their transplantability.
As a trivial example, if both $\AdjacencyMatrix^{\Colour}$ and $\OnSecond{\AdjacencyMatrix}^{\Colour}$
are diagonal matrices for some $\Colour$, then we can replace them
by $-\AdjacencyMatrix^{\Colour}$ and $-\OnSecond{\AdjacencyMatrix}^{\Colour}$
to obtain a new transplantable pair. For any choice of building block,
the two pairs of graphs would give rise to the same tiled manifolds
but with opposite boundary conditions on all outer faces corresponding
to the edge color $\Colour$.

We generalize this idea. For arbitrary $S\subseteq\{1,2,\ldots,\NumberOfColours\}$,
let $(\Gamma_{S,k})_{k}$ and $(\OnSecond{\Gamma}_{S,l})_{l}$ be
the connected components of the graphs obtained by removing all $S$-colored
edges from $\Graph$ and $\OnSecond{\Graph}$, respectively. Assume
we can assign each of $(\Gamma_{S,k})_{k}$ to either $+1$ or $-1$
in such a way that whenever two of them were connected by an $S$-colored
link in $\Graph$, then they are assigned to opposite numbers, similarly
for $(\OnSecond{\Gamma}_{S,l})_{l}$. We encode these partitionings
of the vertices of $\Graph$ and $\OnSecond{\Graph}$ by diagonal
matrices of the form 
\begin{equation}
P=\mathrm{diag}(\pm1,\pm1,\ldots,\pm1)\quad\textrm{ and }\quad\OnSecond P=\mathrm{diag}(\pm1,\pm1,\ldots,\pm1).\label{eq:OrientationMats}
\end{equation}
We show that the graph obtained by swapping the signs of all $S$-colored
loops of $\Graph$ has adjacency matrices $-P\AdjacencyMatrix^{\Colour}P$
for all $\Colour\in S$ and $P\AdjacencyMatrix^{\Colour}P$ for all
$\Colour\notin S$, similarly for $\OnSecond{\Graph}$. Since $\AdjacencyMatrix^{\Colour}$
and $\pm P\AdjacencyMatrix^{\Colour}P$ have the same vanishing entries,
it suffices to consider the non-vanishing ones. If $\AdjacencyMatrix_{ij}^{\Colour}=1$
with $i\neq j$, then $P_{ii}P_{jj}=\mp1$ depending on whether $\Colour\in S$
or $\Colour\notin S$, which yields 
\[
(\mp P\AdjacencyMatrix^{\Colour}P)_{ij}=\mp P_{ii}A_{ij}^{\Colour}P_{jj}=\AdjacencyMatrix_{ij}^{\Colour}.
\]
On the other hand, if $\AdjacencyMatrix_{ii}^{\Colour}\neq0$, then
\[
(\mp P\AdjacencyMatrix^{\Colour}P)_{ii}=\mp P_{ii}\AdjacencyMatrix_{ii}^{\Colour}P_{ii}=\mp\AdjacencyMatrix_{ii}^{\Colour}.
\]
Since $P$ and $\OnSecond P$ are self-inverse, the transplantation
matrix $\OnSecond PTP$ satisfies
\[
(\mp\OnSecond P\OnSecond{\AdjacencyMatrix}^{\Colour}\OnSecond P)(\OnSecond PTP)=\mp\OnSecond P\OnSecond{\AdjacencyMatrix}^{\Colour}TP=\mp\OnSecond PT\AdjacencyMatrix^{\Colour}P=(\OnSecond PTP)(\mp P\AdjacencyMatrix^{\Colour}P).
\]
Hence, swapping the signs of all $S$-colored loops of $\Graph$ and
$\OnSecond{\Graph}$ yields transplantable graphs. Note that we can
choose $S=\{1,2,\ldots,\NumberOfColours\}$ if and only if the loopless
versions of $\Graph$ and $\OnSecond{\Graph}$ are $2$-colorable
which is equivalent to being bipartite.
\begin{defn}
\label{def:DualPair}Each pair of transplantable loop-signed graphs
with bipartite loopless versions has a transplantable dual pair obtained
by swapping all loop signs. A pair is called self-dual if it is isomorphic
to its dual pair.
\end{defn}
Note that loop-signed graphs with bipartite loopless versions give
rise to orientable tiled manifolds if the underlying building block
$B$ is orientable. More precisely, if $\omega$ is a non-vanishing
volume form on $B$ and if $P=\mathrm{diag}(\pm1,\pm1,\ldots,\pm1)$
is as in~(\ref{eq:OrientationMats}), then one obtains a non-vanishing
volume form on the tiled manifold by taking $P_{ii}\omega$ on its
$i$th block. The different signs on neighboring blocks compensate
for the reflection coming from the gluing process.

\subsection{Braiding}

Braiding means replacing the adjacency matrices $\AdjacencyMatrix^{\Colour}$
and $\OnSecond{\AdjacencyMatrix}^{\Colour}$ of some edge color $\Colour$
by conjugates of the form $\AdjacencyMatrix^{\Colour'}\AdjacencyMatrix^{\Colour}\AdjacencyMatrix^{\Colour'}$
and $\OnSecond{\AdjacencyMatrix}^{\Colour'}\OnSecond{\AdjacencyMatrix}^{\Colour}\OnSecond{\AdjacencyMatrix}^{\Colour'}$.
If any of these matrices has negative off-diagonal entries, then they
still describe isospectral manifolds but with respect to the Laplacian
on a possibly non-trivial bundle. If the underlying graphs have bipartite
loopless versions, then one can remove the negative off-diagonal entries
by conjugating with matrices of the form~(\ref{eq:OrientationMats}).
In this way, the $32$ pairs of transplantable graphs with $3$ edge
colors and $7$ vertices in Figure~\ref{fig:Appendix-7-vertices-per-graph}
in the appendix fall into $8$ classes, called quilts~\cite{ConwayHsu1995}.

\subsection{Copying, adding and omission of edge colors and pairs\label{sub:CopyingOfSegments}\label{sub:OmiAndAddOfTransPairs}}

For $\Colour'\in\{1,2,\ldots,\NumberOfColours\}$, the loop-signed
graphs with edge colors $1,2,\ldots,\NumberOfColours+1$ and adjacency
matrices $(\AdjacencyMatrix^{\Colour})_{\Colour=1}^{\NumberOfColours+1}$
and $(\OnSecond{\AdjacencyMatrix}^{\Colour})_{\Colour=1}^{\NumberOfColours+1}$,
where $\AdjacencyMatrix^{\NumberOfColours+1}=\AdjacencyMatrix^{\Colour'}$
and $\OnSecond{\AdjacencyMatrix}^{\NumberOfColours+1}=\OnSecond{\AdjacencyMatrix}^{\Colour'}$,
trivially satisfy~(\ref{eq:TransplantationConditionInSectionFour}).
Alternatively, we can add an edge color by setting $\AdjacencyMatrix^{\NumberOfColours+1}=\OnSecond{\AdjacencyMatrix}^{\NumberOfColours+1}=\pm I_{\NumberOfVertices}$.
On the level of tiled manifolds, this can sometimes be interpreted
as regarding a component of a reflecting face or parts of $\partial_{D}B$
or $\partial_{N}B$ in~(\ref{eq:BoundaryDecompositionOfBuildingBlock})
as a new reflecting face, respectively. On the other hand, one can
trivially omit an edge color which corresponds to undoing the respective
gluings and imposing the same type of boundary conditions on all resulting
outer reflecting faces as indicated in Figure~\ref{fig:Omission}.

On the analogy of adding edge colors, we can add transplantable components
to $\Graph$ and $\OnSecond{\Graph}$ without affecting their transplantability.
According to Theorem~\ref{thm:TransGraphsAreSchreierGraphs} and
Theorem~\ref{thm:TransSchreierGraphs}, the converse holds, that
is, if the components $\Graph_{k}$ and $\OnSecond{\Graph}_{l}$ of
$\Graph$ and $\OnSecond{\Graph}$ are transplantable, then so are
$\Graph\backslash\Graph_{k}=\bigcup_{i\neq k}\Gamma_{i}$ and $\OnSecond{\Graph}\backslash\OnSecond{\Graph}_{l}=\bigcup_{j\neq l}\OnSecond{\Gamma}_{j}$.
Note that the same statement holds for isospectrality, that is, manifolds
remain isospectral if one removes isospectral components. As an example,
we can omit the two identical triangles in Figure~\ref{fig:Omission}.
The resulting pair was discovered by Band~et~al.~\cite{BandParzanchevskiBen-Shach2009}.
Their search was motivated by Chapman\textquoteright{}s two-piece
band~\cite{Chapman1995} with pure boundary conditions, which can
be obtained similarly.

\begin{figure}
\noindent \begin{centering}
\psfragBodyNew\hfill{}\includegraphics[scale=0.42]{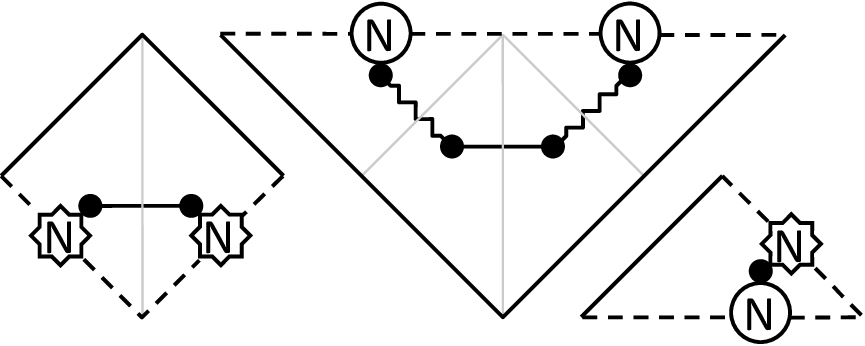}\hspace*{\fill}\includegraphics[scale=0.42]{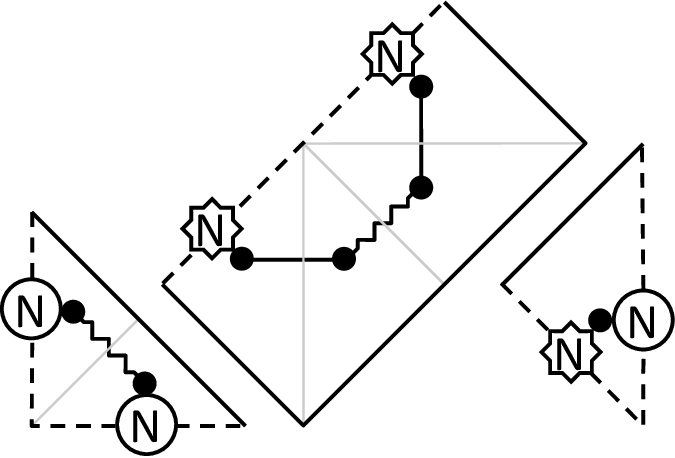}\hfill{}
\par\end{centering}

\caption{Transplantable pair obtained from the pair underlying the Gordon-Webb-Wolpert
drums with Neumann boundary conditions by omission of an edge color.\label{fig:Omission}}
\end{figure}

\subsection{Crossings\label{sub:CrossingofTransplantable}}

In the following, we use tensor products of linear maps. Their basis
dependent matrix representations are known as Kronecker products.
\begin{defn}
The Kronecker product of an $m\times m$ matrix $A$ and an $n\times n$
matrix $B$ is the $mn\times mn$ block matrix
\[
A\Cross B=\left(\begin{array}{cccc}
A_{11}B & A_{12}B & \ldots & A_{1m}B\\
A_{21}B & A_{22}B & \ldots & A_{2m}B\\
\vdots & \vdots & \ddots & \vdots\\
A_{m1}B & A_{m2}B & \cdots & A_{mm}B
\end{array}\right).
\]
\end{defn}
\begin{lem}
\label{lem:Commutation_Relation_of_Kronecker_Product}For each pair
$(m,n)$ of positive integers, there is an $mn\times mn$ permutation
matrix $P^{(m,n)}$ such that for every $m\times m$ matrix $A$ and
every $n\times n$ matrix $B$
\[
\left(A\Cross B\right)P^{(m,n)}=P^{(m,n)}\left(B\Cross A\right).
\]

\end{lem}
Lemma~\ref{lem:Commutation_Relation_of_Kronecker_Product} is an
easy exercise with bases. Note that Kronecker products of symmetric
permutation matrices are symmetric permutation matrices, which justifies
the following definition.
\begin{defn}
Let $\Graph_{1}$ and $\Graph_{2}$ be loop-signed graphs without
Dirichlet loops having $\NumberOfVertices_{1}$, respectively $\NumberOfVertices_{2}$,
vertices and adjacency matrices $(\AdjacencyMatrix_{1}^{\Colour})_{\Colour=1}^{\NumberOfColours_{1}}$,
respectively $(\AdjacencyMatrix_{2}^{\Colour})_{\Colour=1}^{\NumberOfColours_{2}}$.
The crossing\emph{ }$\Crossing{\Graph_{1}}{\Graph_{2}}$ is the loop-signed
graph with $\NumberOfColours_{1}\NumberOfColours_{2}$ edge colors,
$\NumberOfVertices_{1}\NumberOfVertices_{2}$ vertices and adjacency
matrices
\[
\AdjacencyMatrix_{\Cross}^{[\Colour_{1},\Colour_{2}]}=\AdjacencyMatrix_{1}^{\Colour_{1}}\Cross\AdjacencyMatrix_{2}^{\Colour_{2}},
\]
where $[\Colour_{1},\Colour_{2}]=\Colour_{1}+(\Colour_{2}-1)\NumberOfColours_{1}$
for $\Colour_{1}=1,2,\ldots,\NumberOfColours_{1}$ and $\Colour_{2}=1,2,\ldots,\NumberOfColours_{2}$.
\end{defn}

Lemma~\ref{lem:Commutation_Relation_of_Kronecker_Product} implies
that for $\Graph_{1}$ and $\Graph_{2}$ as above, $\Crossing{\Graph_{1}}{\Graph_{2}}$
and $\Crossing{\Graph_{2}}{\Graph_{1}}$ only differ by a renumbering
of their vertices and edge colors. Let $((\Graph_{i},\OnSecond{\Graph}_{i}))_{i=1,2}$
be two pairs of transplantable loop-signed graphs without Dirichlet
loops given by adjacency matrices $(\AdjacencyMatrix_{i}^{\Colour})_{\Colour=1}^{\NumberOfColours_{i}}$
and $(\OnSecond{\AdjacencyMatrix}_{i}^{\Colour})_{\Colour=1}^{\NumberOfColours_{i}}$,
respectively. Let $(T_{i})_{i=1,2}$ be transplantation matrices such
that
\[
\OnSecond{\AdjacencyMatrix}_{i}^{\Colour}=T_{i}\AdjacencyMatrix_{i}^{\Colour}T_{i}^{-1}\qquad\text{for }i=1,2\text{ and }\Colour=1,2,\ldots,\NumberOfColours_{i}.
\]
Using the invertible matrix $T_{1}\Cross T_{2}$, we obtain that for
$\Colour_{1}=1,2,\ldots,\NumberOfColours_{1}$ and $\Colour_{2}=1,2,\ldots,\NumberOfColours_{2}$
\[
\begin{array}{rclcl}
\OnSecond{\AdjacencyMatrix}_{\Cross}^{[\Colour_{1},\Colour_{2}]}(T_{1}\Cross T_{2}) & = & \OnSecond{\AdjacencyMatrix}_{1}^{\Colour_{1}}T_{1}\Cross\OnSecond{\AdjacencyMatrix}_{2}^{\Colour_{2}}T_{2}\\
 & = & T_{1}\AdjacencyMatrix_{1}^{\Colour_{1}}\Cross T_{2}\AdjacencyMatrix_{2}^{\Colour_{2}} & = & (T_{1}\Cross T_{2})\AdjacencyMatrix_{\Cross}^{[\Colour_{1},\Colour_{2}]},
\end{array}
\]
which shows that $\Crossing{\Graph_{1}}{\Graph_{2}}$ and $\Crossing{\OnSecond{\Graph}_{1}}{\OnSecond{\Graph}_{2}}$
are transplantable.

\subsection{Substitutions\label{sub:TheSubstitutionMethod}}

Levitin~et~al.~\cite{LevitinParnovskiPolterovich2006} discovered
the strikingly simple pair of transplantable domains shown in Figure~\ref{fig:SubstitutionInTermsOfTiledDomains},
a broken square that sounds like a broken triangle. The matrix
\[
T=\left(\begin{array}{cc}
-1 & 1\\
1 & 1
\end{array}\right)
\]
gives an intertwining transplantation. We subdivide the triangles
into smaller blocks as indicated in Figure~\ref{fig:SubstitutionInTermsOfTiledDomains},
where we added notches for the sake of clarity. The transplantation
respects this subdivision in the sense that if $\varphi$ solves the
Zaremba problem on the punctured square and has restrictions $(\varphi_{i}){}_{i=1}^{4}$,
then the restrictions $\OnSecond{\varphi}_{1}$ and $\OnSecond{\varphi}_{2}$
of its transform $\OnSecond{\varphi}=T(\varphi)$ on the notched triangle
depend on $\varphi_{1}$ and $\varphi_{2}$ only, similarly for $\OnSecond{\varphi}_{3}$
and $\OnSecond{\varphi}_{4}$. In other words, the subdivided domains
can be regarded as a pair of transplantable tiled manifolds each of
which consists of $4$ blocks. With respect to the subdivision, the
transplantation reads 
\begin{equation}
\left(\begin{array}{cccc}
-1 & 1 & 0 & 0\\
1 & 1 & 0 & 0\\
0 & 0 & -1 & 1\\
0 & 0 & 1 & 1
\end{array}\right)=\left(\begin{array}{cc}
1 & 0\\
0 & 1
\end{array}\right)\Cross T.\label{eq:Transplantation_Matrix_Four_by_Four}
\end{equation}

With respect to the graph representation, subdividing the blocks corresponds
to substituting a whole graph into each of the vertices of the associated
loop-signed graphs as indicated in Figure~\ref{fig:SubstitutionInTermsOfTiledDomains}.
We develop this method in detail. Let $\Graph$ and $\GraphOfSubstituent$
be loop-signed graphs having~$\NumberOfColours$, respectively $\NumberOfEdgeColoursOfSubstituent$,
edge colors and $\NumberOfVertices$, respectively $\NumberOfVerticesOfSubstituent$,
vertices. Let $(\AdjacencyMatrix^{\Colour})_{\Colour=1}^{\NumberOfColours}$
and $(\AdjacencyMatrixOfSubstituent{\ColourOfSubstituent})_{\ColourOfSubstituent=1}^{\NumberOfEdgeColoursOfSubstituent}$
denote their adjacency matrices. In order to assign Neumann loops
of $\GraphOfSubstituent$ to edge colors of $\Graph$, we let $I^{\ColourOfSubstituent}$
be a set of vertices of $\GraphOfSubstituent$ that have a $\ColourOfSubstituent$-colored
Neumann loop, that is,
\[
\LoopIndices^{\ColourOfSubstituent}\subseteq\left\{ i\,|\,(\AdjacencyMatrixOfSubstituent{\ColourOfSubstituent})_{ii}=1\right\} .
\]
We divide $\LoopIndices^{\ColourOfSubstituent}$ into disjoint vertex
subsets $(\LoopIndices_{\Colour}^{\ColourOfSubstituent})_{\Colour=1}^{\NumberOfColours}$
each of which represents the Neumann loops of $\GraphOfSubstituent$
that are assigned to the corresponding edge color of $\Graph$. Let
$(\AdjacencyMatrixOfSubstituent{\ColourOfSubstituent})_{i}$ denote
the $\NumberOfVerticesOfSubstituent\times\NumberOfVerticesOfSubstituent$
matrix which has a unity entry at position $(i,i)$ and is zero at
all other entries. The matrix 
\begin{equation}
(\AdjacencyMatrixOfSubstituent{\ColourOfSubstituent})_{*}=\AdjacencyMatrixOfSubstituent{\ColourOfSubstituent}-\sum_{i\in\LoopIndices^{\ColourOfSubstituent}}(\AdjacencyMatrixOfSubstituent{\ColourOfSubstituent})_{i}\label{eq:Invariant_part_of_substituent}
\end{equation}
describes the $\ColourOfSubstituent$-colored part of $\GraphOfSubstituent$
that reappears at each vertex of $\Graph$.

\begin{figure}
\noindent \begin{centering}
\psfragBody\psfrag{A}{\hspace{-1mm}\raisebox{0.5mm}{\LARGE{$\Rightarrow$}}} \psfrag{1}{$1$}
\psfrag{2}{$2$}
\psfrag{3}{$3$}
\psfrag{4}{$4$}
\psfrag{5}{\hspace{3mm}\raisebox{0mm}{$\Graph$}}
\psfrag{6}{}%
\begin{tabular}{c>{\centering}p{14mm}rcl}
 &  & $\GraphOfSubstituent$\hspace{2mm} & \selectlanguage{american}%
$\SubstitutedGraph{\Graph}{\GraphOfSubstituent}$\selectlanguage{english}%
 & \multirow{2}{*}{\includegraphics[scale=0.42]{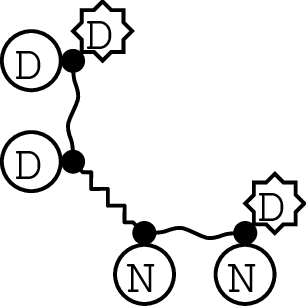}}\tabularnewline[-1mm]
\raisebox{0.8mm}{\includegraphics[scale=0.42]{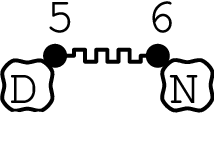}} & \includegraphics[scale=0.42]{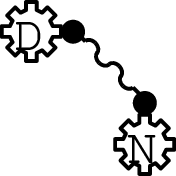} & \includegraphics[scale=0.42]{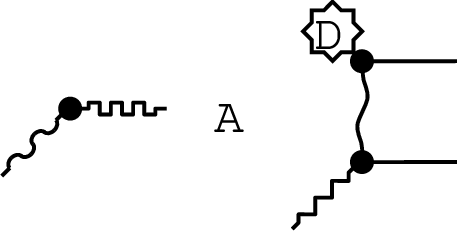} & \psfrag{2}{$3$} \psfrag{3}{$2$}\includegraphics[scale=0.42]{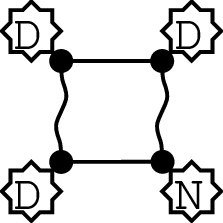} & \tabularnewline
\noalign{\vskip\doublerulesep}
\psfrag{2}{$2$} \psfrag{3}{$3$}\includegraphics[scale=0.45]{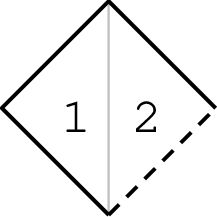} & \includegraphics[scale=0.45]{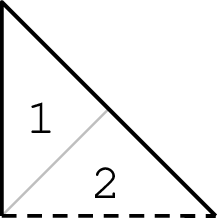} & \includegraphics[scale=0.45]{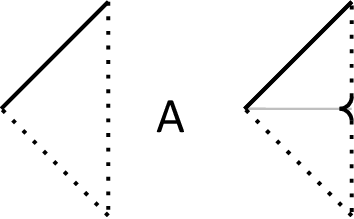}\hspace{3.5mm} & \psfrag{2}{$3$} \psfrag{3}{$2$}\includegraphics[scale=0.45]{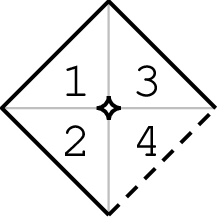} & \psfrag{2}{$3$} \psfrag{3}{$2$}\includegraphics[scale=0.45]{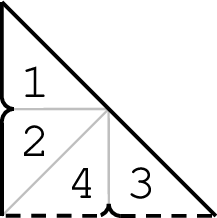}\tabularnewline
\noalign{\vskip\doublerulesep}
\end{tabular}
\par\end{centering}

\noindent \centering{}\caption{Substitution in terms of graphs and manifolds.\label{fig:SubstitutionInTermsOfTiledDomains}}
\end{figure}

\begin{defn}
\label{def:TheAdjacencyMatrixOfSubstitutedGraph}The substituted graph
\emph{$\SubstitutedGraph{\Graph}{\GraphOfSubstituent}$} is the loop-signed
graph with edge colors $\ColourOfSubstituent=1,2,\ldots,\NumberOfEdgeColoursOfSubstituent$
and vertices $1,2,\ldots,\NumberOfVertices\NumberOfVerticesOfSubstituent$
given by the adjacency matrices 
\begin{equation}
\AdjacencyMatrix_{\Substitution}^{\ColourOfSubstituent}=(\AdjacencyMatrixOfSubstituent{\ColourOfSubstituent})_{*}\Cross I_{\NumberOfVertices}+\sum_{\Colour}\sum_{i\in\LoopIndices_{\Colour}^{\ColourOfSubstituent}}(\AdjacencyMatrixOfSubstituent{\ColourOfSubstituent})_{i}\Cross\AdjacencyMatrix^{\Colour},\label{eq:AdjacencyMatricOfSubstitutedGraph}
\end{equation}
where $I_{\NumberOfVertices}$ denotes the $\NumberOfVertices\times\NumberOfVertices$
identity matrix.

\end{defn}
We briefly justify why~(\ref{eq:AdjacencyMatricOfSubstitutedGraph})
yields adjacency matrices. One easily sees that all summands are symmetric
and that none has negative off-diagonal entries. Moreover, one can
view $\AdjacencyMatrix_{\Substitution}^{\ColourOfSubstituent}$ as
a sum of $\NumberOfVerticesOfSubstituent\times\NumberOfVerticesOfSubstituent$
block matrices where the $\NumberOfVertices\times\NumberOfVertices$
blocks are given by multiples of $I_{\NumberOfVertices}$ and $\AdjacencyMatrix^{\Colour}$.
Since $\AdjacencyMatrixOfSubstituent{\ColourOfSubstituent}$, $I_{\NumberOfVertices}$
and $(\AdjacencyMatrix^{\Colour})_{\Colour=1}^{\NumberOfColours}$
are signed permutation matrices, one can use (\ref{eq:Invariant_part_of_substituent})
to show that each row and each column of $\AdjacencyMatrix_{\Substitution}^{\ColourOfSubstituent}$
contains exactly one non-vanishing entry.

The graphs in Figure~\ref{fig:SubstitutionInTermsOfTiledDomains}
exemplify the substitution method. The outward-pointing edges of $\GraphOfSubstituent$
represent its Neumann loops that are assigned to edge colors of $\Graph$,
namely, 
\[
\LoopIndices^{straight}=\LoopIndices_{ridged}^{straight}=\{1,2\}\quad\LoopIndices^{zig-zag}=\LoopIndices_{curly}^{zig-zag}=\{2\}\quad\LoopIndices^{wavy}=\varnothing.
\]
For example, the \emph{ridged }link of $\Graph$ is replaced by two
\emph{straight }links and each \emph{curly} loop of $\Graph$ is replaced
by a \emph{zig-zag} loop with the same sign. We give the decompositions
of two adjacency matrices of \emph{$\SubstitutedGraph{\Graph}{\GraphOfSubstituent}$}:\setlength{\arraycolsep}{2pt}
\[
\begin{array}{ccccccccc}
\AdjacencyMatrix_{\Substitution}^{straight} & = & (\AdjacencyMatrixOfSubstituent{straight})_{1} & \Cross & \AdjacencyMatrix^{ridged} & + & (\AdjacencyMatrixOfSubstituent{straight})_{2} & \Cross & \AdjacencyMatrix^{ridged}\\
\left(\begin{array}{cccc}
0 & 1 & 0 & 0\\
1 & 0 & 0 & 0\\
0 & 0 & 0 & 1\\
0 & 0 & 1 & 0
\end{array}\right) & = & \left(\begin{array}{cc}
1 & 0\\
0 & 0
\end{array}\right) & \Cross & \left(\begin{array}{cc}
0 & 1\\
1 & 0
\end{array}\right) & + & \left(\begin{array}{cc}
0 & 0\\
0 & 1
\end{array}\right) & \Cross & \left(\begin{array}{cc}
0 & 1\\
1 & 0
\end{array}\right)
\end{array}
\]
\[
\begin{array}{ccccccccc}
\AdjacencyMatrix_{\Substitution}^{zig-zag} & = & (\AdjacencyMatrixOfSubstituent{zig-zag})_{*} & \Cross & I_{2} & + & (\AdjacencyMatrixOfSubstituent{zig-zag})_{2} & \Cross & \AdjacencyMatrix^{curly}\\
\left(\begin{array}{cccc}
-1 & 0 & 0 & 0\\
0 & -1 & 0 & 0\\
0 & 0 & -1 & 0\\
0 & 0 & 0 & 1
\end{array}\right) & = & \left(\begin{array}{cc}
-1 & 0\\
0 & 0
\end{array}\right) & \Cross & \left(\begin{array}{cc}
1 & 0\\
0 & 1
\end{array}\right) & + & \left(\begin{array}{cc}
0 & 0\\
0 & 1
\end{array}\right) & \Cross & \left(\begin{array}{cc}
-1 & 0\\
0 & 1
\end{array}\right)
\end{array}
\]

\begin{defn}
Let $G\leq\mathrm{Sym}(X)$ and $H\leq\mathrm{Sym}(Y)$ be permutation
groups on finite sets $X$ and $Y$, respectively. The external wreath
product $G\wr_{Y}H$ is the semi-direct product $G^{Y}\rtimes H$,
where $h\in H$ acts on $f\in G^{Y}$ such that $hf(y)=f(h^{-1}y)$
for $y\in Y$. The imprimitive wreath product action of $G\wr_{Y}H$
on $X\times Y$ is given by $(f,h)(x,y)=(f(hy)x,hy)$.
\end{defn}
\setlength{\arraycolsep}{\myArraycolsep}

\begin{SubstitutionTheorem}\label{TheSubTheorem}Let $\Graph$, $\GraphOfSubstituent$
and $(\LoopIndices_{\Colour}^{\ColourOfSubstituent})_{\Colour=1}^{\NumberOfColours}$
be as above.
\begin{enumerate}
\item If $\Graph$ and $\GraphOfSubstituent$ are connected and for each
edge color $\Colour$ of $\Graph$ there is a Neumann loop of $\GraphOfSubstituent$
that is assigned to it, that is, $\sum_{\ColourOfSubstituent=1}^{\NumberOfEdgeColoursOfSubstituent}\left|\LoopIndices_{\Colour}^{\ColourOfSubstituent}\right|>0$,
then $\SubstitutedGraph{\Graph}{\GraphOfSubstituent}$ is connected.
\item If $\Graph$ and $\GraphOfSubstituent$ are treelike and for each
edge color $\Colour$ of $\Graph$ there is exactly one Neumann loop
of $\GraphOfSubstituent$ that is assigned to it, that is, $\sum_{\ColourOfSubstituent=1}^{\NumberOfEdgeColoursOfSubstituent}\left|\LoopIndices_{\Colour}^{\ColourOfSubstituent}\right|=1$,
then $\SubstitutedGraph{\Graph}{\GraphOfSubstituent}$ is treelike.
\item If $\Graph$ and $\OnSecond{\Graph}$ are transplantable loop-signed
graphs, then $\SubstitutedGraph{\Graph}{\GraphOfSubstituent}$ and
$\SubstitutedGraph{\OnSecond{\Graph}}{\GraphOfSubstituent}$ are transplantable.
More precisely, if $T$ is an intertwining transplantation matrix
for $\Graph$ and $\OnSecond{\Graph}$, then $I_{\NumberOfVerticesOfSubstituent}\Cross T$
is one for $\SubstitutedGraph{\Graph}{\GraphOfSubstituent}$ and $\SubstitutedGraph{\OnSecond{\Graph}}{\GraphOfSubstituent}$.
\item Assume $\GraphOfSubstituent$ has no Dirichlet loops. Let $G$, $\GroupOfSubstituent$
and $\SubstitutedGraph G{\GroupOfSubstituent}$ be the groups generated
by the adjacency matrices of $\Graph$, $\GraphOfSubstituent$ and
$\SubstitutedGraph{\Graph}{\GraphOfSubstituent}$, regarded as permutation
groups on $X=\{\pm\BasisVector_{1}^{V},\pm\BasisVector_{2}^{V},\ldots,\pm\BasisVector_{\NumberOfVertices}^{V}\}$,
$Y=\{\BasisVector_{1}^{\NumberOfVerticesOfSubstituent},\BasisVector_{2}^{\NumberOfVerticesOfSubstituent},\ldots,\BasisVector_{\NumberOfVerticesOfSubstituent}^{\NumberOfVerticesOfSubstituent}\}$
and $Z=\{\pm\BasisVector_{1}^{\NumberOfVertices\NumberOfVerticesOfSubstituent},\pm\BasisVector_{2}^{\NumberOfVertices\NumberOfVerticesOfSubstituent},\ldots,\pm\BasisVector_{\NumberOfVertices\NumberOfVerticesOfSubstituent}^{\NumberOfVertices\NumberOfVerticesOfSubstituent}\}$,
respectively, where $\BasisVector_{i}^{n}$ denotes the $i$th standard
basis vector of $\mathbb{R}^{n}$. Then, $\SubstitutedGraph G{\GroupOfSubstituent}$
is isomorphic to a subgroup of $G\wr_{Y}\GroupOfSubstituent$ and
the isomorphism identifies the action of $\SubstitutedGraph G{\GroupOfSubstituent}$
on $Z$ with the imprimitive wreath product action of $G\wr_{Y}\GroupOfSubstituent$
on $X\times Y$.
\end{enumerate}
\end{SubstitutionTheorem}
\begin{proof}
1) We show that any two vertices of $\SubstitutedGraph{\Graph}{\GraphOfSubstituent}$
are connected by a path of links. Again, we view $\AdjacencyMatrix_{\Substitution}^{\ColourOfSubstituent}$
as a sum of $\NumberOfVerticesOfSubstituent\times\NumberOfVerticesOfSubstituent$
block matrices with $\NumberOfVertices\times\NumberOfVertices$ blocks
given by multiples of $I_{\NumberOfVertices}$ and $\AdjacencyMatrix^{\Colour}$.
If we number the vertices of $\AdjacencyMatrix_{\Substitution}^{\ColourOfSubstituent}$
by $(p-1)\NumberOfVertices+q$ where $1\leq p\leq\NumberOfVerticesOfSubstituent$
and $1\leq q\leq\NumberOfVertices$, then links coming from nonzero
entries of the summand $(\AdjacencyMatrixOfSubstituent{\ColourOfSubstituent})_{*}\Cross I_{\NumberOfVertices}$,
respectively $(\AdjacencyMatrixOfSubstituent{\ColourOfSubstituent})_{i}\Cross\AdjacencyMatrix^{\Colour}$,
connect vertices with equal values of $q$, respectively $p$. In
order to connect the vertices $(p_{1}-1)\NumberOfVertices+q_{1}$
and $(p_{2}-1)\NumberOfVertices+q_{2}$ in $\SubstitutedGraph{\Graph}{\GraphOfSubstituent}$,
one first chooses a sequence of edge colors $\Colour_{1},\Colour_{2},\ldots,\Colour_{L}$
of $\Graph$ describing a self-avoiding connecting path $q_{1}=i_{0},i_{1},\ldots,i_{L}=q_{2}$.
By assumption, the $\Colour_{1}$-connectivity is established by some
edge color of $\GraphOfSubstituent$, that is, there is $\ColourOfSubstituent_{1}\in\{1,2,\ldots,\NumberOfEdgeColoursOfSubstituent\}$
with $\LoopIndices_{\Colour_{1}}^{\ColourOfSubstituent_{1}}\neq\varnothing$.
We choose $m_{1}\in\LoopIndices_{\Colour_{1}}^{\ColourOfSubstituent_{1}}$
and consider a self-avoiding path in $\GraphOfSubstituent$ connecting
$p_{1}$ with $m_{1}$. This gives rise to a path in $\SubstitutedGraph{\Graph}{\GraphOfSubstituent}$
connecting $(p_{1}-1)\NumberOfVertices+q_{1}$ with $(m_{1}-1)\NumberOfVertices+q_{1}$,
which is connected with $(m_{1}-1)\NumberOfVertices+i_{1}$ via a
$\ColourOfSubstituent_{1}$-colored link. In the next step, we choose
$m_{2}\in\LoopIndices_{\Colour_{2}}^{\ColourOfSubstituent_{2}}$ together
with a self-avoiding path in $\GraphOfSubstituent$ connecting $m_{1}$
with $m_{2}$. The rest follows by induction.

2) We have to show that any two vertices of $\SubstitutedGraph{\Graph}{\GraphOfSubstituent}$
are connected by a unique self-avoiding path, in other words, that
the self-avoiding paths constructed above are the only connecting
ones. This follows from the observation that a self-avoiding path
in $\SubstitutedGraph{\Graph}{\GraphOfSubstituent}$ gives rise to
paths in $\Graph$ and $\GraphOfSubstituent$, respectively, and non-uniqueness
would contradict the assumption that these graphs are treelike.

3) We use that $\OnSecond{\AdjacencyMatrix}^{\Colour}T=T\AdjacencyMatrix^{\Colour}$
for $\Colour=1,2,\ldots,\NumberOfColours$ to obtain
\begin{eqnarray*}
\OnSecond{\AdjacencyMatrix}_{\Substitution}^{\ColourOfSubstituent}\left(I_{\NumberOfVerticesOfSubstituent}\Cross T\right) & = & (\AdjacencyMatrixOfSubstituent{\ColourOfSubstituent})_{*}\Cross T+\sum_{\Colour}\sum_{i\in\LoopIndices_{\Colour}^{\ColourOfSubstituent}}(\AdjacencyMatrixOfSubstituent{\ColourOfSubstituent})_{i}\Cross\OnSecond{\AdjacencyMatrix}^{\Colour}T\\
 & = & (\AdjacencyMatrixOfSubstituent{\ColourOfSubstituent})_{*}\Cross T+\sum_{\Colour}\sum_{i\in\LoopIndices_{\Colour}^{\ColourOfSubstituent}}(\AdjacencyMatrixOfSubstituent{\ColourOfSubstituent})_{i}\Cross T\AdjacencyMatrix^{\Colour}=\left(I_{\NumberOfVerticesOfSubstituent}\Cross T\right)\AdjacencyMatrix_{\Substitution}^{\ColourOfSubstituent}.
\end{eqnarray*}

4) The isomorphism maps an adjacency matrix $\AdjacencyMatrix_{\Substitution}^{\ColourOfSubstituent}\in\SubstitutedGraph G{\GroupOfSubstituent}$
to $(f,h)\in G\wr_{Y}\GroupOfSubstituent$ determined as follows.
Let $e_{G}$ denote the neutral element of $G\leq\mathrm{Sym}(X)$.
If $\AdjacencyMatrixOfSubstituent{\ColourOfSubstituent}(\BasisVector_{i}^{\NumberOfVerticesOfSubstituent})=\BasisVector_{j}^{\NumberOfVerticesOfSubstituent}$
for $i\neq j$ or $i=j\notin\LoopIndices^{\ColourOfSubstituent}$,
then $f(\BasisVector_{i}^{\NumberOfVerticesOfSubstituent})=f(\BasisVector_{j}^{\NumberOfVerticesOfSubstituent})=e_{G}$
and $h(\BasisVector_{i}^{\NumberOfVerticesOfSubstituent})=\BasisVector_{j}^{\NumberOfVerticesOfSubstituent}$.
On the other hand, if $i\in\LoopIndices_{\Colour}^{\ColourOfSubstituent}$,
then $f(\BasisVector_{i}^{\NumberOfVerticesOfSubstituent})=\AdjacencyMatrix^{\Colour}$
and $h(\BasisVector_{i}^{\NumberOfVerticesOfSubstituent})=\BasisVector_{i}^{\NumberOfVerticesOfSubstituent}$.
We leave it to the reader to check that this defines an isomorphism
with the desired properties.
\end{proof}
We use Figure~\ref{fig:SubstitutionInTermsOfTiledDomains} to exemplify
the last statement of the previous theorem as well as the adaptions
that are necessary if $\GraphOfSubstituent$ has Dirichlet loops.
The adjacency matrices of $\SubstitutedGraph{\Graph}{\GraphOfSubstituent}$
are determined by their action on $Z=\{\BasisVector_{1}^{4},\BasisVector_{2}^{4},\BasisVector_{3}^{4},\BasisVector_{4}^{4},-\BasisVector_{1}^{4},-\BasisVector_{2}^{4},-\BasisVector_{3}^{4},-\BasisVector_{4}^{4}\}\subset\mathbb{R}^{4}$,
which can be regarded as the set of vertices of the double cover of
$\SubstitutedGraph{\Graph}{\GraphOfSubstituent}$. We identify $Z$
with $X\times Y=\{\BasisVector_{1}^{2},\BasisVector_{2}^{2},-\BasisVector_{1}^{2},-\BasisVector_{2}^{2}\}\times\{\BasisVector_{1}^{2},\BasisVector_{2}^{2}\}\simeq\{1,2,3,4\}\times\{1_{S},2_{S}\}$
via 
\[
\begin{array}{rrrr}
\BasisVector_{1}^{4}\simeq(1,1_{S}) & \BasisVector_{2}^{4}\simeq(2,1_{S}) & \BasisVector_{3}^{4}\simeq(1,2_{S}) & \BasisVector_{4}^{4}\simeq(2,2_{S})\phantom{.}\\
-\BasisVector_{1}^{4}\simeq(3,1_{S}) & -\BasisVector_{2}^{4}\simeq(4,1_{S}) & -\BasisVector_{3}^{4}\simeq(3,2_{S}) & -\BasisVector_{4}^{4}\simeq(4,2_{S}).
\end{array}
\]
We have $G=\langle\AdjacencyMatrix^{curly},\AdjacencyMatrix^{ridged}\rangle\simeq\langle(1,3)(2)(4),(1,2)(3,4)\rangle$.
Writing functions $f\in G^{\{1_{S},2_{S}\}}$ as pairs $(f(1_{S}),f(2_{S}))\in G^{2}$,
an injective homomorphism into $\mathrm{Sym}(X)\wr_{Y}\mathrm{Sym}(Y)$
is given by \setlength{\arraycolsep}{0.8pt}
\[
\begin{array}{lllll}
\AdjacencyMatrix_{\Substitution}^{straight} & \simeq & (1,2)(3,4)(5,6)(7,8) & \simeq & (((1,2)(3,4),(1,2)(3,4)),(1_{S})(2_{S}))\\
\AdjacencyMatrix_{\Substitution}^{wavy} & \simeq & (1,3)(2,4)(5,7)(6,8) & \simeq & (((1)(2)(3)(4),(1)(2)(3)(4)),(1_{S},2_{S}))\\
\AdjacencyMatrix_{\Substitution}^{zig-zag} & \simeq & (1,5)(2,6)(3,7)(4)(8) & \simeq & (((1,3)(2,4),(1,3)(2)(4)),(1_{S})(2_{S})),
\end{array}
\]
\setlength{\arraycolsep}{\myArraycolsep}where the first identification
refers to the action on $Z$. The \emph{zig-zag} Dirichlet loop of
$\GraphOfSubstituent$ is incorporated into the $G^{2}$-part of the
image of $\AdjacencyMatrix_{\Substitution}^{zig-zag}$ as $(1,3)(2,4)$,
which corresponds to $-I_{2}$ on the level of matrices. In this example,
$-I_{2}\in G$, but $-I_{\NumberOfVertices}\notin G$ is possible
in general. In any case, $\SubstitutedGraph G{\GroupOfSubstituent}$
is isomorphic to a subgroup of $\langle G,-I_{\NumberOfVertices}\rangle\wr_{Y}\mathrm{Sym}(Y)$.

\section{Examples of transplantable pairs\label{sec:Examples_of_Trans_Pairs}}

\subsection{Treelike pairs and planar domains\label{sec:TransplantableTrees}}

Since Kac's original question~\cite{Kac1966} aimed at planar domains
with pure boundary conditions, transplantable treelike graphs with
uniform loop signs have been studied the most~\cite{OkadaShudo2001,Thas2006a,Thas2006b,Thas2006,Thas2007}.
Note that the Substitution Theorem implies that there exist infinitely
many such pairs, which answers a question raised in~\cite{GiraudThas2010}.
However, it could be the case that all such pairs arise from the ones
in~\cite{BuserConwayDoyleSemmler1994} by substitution. Okada and
Shudo~\cite{OkadaShudo2001} confirmed the known pairs up to $13$
vertices per graph and just missed the $7$ pairs with $14$ vertices
per graph that arise from the $3$ known ones with $7$ vertices per
graph by substitution. Figure~\ref{fig:TransplantableGraphsWithEvenNumber}
shows one of these pairs with Dirichlet loops. Using Dualization and
Theorem~\ref{thm:Transplantability_is_Equivalent_Inductions}, one
sees that transplantable treelike graphs with uniform loop signs come
from the original Sunada method~\cite{Sunada1985}, that is, from
Gassmann triples $(G,H,\OnSecond H)$. For each of the $7$ pairs
mentioned above, $G$ is isomorphic to $PSL(3,2)\times PSL(3,2)$.
Its permutation action on the vertices is imprimitive, in particular,
it is not 2-transitive, which gives a counterexample to a conjecture
expressed in~\cite{SchillewaertThas2011}.

\begin{figure}
\noindent \begin{centering}
\psfragBodySmall\psfrag{A}{} \psfrag{B}{}\hfill{}\includegraphics[scale=0.3]{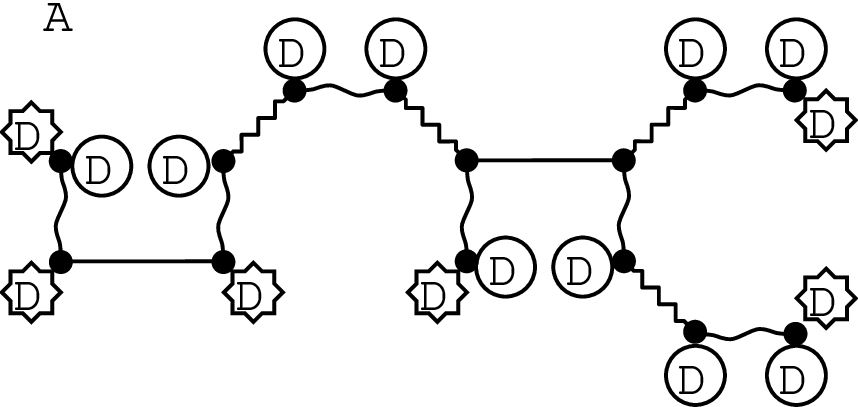}\hfill{}\includegraphics[scale=0.3]{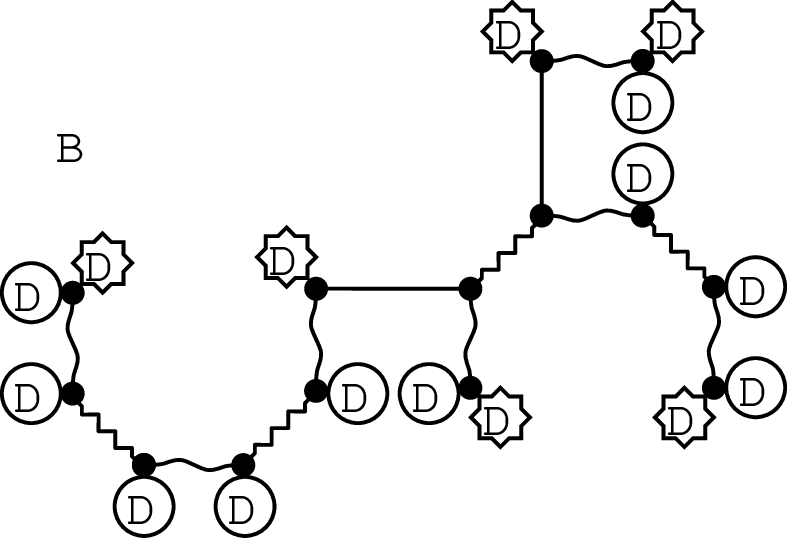}\hfill{}
\par\end{centering}

\caption{Treelike pair obtained from the one underlying the Gorden-Webb-Wolpert
drums by substitution.\label{fig:TransplantableGraphsWithEvenNumber}}
\end{figure}

\begin{figure}
\noindent \hfill{}\subfloat[Transplantable manifolds\label{fig:FourTransplantableDomains}]{\noindent \begin{centering}
\begin{minipage}[b][21mm]{55mm}%
\noindent %
\begin{tabular}{ccc}
\noalign{\vskip2mm}
$B$ & $M_{r}$ & $M_{l}$\tabularnewline
\includegraphics[scale=0.45]{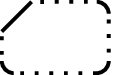} & \includegraphics[scale=0.45]{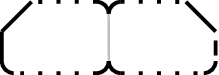} & \includegraphics[scale=0.45]{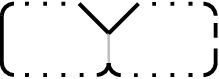}\tabularnewline[4mm]
$M_{b}$ & $M_{br}=M_{rb}$ & $M_{bl}=M_{lb}$\tabularnewline
\includegraphics[scale=0.45]{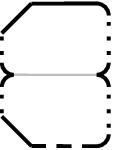} & \includegraphics[scale=0.45]{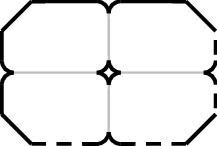} & \includegraphics[scale=0.45]{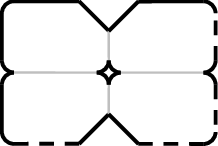}\tabularnewline[1.5mm]
$M_{t}$ & $M_{tr}=M_{rt}$ & $M_{tl}=M_{lt}$\tabularnewline
\includegraphics[scale=0.45]{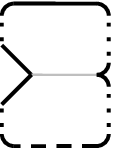} & \includegraphics[scale=0.45]{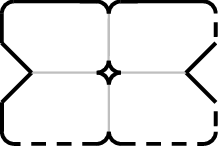} & \includegraphics[scale=0.45]{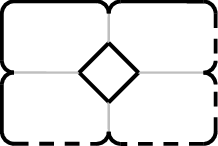}\tabularnewline
\end{tabular}%
\end{minipage}
\par\end{centering}

\centering{}}\hspace*{\fill}\subfloat[Transplantable graphs\label{fig:FourTransplantableGraphs}]{\noindent \begin{centering}
\begin{minipage}[b][24mm]{50mm}%
\noindent \psfragBodySmall%
\begin{tabular}{ccc}
 & $\Graph_{r}$\hspace*{\fill} & $\Graph_{l}$\hspace*{\fill}\tabularnewline[-2mm]
\includegraphics[scale=0.3]{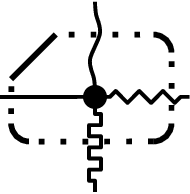} & \includegraphics[scale=0.3]{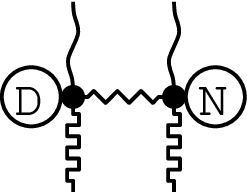} & \includegraphics[scale=0.3]{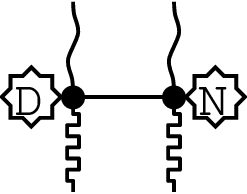}\tabularnewline[2mm]
$\Graph_{b}$ & $\Graph_{br}=\Graph_{rb}$ & $\Graph_{bl}=\Graph_{lb}$\tabularnewline
\includegraphics[scale=0.3]{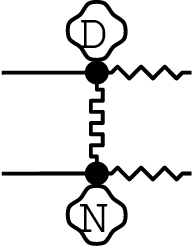} & \includegraphics[scale=0.3]{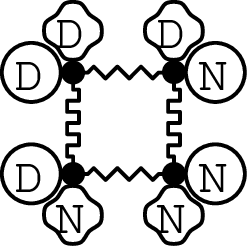} & \includegraphics[scale=0.3]{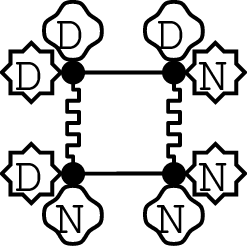}\tabularnewline[1mm]
$\Graph_{t}$ & $\Graph_{tr}=\Graph_{rt}$ & $\Graph_{tl}=\Graph_{lt}$\tabularnewline
\includegraphics[scale=0.3]{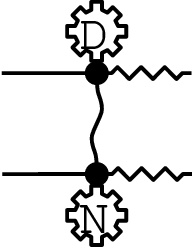} & \includegraphics[scale=0.3]{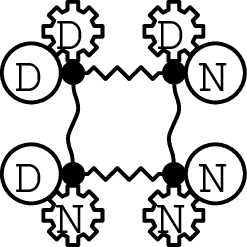} & \includegraphics[scale=0.3]{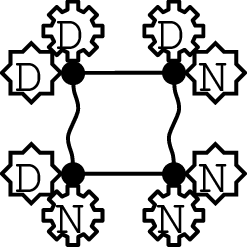}\tabularnewline
\end{tabular}%
\end{minipage}
\par\end{centering}

}\hfill{}

\caption{Transplantable tuple obtained as in~\cite{LevitinParnovskiPolterovich2006}.
\label{fig:FourTransplantable_Tuples}}
\end{figure}

\subsection{Transplantable tuples}

We show how crossings and substitutions give rise to transplantable
tuples, that is, tuples of pairwise transplantable and therefore isospectral
manifolds. Assume first that $\Graph$ and $\OnSecond{\Graph}$ are
transplantable graphs without Dirichlet loops. Since every graph is
trivially transplantable to itself, we can apply the crossing method
to the pairs $(\Graph,\Graph)$, $(\Graph,\OnSecond{\Graph})$, $(\OnSecond{\Graph},\OnSecond{\Graph})$
and $(\OnSecond{\Graph},\Graph)$ to obtain the transplantable pairs
$(\Crossing{\Graph}{\Graph},\Crossing{\Graph}{\OnSecond{\Graph}})$,
$(\Crossing{\Graph}{\OnSecond{\Graph}},\Crossing{\OnSecond{\Graph}}{\OnSecond{\Graph}})$
and $(\Crossing{\OnSecond{\Graph}}{\OnSecond{\Graph}},\Crossing{\OnSecond{\Graph}}{\Graph})$.
Using the transplantable tuple $(\Crossing{\Graph}{\Graph},\Crossing{\Graph}{\OnSecond{\Graph}},\Crossing{\OnSecond{\Graph}}{\OnSecond{\Graph}},\Crossing{\OnSecond{\Graph}}{\Graph})$
and the above pairs, we get $8$ pairwise transplantable graphs. This
process can be continued inductively.

Figure~\ref{fig:FourTransplantableDomains} indicates a more general
method that is based on \cite{LevitinParnovskiPolterovich2006} and
allows for mixed boundary conditions. We can regard $M_{br}$ as being
built out of copies of $M_{b}$ which gives transplantability with
$M_{bl}$. On the other hand, we can take $M_{r}$ as the building
block to show that $M_{br}$ and $M_{tr}$ are transplantable. Similar
arguments show that $(M_{br},M_{bl},M_{tr},M_{tl})$ is a transplantable
tuple.

In the same way, there can be several ways in which a given loop-signed
graph can be interpreted as a substituted graph as indicated in Figure~\ref{fig:FourTransplantableGraphs}.
We generalize this idea. Let $\Graph_{1}$ and $\Graph_{2}$ be loop-signed
graphs with $\NumberOfVertices_{1}$, respectively $\NumberOfVertices_{2}$,
vertices and adjacency matrices $(\AdjacencyMatrix_{1}^{\Colour})_{\Colour=1}^{\NumberOfColours_{1}}$,
respectively $(\AdjacencyMatrix_{2}^{\Colour})_{\Colour=1}^{\NumberOfColours_{2}}$.
Let $\Graph_{1,S}$ be the graph that is obtained by adding a Neumann
loop of each edge color of $\Graph_{2}$ to each vertex of $\Graph_{1}$,
that is, $\Graph_{1,S}$ has adjacency matrices $(\AdjacencyMatrix_{1}^{\Colour})_{\Colour=1}^{\NumberOfColours_{1}+\NumberOfColours_{2}}$,
where $\AdjacencyMatrix_{1}^{\NumberOfColours_{1}+\Colour}=I_{\NumberOfVertices_{1}}$
for $\Colour=1,2,\ldots,\NumberOfColours_{2}$. In Figure~\ref{fig:FourTransplantableGraphs},
the added loops are indicated by outward-pointing edges. In order
to substitute $\Graph_{1,S}$ into $\Graph_{2}$, we set $\LoopIndices^{\Colour}=\varnothing$
for $\Colour=1,2,\ldots,\NumberOfColours_{1}$, that is, $\AdjacencyMatrix_{1*}^{\Colour}=\AdjacencyMatrix_{1}^{\Colour}$,
and $\LoopIndices^{\NumberOfColours_{1}+\Colour}=\LoopIndices_{\Colour}^{\NumberOfColours_{1}+\Colour}=\{1,2,\ldots,\NumberOfVertices_{1}\}$
for $\Colour=1,2,\ldots,\NumberOfColours_{2}$, that is, $\AdjacencyMatrix_{1*}^{\NumberOfColours_{1}+\Colour}=0$.
According to Definition~\ref{def:TheAdjacencyMatrixOfSubstitutedGraph},
the adjacency matrices of $\SubstitutedGraph{\Graph_{2}}{\Graph_{1,S}}$
read
\begin{eqnarray*}
\AdjacencyMatrix_{1\Substitution2}^{\Colour} & = & \AdjacencyMatrix_{1}^{\Colour}\Cross I_{\NumberOfVertices_{2}}\quad\text{for }\Colour=1,2,\ldots,\NumberOfColours_{1}\mbox{ and}\\
\AdjacencyMatrix_{1\Substitution2}^{\NumberOfColours_{1}+\Colour} & = & I_{\NumberOfVertices_{1}}\Cross\AdjacencyMatrix_{2}^{\Colour}\quad\text{for }\Colour=1,2,\ldots,\NumberOfColours_{2}.
\end{eqnarray*}
Similarly, let $\Graph_{2,S}$ be the graph with adjacency matrices
$(\AdjacencyMatrix_{2}^{\Colour-\NumberOfColours_{1}})_{\Colour=1}^{\NumberOfColours_{1}+\NumberOfColours_{2}}$,
where $\AdjacencyMatrix_{2}^{\Colour-\NumberOfColours_{1}}=I_{\NumberOfVertices_{2}}$
for $\Colour=1,2,\ldots,\NumberOfColours_{1}$. If we substitute $\Graph_{2,S}$
into $\Graph_{1}$ as above, we obtain $\SubstitutedGraph{\Graph_{1}}{\Graph_{2,S}}$
with adjacency matrices $\AdjacencyMatrix_{2\Substitution1}^{\Colour}=I_{\NumberOfVertices_{2}}\Cross\AdjacencyMatrix_{1}^{\Colour}$
for $\Colour=1,2,\ldots,\NumberOfColours_{1}$, and $\AdjacencyMatrix_{2\Substitution1}^{\NumberOfColours_{1}+\Colour}=\AdjacencyMatrix_{2}^{\Colour}\Cross I_{\NumberOfVertices_{1}}$
for $\Colour=1,2,\ldots,\NumberOfColours_{2}$. Lemma~\ref{lem:Commutation_Relation_of_Kronecker_Product}
implies that $\SubstitutedGraph{\Graph_{2}}{\Graph_{1,S}}$ and $\SubstitutedGraph{\Graph_{1}}{\Graph_{2,S}}$
are isomorphic, which is indicated in Figure~\ref{fig:FourTransplantableGraphs}
where $\Graph_{br}=\SubstitutedGraph{\Graph_{r}}{\Graph_{b}}$ and
$\Graph_{rb}=\SubstitutedGraph{\Graph_{b}}{\Graph_{r}}$.

We use this result to generate transplantable tuples. Let $(\Graph_{1},\OnSecond{\Graph}_{1})$
and $(\Graph_{2},\OnSecond{\Graph}_{2})$ be pairs of transplantable
loop-signed graphs. Using the above notation, we see that $(\SubstitutedGraph{\Graph_{2}}{\Graph_{1,S}},\SubstitutedGraph{\Graph_{1}}{\Graph_{2,S}})$,
$(\SubstitutedGraph{\OnSecond{\Graph}_{2}}{\Graph_{1,S}},\SubstitutedGraph{\Graph_{1}}{\OnSecond{\Graph}_{2,S}})$,
$(\SubstitutedGraph{\Graph_{2}}{\OnSecond{\Graph}_{1,S}},\SubstitutedGraph{\OnSecond{\Graph}_{1}}{\Graph_{2,S}})$
and $(\SubstitutedGraph{\OnSecond{\Graph}_{2}}{\OnSecond{\Graph}_{1,S}},\SubstitutedGraph{\OnSecond{\Graph}_{1}}{\OnSecond{\Graph}_{2,S}})$
are pairs of isomorphic graphs. The Substitution Theorem shows that
$(\SubstitutedGraph{\Graph_{2}}{\Graph_{1,S}},\SubstitutedGraph{\OnSecond{\Graph}_{2}}{\Graph_{1,S}})$,
$(\SubstitutedGraph{\Graph_{2}}{\OnSecond{\Graph}_{1,S}},\SubstitutedGraph{\OnSecond{\Graph}_{2}}{\OnSecond{\Graph}_{1,S}})$,
$(\SubstitutedGraph{\Graph_{1}}{\Graph_{2,S}},\SubstitutedGraph{\OnSecond{\Graph}_{1}}{\Graph_{2,S}})$
and $(\SubstitutedGraph{\Graph_{1}}{\OnSecond{\Graph}_{2,S}},\SubstitutedGraph{\OnSecond{\Graph}_{1}}{\OnSecond{\Graph}_{2,S}})$
are pairs of transplantable graphs. Hence, $(\SubstitutedGraph{\Graph_{2}}{\Graph_{1,S}},\SubstitutedGraph{\OnSecond{\Graph}_{2}}{\Graph_{1,S}},\SubstitutedGraph{\OnSecond{\Graph}_{2}}{\OnSecond{\Graph}_{1,S}},\SubstitutedGraph{\Graph_{2}}{\OnSecond{\Graph}_{1,S}})$
is a transplantable tuple, which, together with another transplantable
pair $(\Graph_{3},\OnSecond{\Graph}_{3})$, can be used to generate
$8$ pairwise transplantable graphs. We proceed inductively, where
in each step we could use the same transplantable pair $(\Graph_{i},\OnSecond{\Graph}_{i})=(\Graph_{1},\OnSecond{\Graph}_{1})$,
for instance, $(\Graph_{r},\Graph_{l})$ and $(\Graph_{b},\Graph_{t})$
in Figure~\ref{fig:FourTransplantableGraphs} are isomorphic up to
a relabelling of their edge colors.

\subsection{The algorithm\label{sub:TheAlgorithm}}

In order to search for transplantable pairs systematically, one first
creates one representative of each isomorphism class of loop-signed
graphs with a given number of vertices and edge colors and then sorts
them according to finitely many expressions of the form (\ref{eq:FiniteTraceCondition}),
which is justified by the following proposition.
\begin{prop}
\label{pro:BoundOnNumberOfSufficientWords}If $(\AdjacencyMatrix^{\Colour})_{\Colour=1}^{\NumberOfColours}$
and $(\OnSecond{\AdjacencyMatrix}^{\Colour})_{\Colour=1}^{\NumberOfColours}$
are the adjacency matrices of two loop-signed graphs each of which
has $\NumberOfVertices$ vertices, then 
\begin{equation}
\Trace(\OnSecond{\AdjacencyMatrix}^{\Colour_{1}}\OnSecond{\AdjacencyMatrix}^{\Colour_{2}}\cdots\OnSecond{\AdjacencyMatrix}^{\Colour_{l}})=\Trace(\AdjacencyMatrix^{\Colour_{1}}\AdjacencyMatrix^{\Colour_{2}}\cdots\AdjacencyMatrix^{\Colour_{l}})\label{eq:FiniteTraceCondition}
\end{equation}
holds for all words $\Colour_{1}\Colour_{2}\ldots\Colour_{l}$ in
the edge colors if it holds for all words with length $l\leq(2\NumberOfVertices)!$.\end{prop}
\begin{proof}
Recall that $G=\langle\AdjacencyMatrix^{1},\AdjacencyMatrix^{2},\ldots,\AdjacencyMatrix^{\NumberOfColours}\rangle$
can be identified with a subgroup of the symmetric group on $2\NumberOfVertices$
elements. If $L>(2\NumberOfVertices)!$, then any word $\Colour_{1}\Colour_{2}\ldots\Colour_{L}$
contains a subword $\Colour_{k+1}\Colour_{k+2}\ldots\Colour_{k+l}$
with $0<l<L$ such that $\AdjacencyMatrix^{\Colour_{k+1}}\AdjacencyMatrix^{\Colour_{k+2}}\cdots\AdjacencyMatrix^{\Colour_{k+l}}=I_{\NumberOfVertices}$.
Assuming that~(\ref{eq:FiniteTraceCondition}) holds for words up
to length $L-1$, we see that $\OnSecond{\AdjacencyMatrix}^{\Colour_{k+1}}\OnSecond{\AdjacencyMatrix}^{\Colour_{k+2}}\cdots\OnSecond{\AdjacencyMatrix}^{\Colour_{k+l}}=I_{\NumberOfVertices}$.
Hence, $\Colour_{1}\Colour_{2}\ldots\Colour_{L}$ satisfies~(\ref{eq:FiniteTraceCondition})
if $\Colour_{1}\Colour_{2}\ldots\Colour_{k}\Colour_{k+l+1}\Colour_{k+l+2}\ldots\Colour_{L}$
does, and the statement follows by induction.
\end{proof}
The graph generating part of the algorithm can be dealt with by considering
the sequences of vertices that appear when one performs a walk on
the vertices of a loop-signed graph where the edges are used in the
order of their color. Doyle~\cite{Doyle} suggested an algorithm
to tackle the sorting part, which was motivated by the discovery of
the graphs in Figure~\ref{fig:IsospecNonTrans}. Despite not being
transplantable, these pairs are strongly isospectral in the sense
that if $(\AdjacencyMatrix^{\Colour})_{\Colour=1}^{\NumberOfColours}$
and $(\OnSecond{\AdjacencyMatrix}^{\Colour})_{\Colour=1}^{\NumberOfColours}$
denote their adjacency matrices, then for any $z_{1},z_{2},\ldots,z_{\NumberOfColours}\in\mathbb{C}$,
\[
\det\biggl(\sum_{\Colour=1}^{\NumberOfColours}z_{\Colour}\AdjacencyMatrix^{\Colour}\biggr)=\det\biggl(\sum_{\Colour=1}^{\NumberOfColours}z_{\Colour}\OnSecond{\AdjacencyMatrix}^{\Colour}\biggr).
\]
The graphs in Figure~\ref{fig:IsospecNonTrans6V} only differ in
the order in which the edge colors \emph{straight} and \emph{zig-zag}
appear along the outer 6-cycles, which suggests to interpret transplantability
as a non-commutative version of strong isospectrality. More precisely,
if $(\AdjacencyMatrix^{\Colour})_{\Colour=1}^{\NumberOfColours}$
and $(\OnSecond{\AdjacencyMatrix}^{\Colour})_{\Colour=1}^{\NumberOfColours}$
are the adjacency matrices of a pair of transplantable graphs, then
we obtain that for any square matrices $(Z^{\Colour})_{\Colour=1}^{\NumberOfColours}$
of equal size and any $k\in\mathbb{N}$ 
\begin{eqnarray}
\Trace\Biggl(\Biggl(\sum_{\Colour=1}^{\NumberOfColours}\AdjacencyMatrix^{\Colour}\Cross Z^{\Colour}\Biggr)^{k}\Biggr) & = & \Trace\Biggl(\sum_{1\leq\Colour_{1},\Colour_{2},\ldots,\Colour_{k}\leq\NumberOfColours}\Biggl(\prod_{i=1}^{k}\AdjacencyMatrix^{\Colour_{i}}\Cross Z^{\Colour_{i}}\Biggr)\Biggr)\nonumber \\
 & = & \sum_{1\leq\Colour_{1},\Colour_{2},\ldots,\Colour_{k}\leq\NumberOfColours}\Trace\Biggl(\prod_{i=1}^{k}\AdjacencyMatrix^{\Colour_{i}}\Biggr)\Trace\Biggl(\prod_{i=1}^{k}Z^{\Colour_{i}}\Biggr)\label{eq:TransTestExpanded}\\
 & = & \Trace\Biggl(\Biggl(\sum_{\Colour=1}^{\NumberOfColours}\OnSecond{\AdjacencyMatrix}^{\Colour}\Cross Z^{\Colour}\Biggr)^{k}\Biggr).\label{eq:TransTestWithNonCommMats}
\end{eqnarray}
Note that (\ref{eq:TransTestExpanded}) contains the traces of all
products of $k$ adjacency matrices. If the scalars $\Trace(\prod_{i=1}^{k}Z^{\Colour_{i}})\in\mathbb{C}$
were linearly independent over $\mathbb{Q}$ up to the cyclic invariance
of the trace, then one could read off the integer factors $\Trace(\prod_{i=1}^{k}\AdjacencyMatrix^{\Colour_{i}})$.
This is not the case as can be seen by regarding the expressions $\Trace(\prod_{i=1}^{k}Z^{\Colour_{i}})$
as homogeneous polynomials in the entries of $(Z^{\Colour})_{\Colour=1}^{\NumberOfColours}$.
For instance, the graphs in Figure~\ref{fig:PairPassingTest} satisfy
(\ref{eq:TransTestWithNonCommMats}) for any $2\times2$ matrices
$(Z^{\Colour})_{\Colour=1}^{3}$ despite not being transplantable.

\begin{figure}
\noindent \begin{centering}
\hfill{}\subfloat[Dirichlet\label{fig:IsospecNonTrans6V}]{\noindent \centering{}%
\begin{minipage}[b]{35mm}%
\noindent \begin{center}
\begin{tabular}{c}
\includegraphics[scale=0.3]{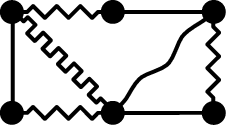}\tabularnewline
\includegraphics[scale=0.3]{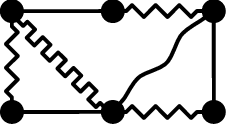}\tabularnewline
\end{tabular}
\par\end{center}%
\end{minipage}}\hfill{}\subfloat[Dirichlet or Neumann\label{fig:IsospecNonTrans9V}]{\noindent \centering{}%
\begin{minipage}[b]{55mm}%
\noindent \begin{center}
\begin{tabular}{c}
\includegraphics[scale=0.3]{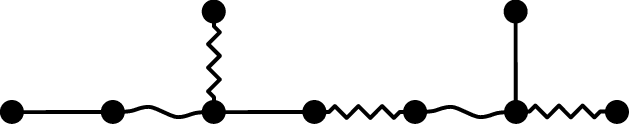}\tabularnewline
\includegraphics[scale=0.3]{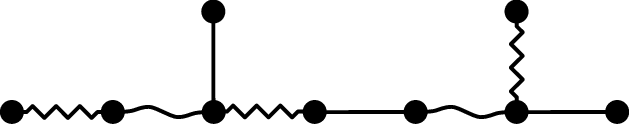}\tabularnewline
\end{tabular}
\par\end{center}%
\end{minipage}}\hfill{}\subfloat[Dirichlet or Neumann\label{fig:PairPassingTest}]{\noindent \begin{centering}
\begin{minipage}[b]{0.35\columnwidth}%
\noindent \begin{center}
\begin{tabular}{c}
\includegraphics[scale=0.3]{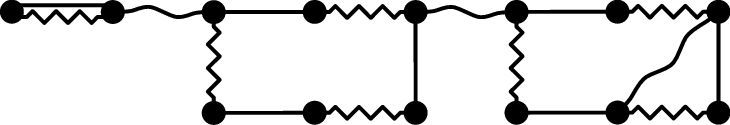}\tabularnewline
\includegraphics[scale=0.3]{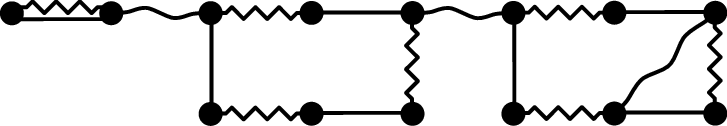}\tabularnewline
\end{tabular}
\par\end{center}%
\end{minipage}
\par\end{centering}

}\hfill{}
\par\end{centering}

\caption{Non-transplantable but strongly isospectral graphs when given uniform
loop signs as indicated. The last pair satisfies (\ref{eq:TransTestWithNonCommMats})
for any $2\times2$ matrices $(Z^{\Colour})_{\Colour=1}^{3}$.\label{fig:IsospecNonTrans}}
\end{figure}

After sorting out pairs of possibly transplantable graphs via (\ref{eq:TransTestWithNonCommMats}),
the existence of an intertwining transplantation $T$ is determined
by considering the following action of the adjacency matrices on the
entries of $T$. If $(\AdjacencyMatrix^{\Colour})_{\Colour=1}^{\NumberOfColours}$
and $(\OnSecond{\AdjacencyMatrix}^{\Colour})_{\Colour=1}^{\NumberOfColours}$
are $\NumberOfVertices\times\NumberOfVertices$ adjacency matrices
and if $\AdjacencyMatrix_{jl}^{\Colour}\neq0$ as well as $\OnSecond{\AdjacencyMatrix}_{ik}^{\Colour}\neq0$,
then any intertwining $T$ satisfies
\begin{equation}
T_{ij}\AdjacencyMatrix_{jl}^{\Colour}=(T\AdjacencyMatrix^{\Colour})_{il}=(\OnSecond{\AdjacencyMatrix}^{\Colour}T)_{il}=\OnSecond{\AdjacencyMatrix}_{ik}^{\Colour}T_{kl}.\label{eq:FinalTransTest}
\end{equation}
We therefore consider the action of the group $\langle(\AdjacencyMatrix^{\Colour},\OnSecond{\AdjacencyMatrix}^{\Colour})\rangle_{\Colour=1}^{\NumberOfColours}$
on the set $\{1,2,\ldots,\NumberOfVertices\}^{2}\times\{-1,1\}$ given
by 
\[
(\AdjacencyMatrix^{\Colour},\OnSecond{\AdjacencyMatrix}^{\Colour})((i,j),\pm1)=((k,l),\pm\AdjacencyMatrix_{jl}^{\Colour}\OnSecond{\AdjacencyMatrix}_{ik}^{\Colour}),
\]
where $(k,l)$ is the unique pair such that $\AdjacencyMatrix_{jl}^{\Colour}\neq0$
and $\OnSecond{\AdjacencyMatrix}_{ik}^{\Colour}\neq0$. If an orbit
contains both $((i,j),+1)$ and $((i,j),-1)$ for some $i$ and $j$,
then all entries of $T$ with indices in that orbit must be zero.
At the end, one checks whether the remaining orbits can be assigned
to real numbers such that the resulting matrix $T$ becomes invertible.

\subsection{Pairs with 2 edge colors\label{sub:BBlocksWith2ReflSegs}}

The method in~\cite{LevitinParnovskiPolterovich2006,JakobsonLevitinNadirashviliPolterovich2006}
corresponds to the case of connected tiled manifolds with building
blocks that have $2$ reflecting faces. In order to classify such
pairs, let $\Graph$ and $\OnSecond{\Graph}$ be transplantable connected
loop-signed graphs with $\NumberOfVertices$ vertices and adjacency
matrices $(\AdjacencyMatrix^{1},\AdjacencyMatrix^{2})$ and $(\OnSecond{\AdjacencyMatrix}^{1},\OnSecond{\AdjacencyMatrix}^{2})$,
respectively. Recall that transplantability is equivalent to
\begin{equation}
\Trace(\OnSecond{\AdjacencyMatrix}^{\Colour_{1}}\OnSecond{\AdjacencyMatrix}^{\Colour_{2}}\cdots\OnSecond{\AdjacencyMatrix}^{\Colour_{l}})=\Trace(\AdjacencyMatrix^{\Colour_{1}}\AdjacencyMatrix^{\Colour_{2}}\cdots\AdjacencyMatrix^{\Colour_{l}})\label{eq:Trace_Condition_In_Context_of_Two_EdgeColours}
\end{equation}
for all sequences $\Colour_{1}\Colour_{2}\ldots\Colour_{l}$ of edge
colors. If $\NumberOfVertices$ is odd, it is easy to see that $\Trace(\AdjacencyMatrix^{1})$
and $\Trace(\AdjacencyMatrix^{2})$ determine $\Graph$ up to isomorphism.
If $\NumberOfVertices$ is even, then $\Trace((\AdjacencyMatrix^{1}\AdjacencyMatrix^{2})^{\NumberOfVertices/2})=\NumberOfVertices$
if and only if $\Graph$ is a $\NumberOfVertices$-cycle. If, on the
other hand, $\Graph$ has two $\Colour$-colored loops, then their
signs are determined by $\Trace(\AdjacencyMatrix^{\Colour})$ up to
isomorphism of $\Graph$. The only case in which $\Graph$ and $\OnSecond{\Graph}$
are non-isomorphic and could be transplantable is shown in Figure~\ref{fig:TransplantablePairsWithTwoEdgeColours},
where $\Trace(\AdjacencyMatrix^{1})=\Trace(\AdjacencyMatrix^{2})=0$
and $\OnSecond{\AdjacencyMatrix}^{1}=\AdjacencyMatrix^{2}$ as well
as $\OnSecond{\AdjacencyMatrix}^{2}=\AdjacencyMatrix^{1}$. We verify
their transplantability. Since adjacency matrices are self-inverse,
it suffices to consider alternating sequences in~(\ref{eq:Trace_Condition_In_Context_of_Two_EdgeColours}),
that is, $\Colour_{i}\neq\Colour_{i+1}$ for $1\leq i\leq l-1$. If
$l$ is odd, then $\Colour_{1}=\Colour_{l}$. Using the cyclic invariance
of the trace, we can reduce to the sequence $\Colour_{2}\ldots\Colour_{l-1}$,
and (\ref{eq:Trace_Condition_In_Context_of_Two_EdgeColours}) follows
by induction. If $l$ is even, then (\ref{eq:Trace_Condition_In_Context_of_Two_EdgeColours})
follows directly from
\[
\Trace((\AdjacencyMatrix^{1}\AdjacencyMatrix^{2})^{l/2})=\Trace((\AdjacencyMatrix^{2}\AdjacencyMatrix^{1})^{l/2})=\Trace((\OnSecond{\AdjacencyMatrix}^{1}\OnSecond{\AdjacencyMatrix}^{2})^{l/2})=\Trace((\OnSecond{\AdjacencyMatrix}^{2}\OnSecond{\AdjacencyMatrix}^{1})^{l/2}).
\]
The preceding arguments provide an alternative proof of the extension
of~\cite[Theorem 4.2]{LevitinParnovskiPolterovich2006} given in~\cite{BandParzanchevskiBen-Shach2009}.

\begin{figure}
\noindent \begin{centering}
\psfragBody%
\begin{tabular}{cc}
\includegraphics[scale=0.42]{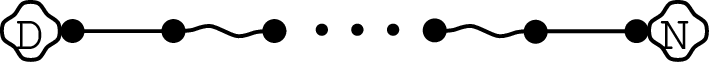} & \includegraphics[scale=0.42]{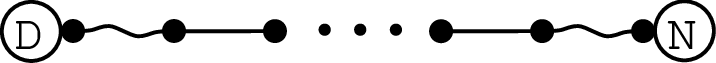}\tabularnewline
\end{tabular}
\par\end{centering}

\caption{Transplantable non-isomorphic connected loop-signed graphs with $2$
edge colors.\label{fig:TransplantablePairsWithTwoEdgeColours}}
\end{figure}

\subsection{Pairs with 3 edge colors\label{sub:PairsWithThreeReflectingSegmentsPerBuildingBlock}}

Table~\ref{tab:ResultsThreeEdgeColours} and Table~\ref{tab:ResultsThreeEdgeColoursHomo}
contain the results of the computer-aided search for transplantable
pairs with $3$ edge colors. For instance, there are $40$ isomorphism
classes of connected loop-signed graphs with $3$ edge colors and
$2$ vertices, among which there are $9$ transplantable pairs. If
we identify pairs which differ only in a renumbering of their edge
colors, then we obtain $3$ classes, which arise from the pair with
only $2$ edge colors by adding or copying of an edge color. As a
consequence of the Substitution Theorem, there are transplantable
pairs for all even numbers of vertices per graph. Moreover, each transplantable
tuple of length $l$ contributes $l(l-1)/2$ pairs.

The $32$ classes of pairs with $7$ vertices per graph are shown
in Figure~\ref{fig:Appendix-7-vertices-per-graph} in the appendix.
Their  loopless versions first appeared in~\cite{OkadaShudo2001}.
In particular, there exist $10$ versions of isospectral broken Gordon-Webb-Wolpert
drums~\cite{GordonWebbWolpert1992} shown in Figure~\ref{fig:IsospectralBrokenDrumsShapedLikeGordon}.
Pair $3$ was discovered by Parzanchevski and Band~\cite{ParzanchevskiBand2010}
using the sign representations of two $S_{4}$-subgroups of $PSL(3,2)$.
Isospectrality had been conjectured for pairs $5$ and $6$ by Driscoll
and Gottlieb~\cite{DriscollGottlieb2003}, who computed the respective
first $30$ eigenvalues of these manifolds numerically to high precision.

With regard to Table~\ref{tab:ResultsThreeEdgeColoursHomo}, recall
that transplantable pairs without Dirichlet loops originate from Gassmann
triples. For instance, the $19$ Neumann pairs with $11$ vertices
per graph come from $PSL(2,11)$. According to~\cite{Doyle}, they
were first discovered by John Conway. In accordance with~\cite{BosmaSmit2002},
there are no such pairs with $9$ or $10$ vertices per graph.

\begin{figure}
\noindent \begin{centering}
\newcommand{\vraise}{14mm}\newcommand{\neghspace}{-29mm} \newcommand{\poshspace}{20mm}\psfrag{1}{} \psfrag{2}{} \psfrag{3}{} \psfrag{4}{} \psfrag{5}{} \psfrag{6}{} \psfrag{7}{}%
\begin{tabular}{lc>{\raggedright}p{0.5mm}cc}
\includegraphics[scale=0.31]{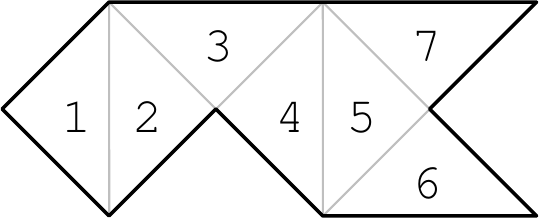}\raisebox{\vraise}{\hspace{\neghspace}1/2\hspace{\poshspace}} & \includegraphics[scale=0.31]{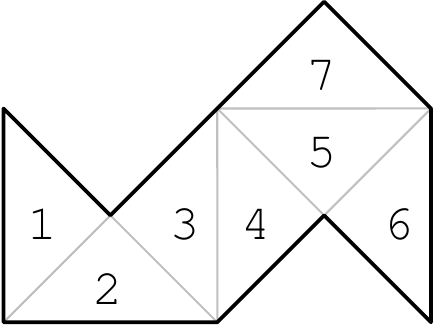} &  & \includegraphics[scale=0.31]{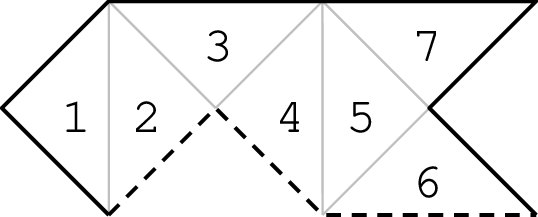}\raisebox{\vraise}{\hspace{\neghspace}3/4\hspace{\poshspace}} & \includegraphics[scale=0.31]{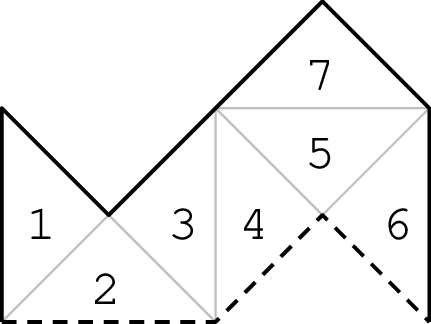}\tabularnewline
\includegraphics[scale=0.31]{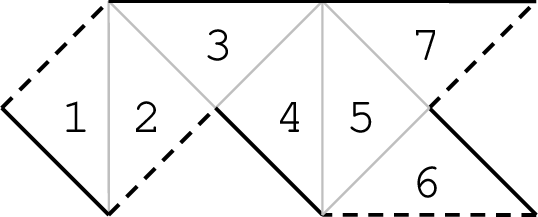}\raisebox{\vraise}{\hspace{\neghspace}5/6\hspace{\poshspace}} & \includegraphics[scale=0.31]{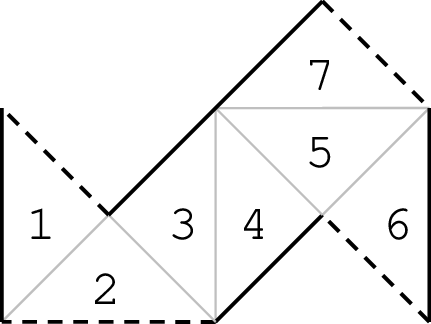} &  & \includegraphics[scale=0.31]{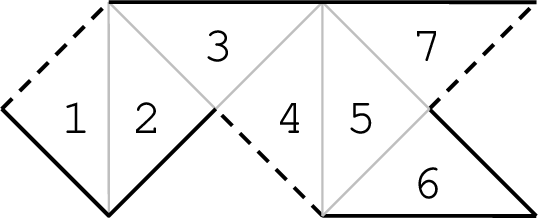}\raisebox{\vraise}{\hspace{\neghspace}7/8\hspace{\poshspace}} & \includegraphics[scale=0.31]{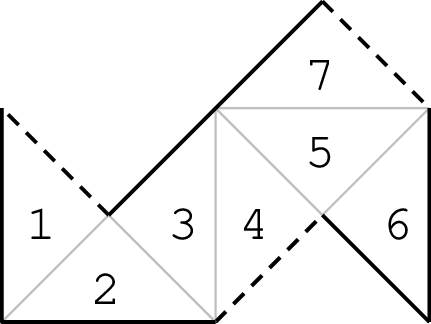}\tabularnewline
\includegraphics[scale=0.31]{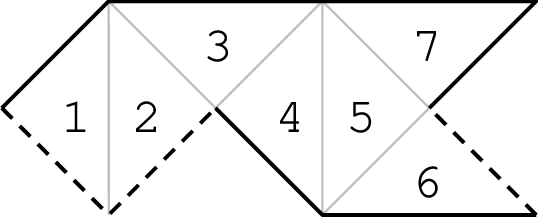}\raisebox{\vraise}{\hspace{\neghspace}9/10\hspace{\poshspace}} & \includegraphics[scale=0.31]{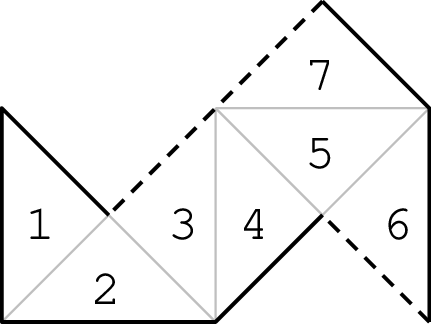} &  & \includegraphics[scale=0.31]{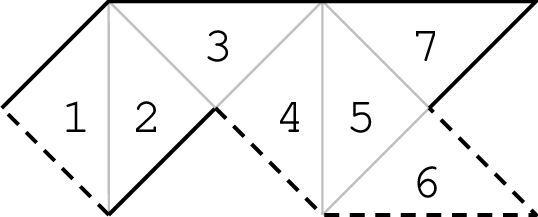}\raisebox{\vraise}{\hspace{\neghspace}11/12\hspace{\poshspace}} & \includegraphics[scale=0.31]{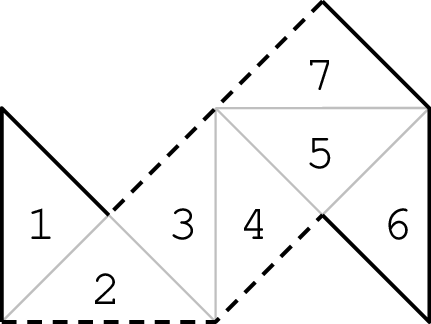}\tabularnewline
\end{tabular}
\par\end{centering}

\caption{Isospectral broken Gordon-Webb-Wolpert drums, where the first (second)
numbers refer to the case in which solid lines indicate Dirichlet
(Neumann) and dashed ones Neumann (Dirichlet) boundary conditions.\label{fig:IsospectralBrokenDrumsShapedLikeGordon}}
\end{figure}

\begin{table}
\noindent \begin{centering}
\begin{tabular}{ccccccc}
 & \multicolumn{2}{c}{Loop-signed} & \multicolumn{2}{c}{Transplantable} & \multicolumn{2}{c}{Transplantable}\tabularnewline
$V$ & graphs & (treelike) & pairs & (treelike) & classes & (treelike)\tabularnewline
\hline 
$2$ & $40$ & $(30)$ & $9$ & $(6)$ & $3$ & $(2)$\tabularnewline
$3$ & $128$ & $(96)$ & $0$ & $(0)$ & $0$ & $(0)$\tabularnewline
$4$ & $737$ & $(472)$ & $118$ & $(64)$ & $28$ & $(18)$\tabularnewline
$5$ & $3\,848$  & $(2\,304)$ & $0$ & $(0)$ & $0$ & $(0)$\tabularnewline
$6$ & $24\,360$  & $(12\,792)$ & $957$ & $(294)$ & $176$ & $(56)$\tabularnewline
$7$ & $156\,480$ & $(73\,216)$ & $112$ & $(112)$ & $32$ & $(32)$\tabularnewline
$8$ & $1\,076\,984$ & $(439\,968)$ & $13\,349$ & $(2\,112)$ & $2\,343$ & $(375)$\tabularnewline
$9$ & $7\,625\,040$ & $(2\,715\,648)$ & $0$ & $(0)$ & $0$ & $(0)$\tabularnewline
\end{tabular}
\par\end{centering}

\caption{Transplantable connected loop-signed graphs with $3$ edge colors.
The last $2$ columns contain the number of equivalence classes with
respect to the relation generated by renumberings of edge colors.\label{tab:ResultsThreeEdgeColours}}
\end{table}

\begin{table}
\noindent \begin{centering}
\begin{tabular}{c>{\centering}p{17mm}c>{\centering}p{8mm}>{\centering}p{8mm}>{\centering}p{8mm}>{\centering}p{8mm}>{\centering}p{8mm}>{\centering}p{8mm}}
 & \multicolumn{2}{c}{Edge-colored} & \multicolumn{2}{c}{Dirichlet} & \multicolumn{2}{c}{Neumann} & \multicolumn{2}{c}{Treelike}\tabularnewline
$V$ & graphs & trees & \multicolumn{2}{c}{pairs classes} & \multicolumn{2}{c}{pairs classes} & \multicolumn{2}{c}{pairs classes}\tabularnewline
\hline 
$7$ & $1\,407$ & $143$ & $7$ & $3$ & $7$ & $3$ & $7$ & $3$\tabularnewline
$8$ & $6\,877$ & $450$ & $64$ & $16$ & $28$ & $8$ & $0$ & $0$\tabularnewline
$9$ & $28\,665$ & $1\,326$ & $0$ & $0$ & $0$ & $0$ & $0$ & $0$\tabularnewline
$10$ & $142\,449$ & $4\,262$ & $0$ & $0$ & $0$ & $0$ & $0$ & $0$\tabularnewline
$11$ & $681\,467$ & $13\,566$ & $34$ & $9$ & $70$ & $19$ & $0$ & $0$\tabularnewline
$12$ & $3\,535\,172$ & $44\,772$ & $2\,362$ & $440$ & $42$ & $10$ & $0$ & $0$\tabularnewline
$13$ & $18\,329\,101$ & $148\,580$ & $26$ & $9$ & $26$ & $9$ & $26$ & $9$\tabularnewline
$14$ & $99\,531\,092$ & $502\,101$ & $345$ & $77$ & $798$ & $163$ & $42$ & $7$\tabularnewline
$15$ & $546\,618\,491$ & $1\,710\,855$ & $51$ & $13$ & $159$ & $33$ & $15$ & $4$\tabularnewline
\end{tabular}
\par\end{centering}

\caption{Pairs and equivalence classes as in Table~\ref{tab:ResultsThreeEdgeColours}
but with uniform loop signs (no such pair with $\NumberOfVertices<7$).
\label{tab:ResultsThreeEdgeColoursHomo}}
\end{table}

\begin{figure}
\noindent \begin{centering}
\psfragBodySmall\psfrag{c}{\raisebox{0mm}{$\Graph$}} \psfrag{C}{\raisebox{-1mm}{$\OnSecond{\Graph}$}}\psfrag{a}{\hspace{1mm}\raisebox{-1.5mm}{$\Graph_{S}$}} \psfrag{b}{\hspace{-2mm}\raisebox{-1.5mm}{$\OnSecond{\Graph}_{S}$}} \psfrag{A}{\hspace{-2mm}\raisebox{0mm}{\Large{$\Rightarrow$}}} \psfrag{B}{\hspace{-2mm}\raisebox{0mm}{\Large{$\Leftarrow$}}}\raisebox{9mm}{\includegraphics[scale=0.3]{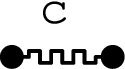}}\hspace{-10.5mm}\includegraphics[scale=0.35]{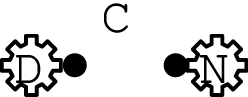}\hspace{5mm}\includegraphics[scale=0.3]{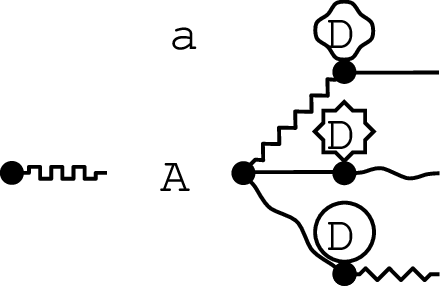}\hspace{0.3cm}\includegraphics[scale=0.3]{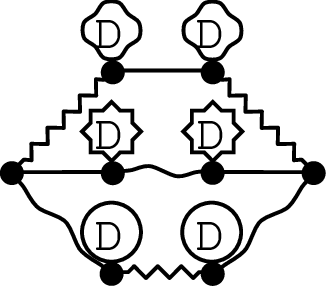}\hspace{0.3cm}\includegraphics[scale=0.3]{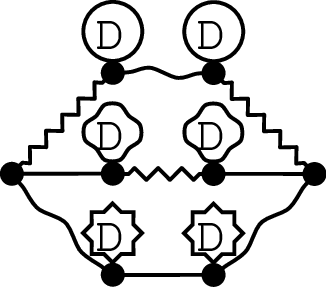}\hspace{0.3cm}\includegraphics[scale=0.3]{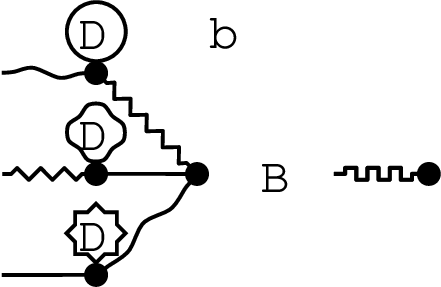}
\par\end{centering}

\caption{A self-dual pair giving rise to transplantable graphs with Dirichlet
loop signs.\label{fig:Diamond_8V}}
\end{figure}

In the following, we demonstrate how the $6$ self-dual pairs in Figure~\ref{fig:Appendix-4-vertices-per-graph}
in the appendix give rise to $6$ of the $8$ Neumann pairs with $8$
vertices per graph. The remaining $2$ pairs are shown in Figure~\ref{fig:OrientableNonorientableNeumannBC}.
We consider the transplantable pair $(\Graph,\OnSecond{\Graph})$
and the self-dual pair $(\Graph_{S},\OnSecond{\Graph}_{S})$ in Figure~\ref{fig:Diamond_8V},
where outward-pointing edges indicate Neumann loops. According to
the Substitution Theorem, $\SubstitutedGraph{\Graph}{\Graph_{S}}$
and $\SubstitutedGraph{\OnSecond{\Graph}}{\Graph_{S}}$ are transplantable
as well as $\SubstitutedGraph{\Graph}{\OnSecond{\Graph}_{S}}$ and
$\SubstitutedGraph{\OnSecond{\Graph}}{\OnSecond{\Graph}_{S}}$, where
the indicated loop assignments shall be used. Note that $\SubstitutedGraph{\OnSecond{\Graph}}{\Graph_{S}}$
and $\SubstitutedGraph{\OnSecond{\Graph}}{\OnSecond{\Graph}_{S}}$
have an identical component with Dirichlet loops only. Omitting these
components leaves us with the transplantable pair $(\Graph_{S},\OnSecond{\Graph}_{S})$
we started with. Hence, $\SubstitutedGraph{\OnSecond{\Graph}}{\Graph_{S}}$
and $\SubstitutedGraph{\OnSecond{\Graph}}{\OnSecond{\Graph}_{S}}$
are transplantable for which reason $\SubstitutedGraph{\Graph}{\Graph_{S}}$
and $\SubstitutedGraph{\Graph}{\OnSecond{\Graph}_{S}}$ are transplantable.
Since these graphs have bipartite loopless versions, we can pass to
the dual pair having Neumann loops only. In the same way, some of
the pairs with $14$ vertices per graph and uniform loop signs come
from graphs with $7$ vertices and mixed loop signs.

\newpage{}

\section{Inaudible properties\label{sec:InaudibleProps}}

\subsection{One cannot hear which parts of a drum are broken\label{sub:SelfDualAndAlmostSelfDualPairs}}

The transplantable half-disks in Figure~\ref{fig:IsospectralHalfDisks}
were used to show that two broken versions of a drum with drumheads
attached exactly where the other version's drumhead is free can sound
the same~\cite{JakobsonLevitinNadirashviliPolterovich2006}. These
domains come from a self-dual pair, which in turn arises from the
pair with $2$ edge colors and $2$ vertices per graph as indicated
in Figure~\ref{fig:Appendix-4-vertices-per-graph} in the appendix.
Figure~\ref{fig:SelfdualPairs} presents another self-dual\emph{
}pair giving rise to domains with a single Dirichlet boundary component.

\begin{figure}[b]
\begin{centering}
\subfloat[Isospectral half-disks~\cite{JakobsonLevitinNadirashviliPolterovich2006}.\label{fig:IsospectralHalfDisks}]{\noindent \begin{centering}
\begin{minipage}[t]{65mm}%
\noindent \begin{center}
\psfragBody%
\begin{tabular}{cc}
\noalign{\vskip4mm}
\includegraphics[scale=0.42]{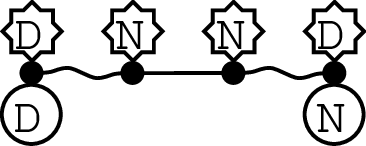} & \includegraphics[scale=0.48]{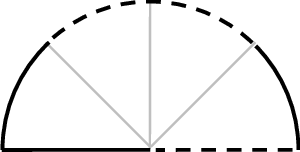}\tabularnewline
\includegraphics[scale=0.42]{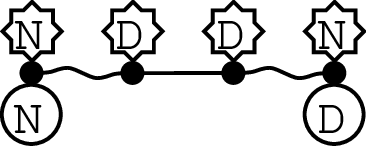} & \includegraphics[scale=0.48]{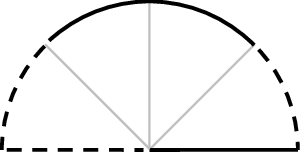}\tabularnewline
\end{tabular}
\par\end{center}%
\end{minipage}
\par\end{centering}

\noindent \centering{}}\hspace{4mm}\subfloat[Self-dual pair\label{fig:SelfdualPairs}]{\begin{centering}
\begin{minipage}[t]{42mm}%
\begin{center}
\psfragBody%
\begin{tabular}{cc}
\includegraphics[scale=0.42]{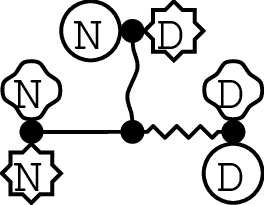} & \includegraphics[scale=0.48]{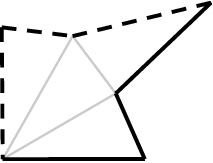}\tabularnewline
\includegraphics[scale=0.42]{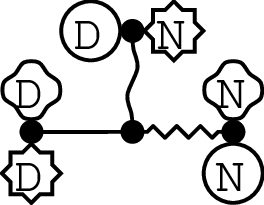} & \includegraphics[scale=0.48]{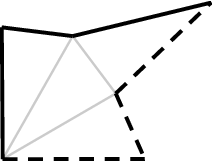}\tabularnewline
\end{tabular}
\par\end{center}%
\end{minipage}
\par\end{centering}

}
\par\end{centering}

\caption{Self-dual pairs and domains, whose spectra are invariant under swapping
all boundary conditions.}
\end{figure}

\subsection{One cannot hear the fundamental group of a broken drum}

Isospectral manifolds with different fundamental groups first appeared
in~\cite{Vign'eras1980,Vign'eras1980a}. Planar examples with mixed
boundary conditions were presented in~\cite{LevitinParnovskiPolterovich2006}.
Similar examples can be obtained from Figure~\ref{fig:Appendix-4-vertices-per-graph}
in the appendix. In contrast, transplantable connected graphs with
uniform loop signs have loopless versions with isomorphic fundamental
groups, which can be seen as follows. If a connected loop-signed graph
with either Dirichlet or Neumann loops is given by the $\NumberOfVertices\times\NumberOfVertices$
adjacency matrices $(\AdjacencyMatrix^{\Colour})_{\Colour=1}^{\NumberOfColours}$,
then the number~$\NumberOfVertices$ of vertices and the number~$E$
of edges that belong to its loopless version are determined by expressions
of the form~(\ref{eq:TraceCondition}), namely, 
\[
\NumberOfVertices=\Trace(I_{\NumberOfVertices})\qquad\text{and}\qquad E=\frac{1}{2}\sum_{\Colour=1}^{\NumberOfColours}(\Trace(I_{\NumberOfVertices})\pm\Trace(\AdjacencyMatrix^{\Colour})).
\]

\subsection{One cannot hear whether a broken drum is orientable\label{sub:Orientability}}

Doyle and Rossetti \cite{DoyleRossetti2008} showed that orientability
of closed hyperbolic surfaces can be heard. In contrast, B\'erard
and Webb~\cite{B'erardWebb1995} constructed a pair of Neumann but
not Dirichlet isospectral manifolds one of which is orientable while
the other is not. Their example corresponds to the pairs in Figure~\ref{fig:OrientableNonorientableNeumannBC},
which can be obtained from each other by braiding. Table~\ref{tab:OriNonoriNeuPair}
provides a similar pair with $15$ instead of $8$ vertices per graph,
and Figure~\ref{fig:OrientableNonorientableMixedBC} shows the first
example of this kind with mixed loop signs. These manifolds have different
numbers of Dirichlet boundary components, and one can alter the building
block continuously so that the first one becomes planar. In contrast,
the following observation by Doyle~\cite{Doyle} raises the question
whether orientability is encoded in the Dirichlet spectrum of a connected
manifold.

\begin{figure}
\noindent \centering{}\psfragBody\includegraphics[scale=0.42]{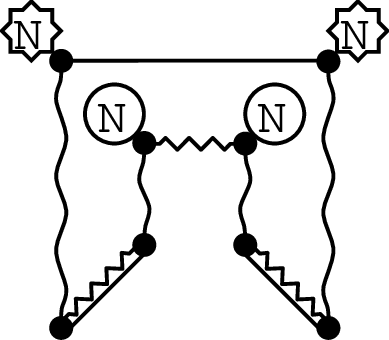}\hspace{1mm}\includegraphics[scale=0.42]{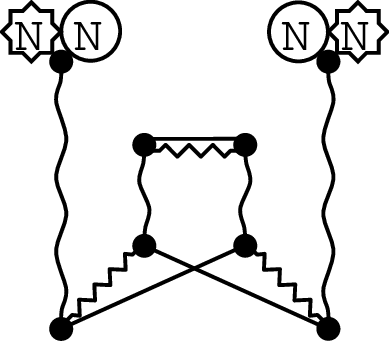}\hspace{5mm}\includegraphics[scale=0.42]{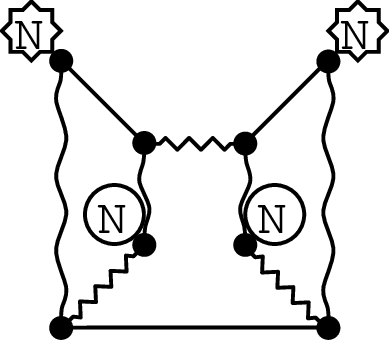}\hspace{1mm}\includegraphics[scale=0.42]{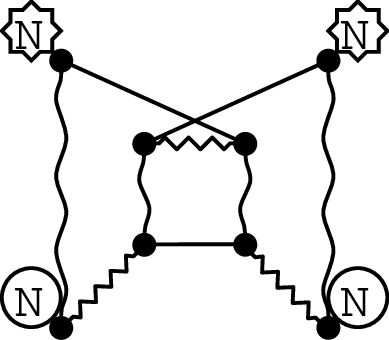}\caption{Transplantable pairs. Only the respective first graph has a bipartite
loopless version.\label{fig:OrientableNonorientableNeumannBC}}
\end{figure}

\begin{table}
\noindent \begin{centering}
\begin{tabular}{l|l}
Bipartite graph & Non-bipartite graph\tabularnewline
\hline 
{\small{(1,8)(3,10)(5,12)(7,14)}} & {\small{(1,8)(3,10)(5,12)(7,14)}}\tabularnewline
{\small{(1,2)(4,7)(8,14)(9,12)(10,15)(11,13)}} & {\small{(1,6)(2,13)(3,11)(4,12)(5,10)(9,14)}}\tabularnewline
{\small{(1,6)(2,12)(3,10)(4,13)(5,11)(8,15)}} & {\small{(1,12)(2,9)(3,5)(4,15)(7,10)(8,14)}}\tabularnewline
\end{tabular}
\par\end{centering}

\caption{Colored links of transplantable graphs with $15$ vertices, all of
whose loops carry Neumann signs.\label{tab:OriNonoriNeuPair}}
\end{table}

\begin{figure}
\noindent \begin{centering}
\psfragBody\psfrag{1}{\raisebox{0.5mm}{$1$}} \psfrag{2}{\raisebox{0.5mm}{\hspace{-0.2mm}$2$}} \psfrag{3}{\raisebox{0.5mm}{\hspace{-0.5mm}$3$}} \psfrag{4}{\raisebox{0.5mm}{\hspace{-0.5mm}$4$}}
\psfrag{A}{\raisebox{1mm}{\hspace{0mm}$M_1$}}
\psfrag{B}{\raisebox{1mm}{\hspace{0mm}$M_2$}}
\psfrag{C}{\raisebox{1mm}{\hspace{0mm}$M_3$}}%
\begin{tabular}{>{\raggedright}p{27mm}>{\centering}p{50mm}c}
\includegraphics[scale=0.42]{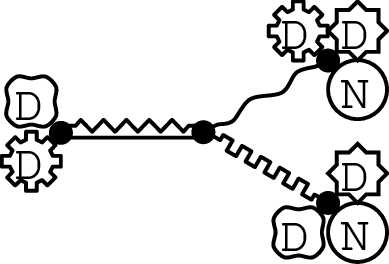} & \includegraphics[scale=0.48]{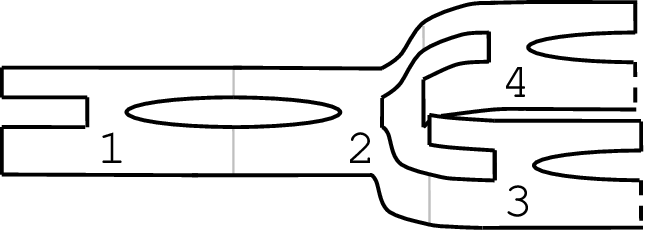} & \multirow{2}{*}{%
\begin{tabular}{l}
\noalign{\vskip-9mm}
\setlength{\arraycolsep}{1pt}$\begin{array}{cc}
 & T=T^{T}=T^{-1}\\
\frac{1}{2}\hspace{-1.5mm}\, & \left(\begin{array}{cccc}
1 & -1 & 1 & 1\\
-1 & 1 & 1 & 1\\
1 & 1 & 1 & -1\\
1 & 1 & -1 & 1
\end{array}\right)
\end{array}$\setlength{\arraycolsep}{\myArraycolsep}\tabularnewline[10mm]
\end{tabular}}\tabularnewline
\noalign{\vskip1mm}
\includegraphics[scale=0.42]{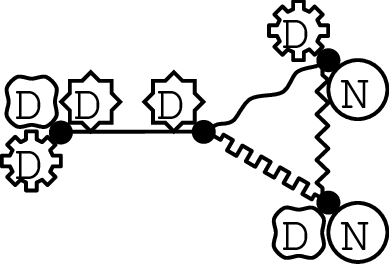} & \includegraphics[scale=0.48]{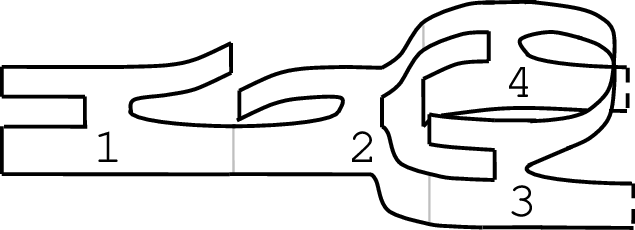} & \tabularnewline
\noalign{\vskip2mm}
\end{tabular}
\par\end{centering}

\caption{Transplantable manifolds with mixed boundary conditions, only one
of which is orientable. \label{fig:OrientableNonorientableMixedBC}}
\end{figure}

\begin{prop}
If two connected loop-signed graphs without Neumann loops are transplantable,
then their loopless versions are either both bipartite or non-bipartite.\label{prop:Orientability_of_Graphs_without_Neumann_Loops}\end{prop}
\begin{proof}
Let $\Graph$ and $\OnSecond{\Graph}$ be transplantable loop-signed
graphs with $\NumberOfVertices\times\NumberOfVertices$ adjacency
matrices $(\AdjacencyMatrix^{\Colour})_{\Colour=1}^{\NumberOfColours}$
and $(\OnSecond{\AdjacencyMatrix}^{\Colour})_{\Colour=1}^{\NumberOfColours}$
having non-positive diagonal entries. If $\Graph$ is loopless, that
is, $\Trace(\AdjacencyMatrix^{\Colour})=0$ for $\Colour=1,2,\ldots,\NumberOfColours$,
then the Trace Theorem shows that $\OnSecond{\Graph}$ cannot have
loops either. If, in addition, $\Graph$ has an odd cycle with associated
sequence of edge colors $\Colour_{1}\Colour_{2}\ldots\Colour_{2n+1}$,
then $0<\Trace(\AdjacencyMatrix^{\Colour_{1}}\AdjacencyMatrix^{\Colour_{2}}\cdots\AdjacencyMatrix^{\Colour_{2n+1}})=\Trace(\OnSecond{\AdjacencyMatrix}^{\Colour_{1}}\OnSecond{\AdjacencyMatrix}^{\Colour_{2}}\cdots\OnSecond{\AdjacencyMatrix}^{\Colour_{2n+1}})$,
and $\OnSecond{\Graph}$ is non-bipartite as well. Hence, we may assume
that both $\Graph$ and $\OnSecond{\Graph}$ have loops. In the following,
we consider the Markov chain with $\NumberOfVertices\times\NumberOfVertices$
transition matrix $P$ with entries
\[
P_{ij}=\frac{1}{\NumberOfColours}\sum_{\Colour=1}^{\NumberOfColours}|\AdjacencyMatrix_{ij}^{\Colour}|.
\]
It represents a random walk on the vertices of $\Graph$, where at
each step one of the $\NumberOfColours$ incident links or loops is
chosen with equal probability. Since $\Graph$ is connected and has
loops, $P$ is irreducible, aperiodic and has the invariant distribution
$\NumberOfVertices^{-1}(1,1,\ldots,1)$ resulting in convergence to
equilibrium~\cite[Theorem 1.8.3]{Norris1998} 
\[
\lim_{k\to\infty}(P^{k})_{ij}=\NumberOfVertices^{-1}.
\]
If the loopless version of $\Graph$ is bipartite, then every product
of adjacency matrices of $\Graph$ with an even number of factors
is a signed permutation matrix with non-negative entries on the diagonal.
Hence, 
\begin{equation}
1=\Trace\Bigl(\lim_{k\to\infty}P^{2k}\Bigr)=\lim_{k\to\infty}\Trace(P^{2k})=\lim_{k\to\infty}\NumberOfColours^{-2k}\Trace\biggl(\sum_{\Colour=1}^{\NumberOfColours}\AdjacencyMatrix^{\Colour}\biggr)^{2k}.\label{eq:LimitOfTracePowers}
\end{equation}
If, however, the loopless version of $\Graph$ has an odd cycle, then
we can find $L\in\mathbb{N}$ such that for any pair $(i,j)$ of vertices
of $\Graph$, there exist two paths of length $L$ from $i$ to $j$,
such that one uses an odd number of loops whereas the other uses an
even number of loops. For $2k>L$, we partition the set of $2k$-cycles
into equivalence classes such that cycles are equivalent if their
initial subpaths of length $L$ start and end at the same vertices.
Note that there are at most $\NumberOfColours^{L}$ paths of length
$L$ between any two vertices of $\Graph$ which shows that in each
class, the fraction of cycles starting with one of the $2\NumberOfVertices^{2}$
above-mentioned paths is at least $2\NumberOfColours^{-L}$. Since
half of these $2k$-cycles use an odd number of loops, the fraction
of such $2k$-cycles in $\Graph$ is at least $\NumberOfColours^{-L}$,
independent of~$k$. Moreover, the first equality in~(\ref{eq:LimitOfTracePowers})
remains true, that is, the number of $2k$-cycles in $\Graph$ times
$C^{-2k}$ converges to $1$. Thus, the right-hand side of~(\ref{eq:LimitOfTracePowers})
is bounded by $1-\NumberOfColours^{-L}$, which completes the proof
by virtue of the Trace Theorem.
\end{proof}

\subsection{One cannot hear whether a drum is connected\label{sub:NbrOfComp}}

It is well-known that the number of components of a manifold with
pure Neumann boundary conditions equals the multiplicity of $0$ as
an eigenvalue. In contrast, Figure~\ref{fig:Dirichlet_Drums_With_Different_Connectivity}
presents the first known pair of Dirichlet isospectral manifolds with
different numbers of components. This is also the first known pair
of Dirichlet isospectral manifolds that are not Neumann isospectral.
The graphs in Figure~\ref{fig:OrientableNonorientableMixedBC} give
rise to the same manifolds if one uses an appropriate building block.
In view of the pair in Figure~\ref{fig:Dirichlet_Drums_With_Different_Connectivity},
note that if two transplantable graphs without Neumann loops have
different numbers of components, then their loopless versions cannot
both be bipartite, for if they were, their dual pairs would give rise
to Neumann isospectral manifolds with different numbers of components.

\begin{figure}
\noindent \begin{centering}
\psfragBody\psfrag{1}{\raisebox{0.5mm}{$1$}} \psfrag{2}{\raisebox{0.5mm}{\hspace{-0.2mm}$2$}} \psfrag{3}{\raisebox{0.5mm}{\hspace{-0.5mm}$3$}} \psfrag{4}{\raisebox{0.5mm}{\hspace{-0.5mm}$4$}}
\psfrag{A}{\raisebox{1mm}{\hspace{0mm}$M_1$}}
\psfrag{B}{\raisebox{1mm}{\hspace{0mm}$M_2$}}
\psfrag{C}{\raisebox{1mm}{\hspace{0mm}$M_3$}}%
\begin{tabular}{>{\raggedright}p{27mm}>{\centering}p{50mm}c}
\includegraphics[scale=0.42]{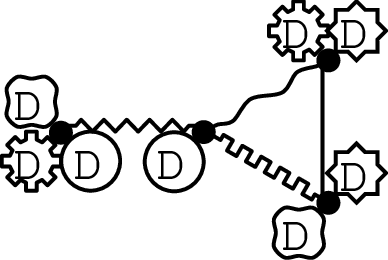} & \includegraphics[scale=0.48]{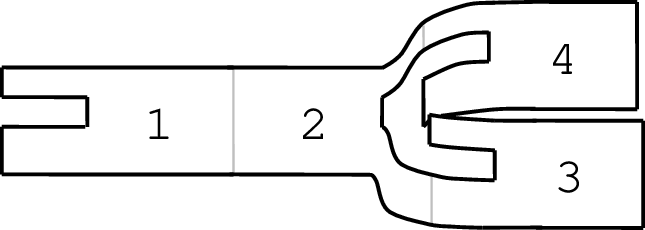} & \multirow{2}{*}{%
\begin{tabular}{l}
\noalign{\vskip-9mm}
\setlength{\arraycolsep}{1pt}$\begin{array}{cc}
 & T=T^{T}=T^{-1}\\
\frac{1}{2}\hspace{-1.5mm}\, & \left(\begin{array}{cccc}
1 & -1 & 1 & 1\\
-1 & 1 & 1 & 1\\
1 & 1 & 1 & -1\\
1 & 1 & -1 & 1
\end{array}\right)
\end{array}$\setlength{\arraycolsep}{\myArraycolsep}\tabularnewline[10mm]
\end{tabular}}\tabularnewline
\noalign{\vskip1mm}
\includegraphics[scale=0.42]{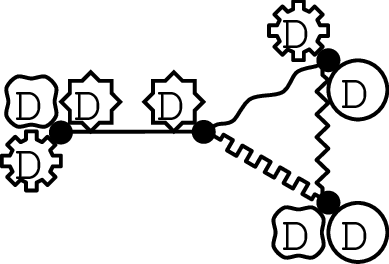} & \includegraphics[scale=0.48]{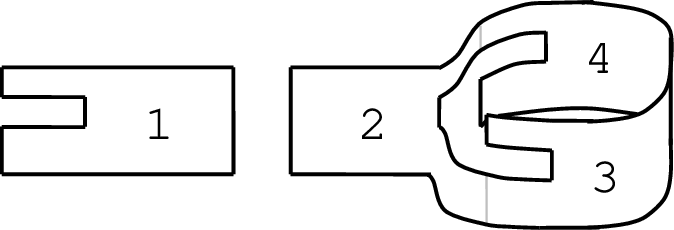} & \tabularnewline
\noalign{\vskip2mm}
\end{tabular}
\par\end{centering}

\caption{Transplantable manifolds with pure Dirichlet boundary conditions,
only one of which is connected. \label{fig:Dirichlet_Drums_With_Different_Connectivity}}
\end{figure}

\begin{figure}
\noindent \begin{centering}
\psfragBody\psfrag{1}{\raisebox{0.5mm}{$1$}} \psfrag{2}{\raisebox{0.5mm}{\hspace{-0.2mm}$2$}} \psfrag{3}{\raisebox{0.5mm}{\hspace{-0.5mm}$3$}} \psfrag{4}{\raisebox{0.5mm}{\hspace{-0.5mm}$4$}}
\psfrag{A}{\raisebox{1mm}{\hspace{0mm}$M_1$}}
\psfrag{B}{\raisebox{1mm}{\hspace{0mm}$M_2$}}
\psfrag{C}{\raisebox{1mm}{\hspace{0mm}$M_3$}}%
\begin{tabular}{>{\raggedright}p{27mm}>{\centering}p{50mm}c}
\includegraphics[scale=0.42]{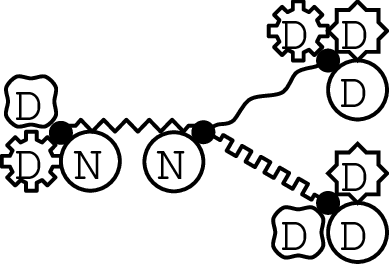} & \includegraphics[scale=0.48]{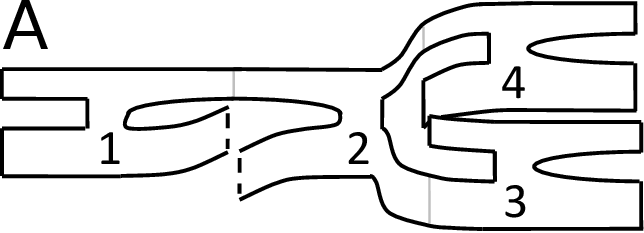} & \multirow{3}{*}{%
\begin{tabular}{l}
\noalign{\vskip-10mm}
\setlength{\arraycolsep}{2pt}$\begin{array}{cc}
 & T=T^{T}=T^{-1}\\
\hspace{-0.2mm}\frac{1}{2}\hspace{-2.3mm}\, & \left(\begin{array}{cccc}
1 & 1 & -1 & -1\\
1 & 1 & 1 & 1\\
-1 & 1 & 1 & -1\\
-1 & 1 & -1 & 1
\end{array}\right)
\end{array}$\setlength{\arraycolsep}{\myArraycolsep}\tabularnewline[12.5mm]
\multirow{1}{*}{\setlength{\arraycolsep}{1pt}$\begin{array}{cc}
\frac{1}{2}\hspace{-1.5mm}\, & \left(\begin{array}{cccc}
1 & -1 & 1 & 1\\
-1 & 1 & 1 & 1\\
1 & 1 & 1 & -1\\
1 & 1 & -1 & 1
\end{array}\right)\end{array}$\setlength{\arraycolsep}{\myArraycolsep}}\tabularnewline
\end{tabular}}\tabularnewline
\noalign{\vskip1mm}
\includegraphics[scale=0.42]{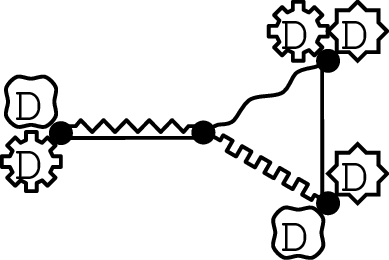} & \includegraphics[scale=0.48]{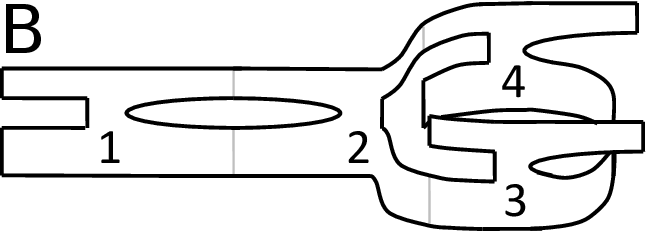} & \tabularnewline
\noalign{\vskip1mm}
\includegraphics[scale=0.42]{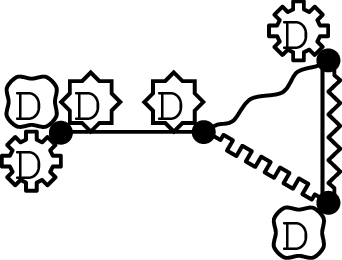} & \includegraphics[scale=0.48]{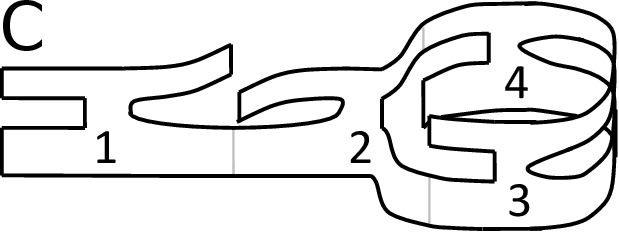} & \tabularnewline
\end{tabular}
\par\end{centering}

\caption{Transplantable manifolds, only one of which carries mixed boundary
conditions. To put it another way, a~broken\emph{ }drum that sounds
unbroken.\label{fig:ABrokenDrumThatSoundsUnbroken}}
\end{figure}

\subsection{One cannot hear whether a drum is broken\label{sub:BrokennessAndIsotropyOrder}}

We finish with the transplantable triple shown in Figure~\ref{fig:ABrokenDrumThatSoundsUnbroken}.
This is the first example of a connected manifold with mixed boundary
conditions that is isospectral to a connected manifold with pure Dirichlet
boundary conditions. Note that the building block could be altered
continuously so that $M_{1}$ becomes planar, whereas $M_{2}$ and
$M_{3}$ contain M\"obius strips. In contrast, a manifold with pure
Neumann boundary conditions has $0$ as an eigenvalue and is therefore
never isospectral, and thus never transplantable, to a connected manifold
with mixed boundary conditions. Hence, if a connected loop-signed
graph has a Dirichlet loop, then any graph it is transplantable to
must have one, which can also be shown using the technique of Proposition~\ref{prop:Orientability_of_Graphs_without_Neumann_Loops}.
Moreover, $M_{2}$ and $M_{3}$ have the same heat content, which
follows from the existence of a unitary transplantation matrix each
of whose columns sums to $1$~\cite{Band}. They constitute the first
example of Dirichlet isospectral connected flat manifolds with this
property. In the style of~\cite{GordonWebbWolpert1992}, one can
interpret $M_{1}$ as an orbifold with Dirichlet boundary by gluing
two copies of $M_{1}$ along their Neumann boundary parts and then
taking the quotient with respect to the involution given by interchanging
of the copies. In conclusion, the triple shows that the number of
Neumann boundary components is not spectrally determined, and that
an orbifold can be Dirichlet isospectral to a manifold.

\setcounter{section}{0} \renewcommand{\thesection}{\Alph{section}} 

\newpage{}

\section{Graph gallery\label{cha:PictureGallery}}

\vspace{-2mm}

\begin{figure}[H]
\noindent \begin{centering}
\psfragAppendix%
\begin{tabular}{>{\raggedright}p{13mm}>{\raggedright}p{13mm}>{\raggedright}p{18mm}>{\raggedright}p{18mm}>{\raggedright}p{18mm}l}
1 & 2 & 3 & 4 & 5 & \multicolumn{1}{l}{6\hspace{2mm}\raisebox{-2mm}{\includegraphics[scale=0.3]{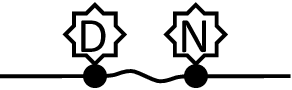}}}\tabularnewline[2mm]
\noalign{\vskip-1mm}
\includegraphics[scale=0.3]{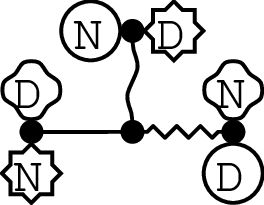} & \includegraphics[scale=0.3]{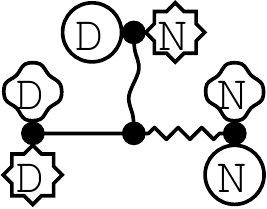} & \includegraphics[scale=0.3]{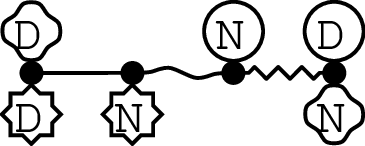} & \includegraphics[scale=0.3]{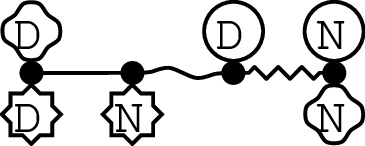} & \includegraphics[scale=0.3]{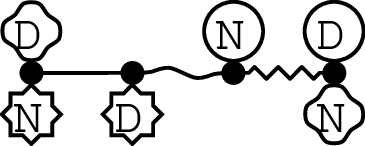} & \multicolumn{1}{l}{\includegraphics[scale=0.3]{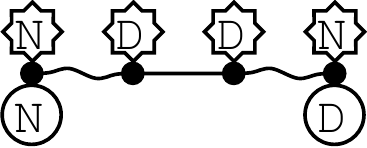}}\tabularnewline[2mm]
\noalign{\vskip-1mm}
\includegraphics[scale=0.3]{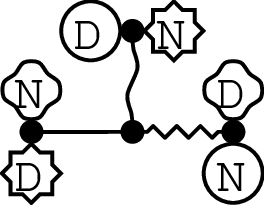} & \includegraphics[scale=0.3]{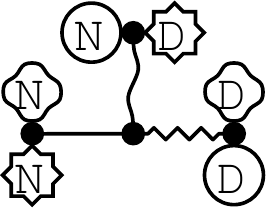} & \includegraphics[scale=0.3]{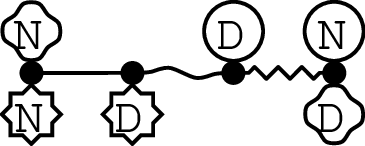} & \includegraphics[scale=0.3]{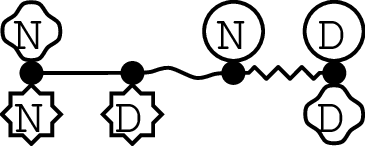} & \includegraphics[scale=0.3]{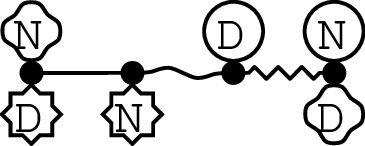} & \multicolumn{1}{>{\raggedright}p{18mm}}{\includegraphics[scale=0.3]{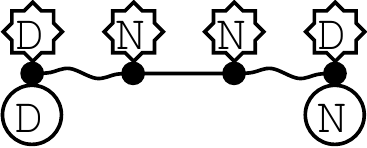}}\tabularnewline[2mm]
\noalign{\vskip-1mm}
\end{tabular}
\par\end{centering}

\noindent \begin{centering}
Self-dual treelike pairs
\par\end{centering}

\vspace{5mm}

\noindent \begin{centering}
\psfragAppendix%
\begin{tabular}{>{\raggedright}p{1mm}l>{\raggedright}p{1mm}>{\raggedright}p{4mm}r>{\raggedright}p{1mm}>{\raggedright}p{4mm}>{\raggedleft}p{15mm}>{\raggedright}p{1mm}>{\raggedright}p{7mm}l}
\raisebox{5mm}{7} & \includegraphics[scale=0.3]{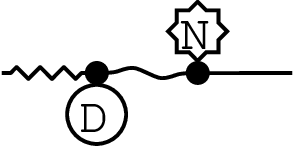} &  & \raisebox{5mm}{8} & \includegraphics[scale=0.3]{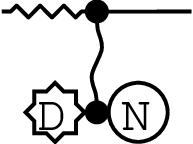} &  & \raisebox{5mm}{9/10} & \includegraphics[scale=0.3]{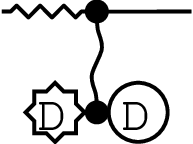} &  & \raisebox{5mm}{11/12} & \includegraphics[scale=0.3]{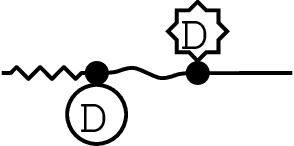}\tabularnewline
\multicolumn{2}{r}{\includegraphics[scale=0.3]{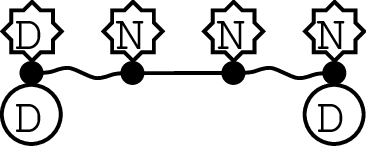}} &  & \multicolumn{2}{r}{\includegraphics[scale=0.3]{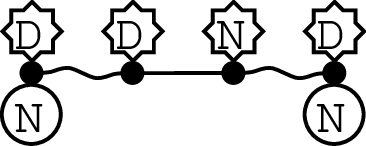}} &  & \multicolumn{2}{r}{\includegraphics[scale=0.3]{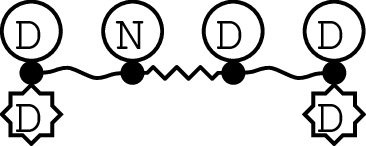}} &  & \multicolumn{2}{r}{\includegraphics[scale=0.3]{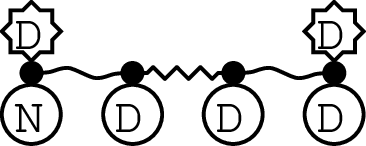}}\tabularnewline
\noalign{\vskip1mm}
\multicolumn{2}{r}{\includegraphics[scale=0.3]{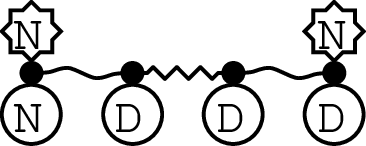}} &  & \multicolumn{2}{r}{\includegraphics[scale=0.3]{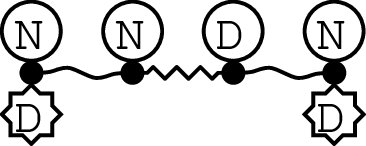}} &  & \multicolumn{2}{r}{\includegraphics[scale=0.3]{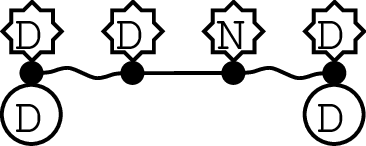}} &  & \multicolumn{2}{r}{\includegraphics[scale=0.3]{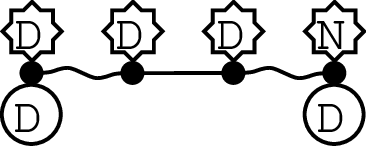}}\tabularnewline
\noalign{\vskip2mm}
\end{tabular}\\
\begin{tabular}{llllllll}
\noalign{\vskip1mm}
\raisebox{5mm}{13/14} & \includegraphics[scale=0.3]{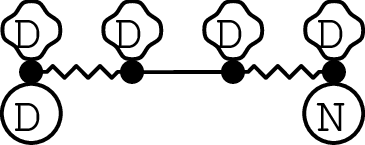} &  & \raisebox{5mm}{15/16} & \includegraphics[scale=0.3]{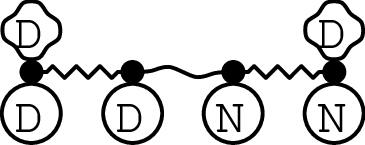} &  & \raisebox{5mm}{17/18} & \includegraphics[scale=0.3]{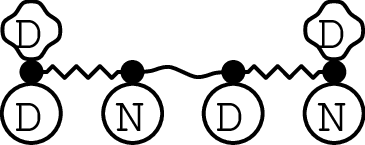}\tabularnewline
\noalign{\vskip1mm}
 & \includegraphics[scale=0.3]{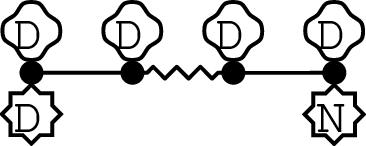} &  &  & \includegraphics[scale=0.3]{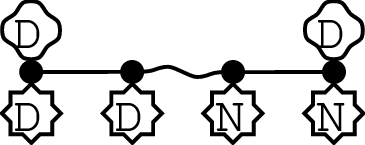} &  &  & \includegraphics[scale=0.3]{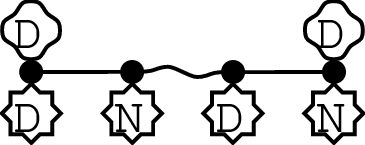}\tabularnewline
\noalign{\vskip2mm}
\end{tabular}
\par\end{centering}

\noindent \begin{centering}
Non-self-dual treelike pairs
\par\end{centering}

\noindent \begin{centering}
\begin{tabular}{c>{\centering}p{5mm}c}
\psfragAppendix\raisebox{1cm}{19}\hspace*{3mm}\includegraphics[scale=0.3]{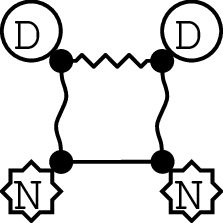}\hspace*{5mm}\includegraphics[scale=0.3]{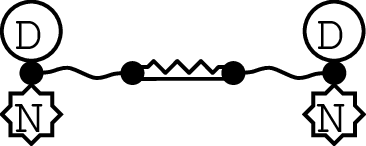} &  & \psfragAppendix\raisebox{7mm}{20}\includegraphics[scale=0.3]{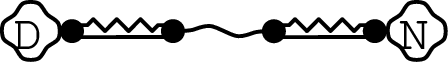}\hspace*{5mm}\includegraphics[scale=0.3]{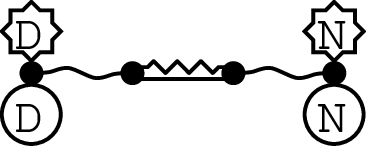}\tabularnewline[3mm]
\end{tabular}
\par\end{centering}

\noindent \begin{centering}
\psfragAppendix%
\begin{tabular}{rrrrr}
\noalign{\vskip-1mm}
\raisebox{5mm}{21}\hspace*{1mm}\includegraphics[scale=0.3]{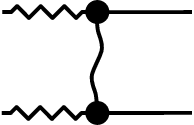} & \raisebox{5mm}{22}\hspace*{5mm}\includegraphics[scale=0.3]{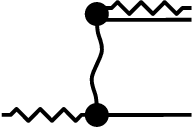} & \raisebox{5mm}{23/24}\hspace*{1mm}\includegraphics[scale=0.3]{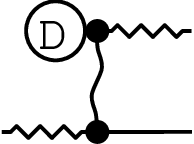} & \raisebox{5mm}{25/26}\hspace*{1mm}\includegraphics[scale=0.3]{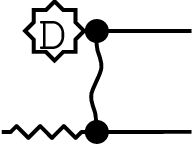} & \raisebox{5mm}{27/28}\includegraphics[scale=0.3]{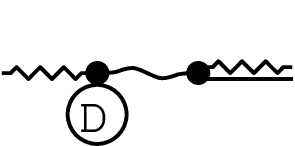}\tabularnewline
\includegraphics[scale=0.3]{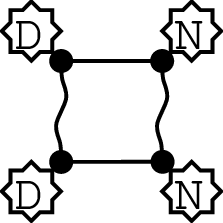} & \includegraphics[scale=0.3]{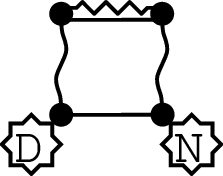} & \includegraphics[scale=0.3]{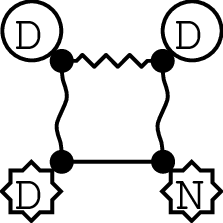} & \includegraphics[scale=0.3]{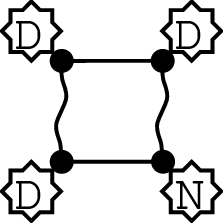} & \includegraphics[scale=0.3]{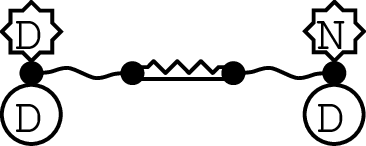}\tabularnewline
\includegraphics[scale=0.3]{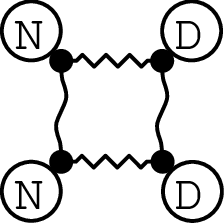} & \includegraphics[scale=0.3]{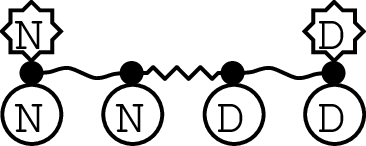} & \includegraphics[scale=0.3]{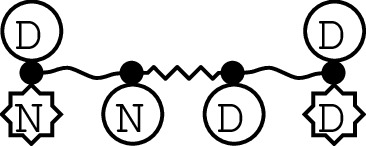} & \includegraphics[scale=0.3]{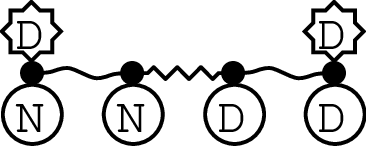} & \includegraphics[scale=0.3]{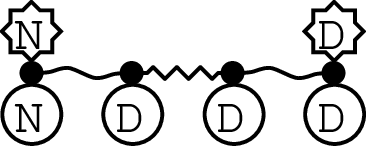}\tabularnewline[2mm]
\end{tabular}
\par\end{centering}

\noindent \begin{centering}
Non-treelike pairs
\par\end{centering}

\caption{Equivalence classes of transplantable pairs with $3$ edge colors
and $4$ vertices per graph, where second numbers refer to dual pairs.
The classes $13$, $14$ and $20$ arise from the class with $2$
edge colors and $4$ vertices per graph by adding or copying of an
edge color, respectively. The classes $6$ to $12$ and $21$ to $28$
arise from the class with $2$ edge colors and $2$ vertices per graph
by substitution which is indicated by suitable substituents.\label{fig:Appendix-4-vertices-per-graph}}
\end{figure}
\vfill{}
~

\begin{figure}[H]
\noindent \begin{centering}
\newcommand{\vraise}{9mm}\psfragAppendix%
\begin{tabular}{>{\centering}p{-1mm}c>{\centering}p{1mm}>{\centering}p{-1mm}cc>{\centering}p{-1mm}c}
\raisebox{\vraise}{1/2} & \includegraphics[scale=0.3]{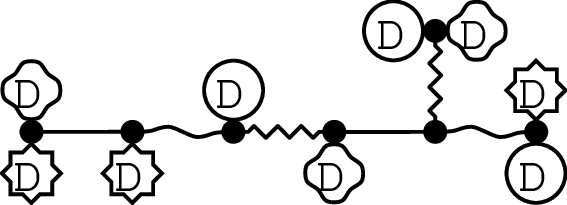} &  & \raisebox{\vraise}{3/4} & \includegraphics[scale=0.3]{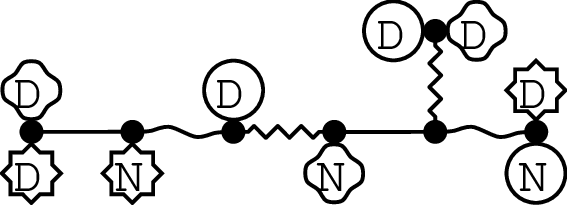} &  & \raisebox{\vraise}{5/6} & \includegraphics[scale=0.3]{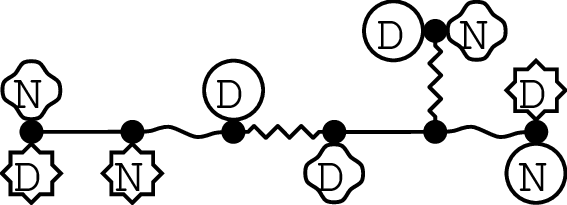}\tabularnewline
 & \includegraphics[scale=0.3]{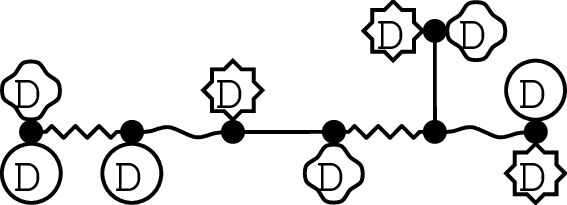} &  &  & \includegraphics[scale=0.3]{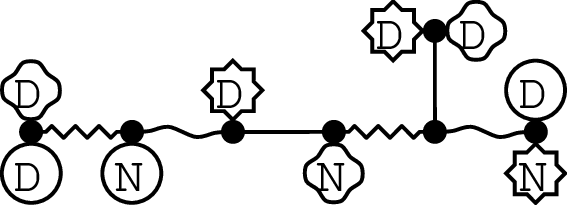} &  &  & \includegraphics[scale=0.3]{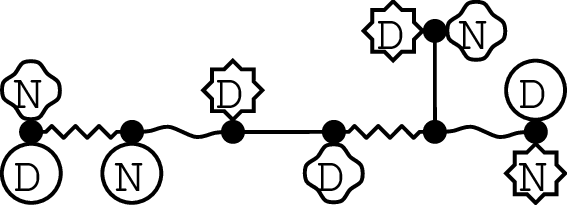}\tabularnewline[1mm]
\raisebox{\vraise}{7/8} & \includegraphics[scale=0.3]{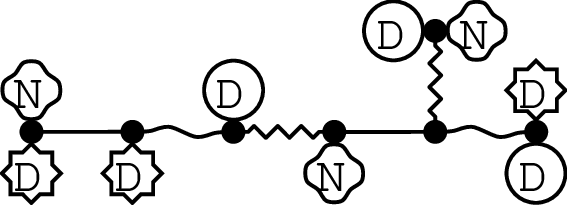} &  & \raisebox{\vraise}{9/10} & \includegraphics[scale=0.3]{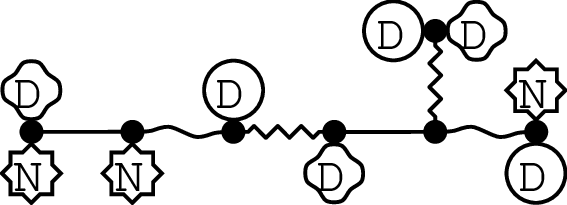} &  & \raisebox{\vraise}{11/12} & \includegraphics[scale=0.3]{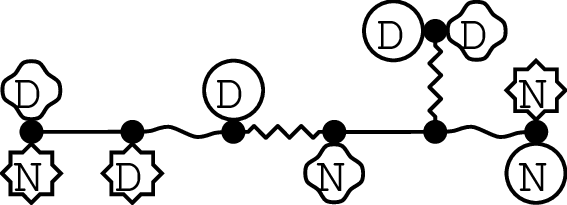}\tabularnewline
 & \includegraphics[scale=0.3]{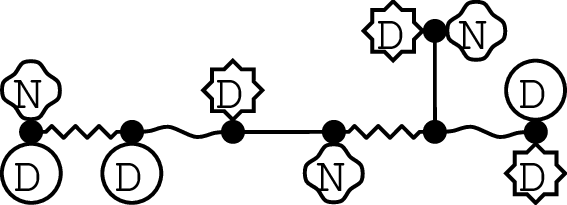} &  &  & \includegraphics[scale=0.3]{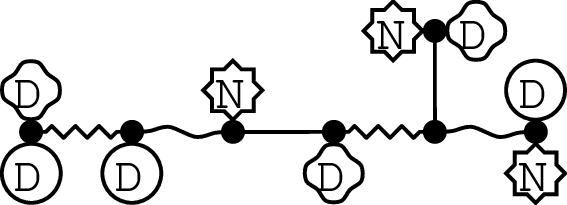} &  &  & \includegraphics[scale=0.3]{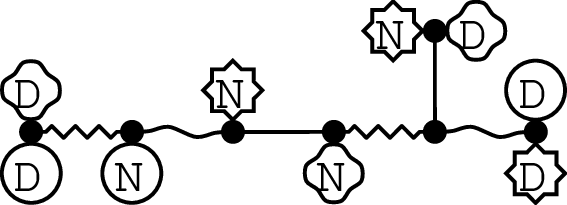}\tabularnewline[3mm]
\raisebox{\vraise}{13/14} & \includegraphics[scale=0.3]{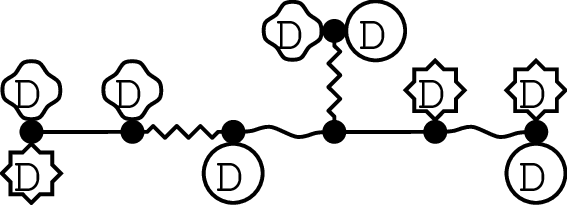} &  & \raisebox{\vraise}{15/16} & \includegraphics[scale=0.3]{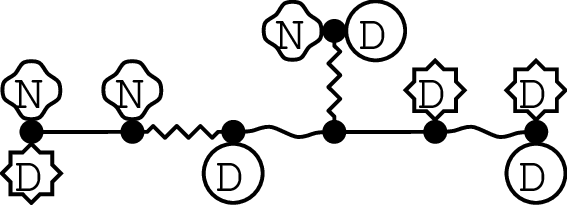} &  & \raisebox{\vraise}{17/18} & \includegraphics[scale=0.3]{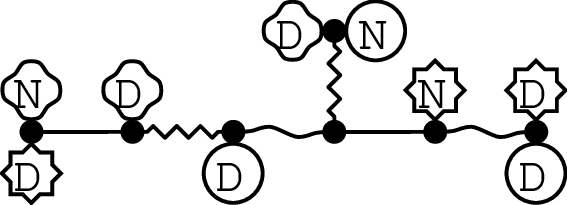}\tabularnewline
 & \includegraphics[scale=0.3]{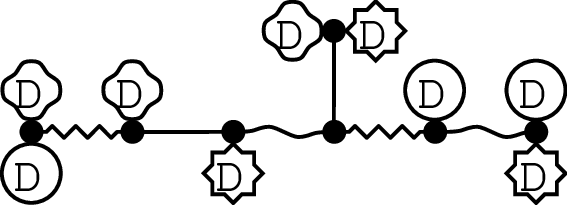} &  &  & \includegraphics[scale=0.3]{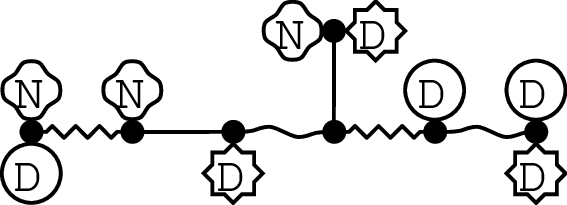} &  &  & \includegraphics[scale=0.3]{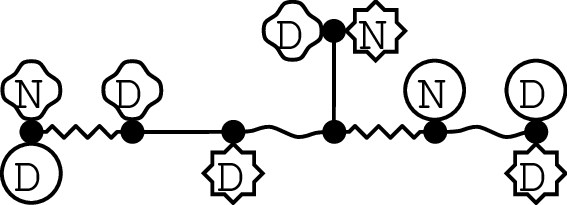}\tabularnewline[1mm]
\noalign{\vskip-3mm}
 &  &  &  &  &  &  & \tabularnewline
\raisebox{\vraise}{19/20} & \includegraphics[scale=0.3]{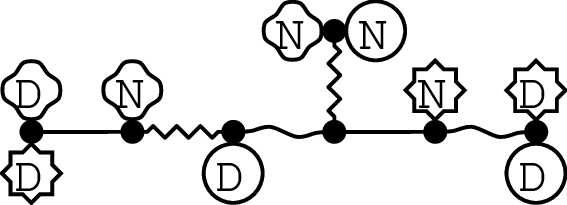} &  & \raisebox{\vraise}{21/22} & \includegraphics[scale=0.3]{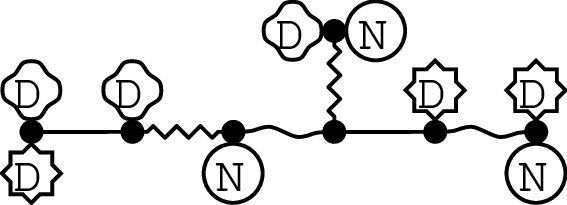} &  & \raisebox{\vraise}{23/24} & \includegraphics[scale=0.3]{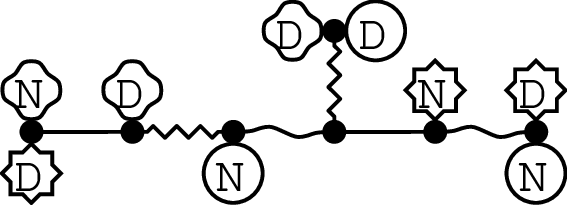}\tabularnewline
 & \includegraphics[scale=0.3]{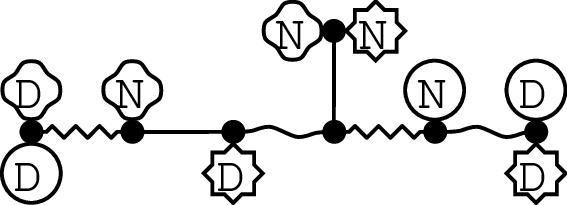} &  &  & \includegraphics[scale=0.3]{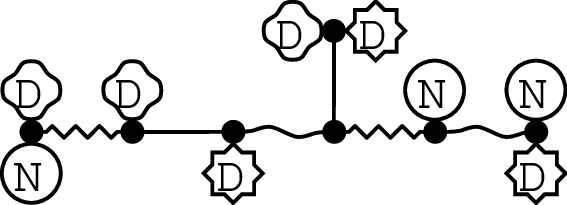} &  &  & \includegraphics[scale=0.3]{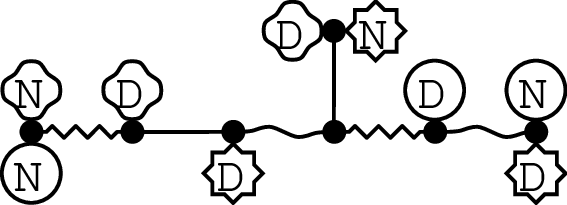}\tabularnewline[3mm]
\end{tabular}\\
\begin{tabular}{>{\raggedright}p{-1mm}>{\raggedright}p{20mm}>{\raggedright}p{-1mm}>{\raggedright}p{20mm}>{\raggedright}p{-1mm}>{\raggedright}p{20mm}>{\raggedright}p{-1mm}>{\raggedright}p{20mm}}
25/26 & \multirow{1}{20mm}{\includegraphics[scale=0.3]{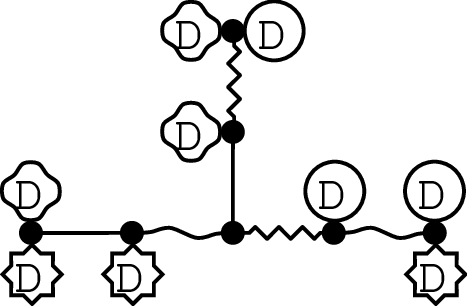}} & 27/28 & \multirow{1}{20mm}{\includegraphics[scale=0.3]{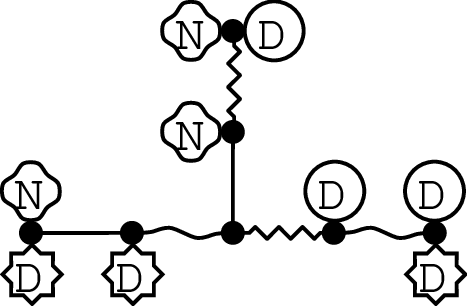}} & 29/30 & \multirow{1}{20mm}{\includegraphics[scale=0.3]{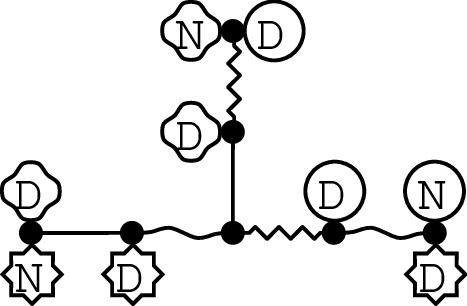}} & 31/32 & \multirow{1}{20mm}{\includegraphics[scale=0.3]{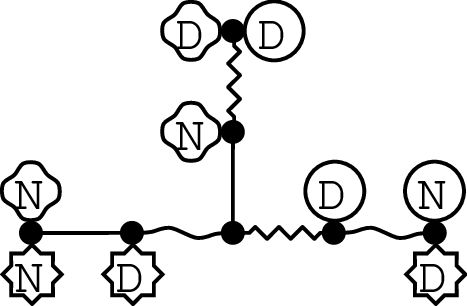}}\tabularnewline[13mm]
 & \includegraphics[scale=0.3]{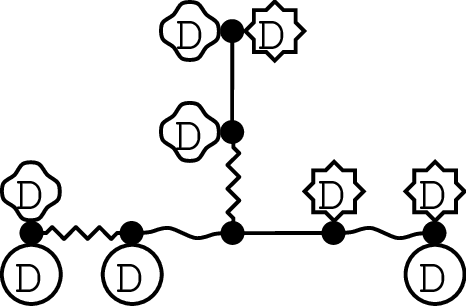} &  & \includegraphics[scale=0.3]{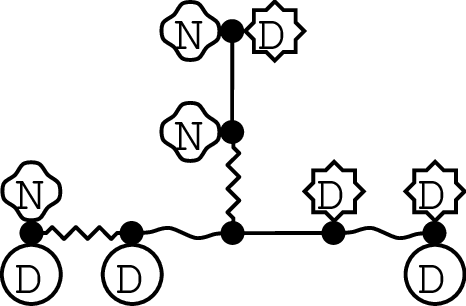} &  & \includegraphics[scale=0.3]{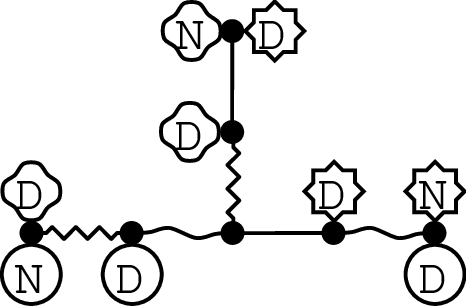} &  & \includegraphics[scale=0.3]{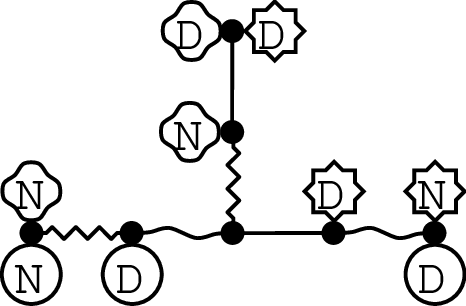}\tabularnewline
\end{tabular}
\par\end{centering}

\caption{Equivalence classes of transplantable pairs with $3$ edge colors
and $7$ vertices per graph, where second numbers refer to dual pairs.
Using the methods of Section~\ref{sec:GeneratingTools}, each class
can be obtained from one of the classes~$1$ and $3$, which in turn
arise from $S_{4}$-subgroups of $PSL(3,2)$ and their trivial and
sign representations, respectively~\cite{BuserConwayDoyleSemmler1994,ParzanchevskiBand2010}.
The corresponding spaces of matrices satisfying~(\ref{eq:TransplantationCondition})
are two- and one-dimensional, respectively. The graphs of each class
differ only by a swap of the edge colors \emph{straight }and \emph{zig-zag}
and swaps of their loop signs. The classes $3$ to $12$ represent
$10$ versions of broken Gordon-Webb-Wolpert drums~\cite{GordonWebbWolpert1992}.
\label{fig:Appendix-7-vertices-per-graph}}
\end{figure}

\newpage{}

\bibliographystyle{amsalpha}
\bibliography{Inaudible_Properties}

\noindent \texttt{\small{}}%
\begin{minipage}[t][5mm]{1\columnwidth}%
\noindent \noun{\small{\hfill{}Dartmouth College, Hanover, New Hampshire,
USA}}{\small \par}

\noindent \noun{\small{\hfill{}}}\emph{\small{E-mail address:}}{\small{
}}\texttt{\small{peter.herbrich@dartmouth.edu}}%
\end{minipage}
\end{document}